\documentclass[11pt,reqno]{amsart}
\usepackage[T1]{fontenc}
\usepackage{lmodern}
\usepackage{amsmath,amssymb,mathtools,mathrsfs}
\usepackage[a4paper,margin=23mm]{geometry}
\usepackage[colorlinks=true,linkcolor=blue,citecolor=blue,urlcolor=blue]{hyperref}
\usepackage{cite}

\usepackage{graphicx,xcolor,tikz}
\usepackage{microtype}

\numberwithin{equation}{section}
\newtheorem{theorem}{Theorem}[section]
\newtheorem{proposition}[theorem]{Proposition}
\newtheorem{lemma}[theorem]{Lemma}

\theoremstyle{definition}
\newtheorem{definition}[theorem]{Definition}
\theoremstyle{remark}
\newtheorem{remark}[theorem]{Remark}

\DeclareMathOperator{\Div}{div}

\newcommand{\bu}{\mathbf u}
\newcommand{\bv}{\mathbf v}
\newcommand{\bw}{\mathbf w}
\newcommand{\be}{\mathbf e}

\newcommand{\bphi}{\boldsymbol\varphi}
\newcommand{\bsigma}{\boldsymbol\sigma}
\newcommand{\cE}{\mathcal E}
\newcommand{\cS}{\mathcal S}
\newcommand{\cD}{\mathcal D}
\newcommand{\cB}{\mathcal B}
\newcommand{\eps}{\varepsilon}
\newcommand{\dx}{\,\mathrm dx}
\newcommand{\dz}{\,\mathrm dz}
\newcommand{\dy}{\,\mathrm dy}
\newcommand{\dt}{\,\mathrm dt}
\newcommand{\dA}{\,\mathrm dx\,\mathrm dy}
\newcommand{\norm}[1]{\lVert #1\rVert}

\allowdisplaybreaks[4]

\title[Beam-fluid interaction]{Global strong solutions without contact for a viscous fluid interacting with a purely elastic beam}

\author{Yadong Liu}

\address{
	School of Mathematical Sciences, and Ministry of Education Key Laboratory of NSLSCS, and Key Laboratory of Jiangsu Provincial Universities of FDMTA, Nanjing Normal University, Nanjing 210023, P. R. China
}
\email{ydliu@njnu.edu.cn}
\date{\today}



\begin{document}
	\begin{abstract}		
		We consider a two-dimensional incompressible viscous fluid confined in a periodic channel between a fixed flat wall and a purely elastic periodic beam. For compatible initial fluid velocity in $H^1$, beam height in $H^3$, and beam velocity in $H^1$, with strictly positive initial height, we establish global existence and uniqueness of strong solutions. In particular, there is no contact of the beam and bottom, which is built on a uniform distance bound derived by the mechanical energy and fluid dissipation through a height-dependent degenerate elliptic regularization of the classical distance test. Combining tangential differentiation, a fractional estimate for the height, and a negative-order tangential estimate for the accelerations, we establish a global regularity estimate. Moreover, quadratic time-weighted estimates that quantify the higher regularity at positive times are justified. As a consequence, the work gives an affirmative answer to the problem proposed by Grandmont and Hillairet (Arch. Ration. Mech. Anal. 220(3): 1283--1333, 2016) of whether fluid dissipation alone is sufficient to derive the distance estimate.
	\end{abstract}
	\maketitle
	
	\section{Introduction}
	
	We study the interaction of a two-dimensional incompressible viscous fluid with a purely elastic beam forming the upper boundary of a periodic channel. The beam deforms under the fluid stress, while its motion determines both the fluid domain and the velocity at the interface. Such coupling between fluid motion and structural deformation arises in blood flow through compliant vessels, flow-induced vibration in engineering, transport through flexible channels, and biological growth \cite{BGN2014,KKLTTW2018}. Its mathematical analysis combines the dissipative dynamics of the fluid with the oscillatory response of the structure on an evolving domain.
	
	The existence theory for fluid--structure interaction has been developed for a broad range of models. Weak solutions have been constructed for fluid--plate systems~\cite{ChambolleDesjardinsEstebanGrandmont2005,Grandmont2008,CanicMuhaTawri2026}, fluid--shell systems~\cite{LengelerRuzicka2014,MuhaCanic2013,MuhaSchwarzacher2022}, and fluids coupled to nonlinear viscoelastic bulk solids~\cite{BenesovaKampschulteSchwarzacher2024}. Compressible counterparts include~\cite{Boulakia2004,BreitSchwarzacher2018,LiuMitraNecasova2026}, while inviscid fluid-structure interaction were studied in~\cite{AlazardKukavicaTuffaha2025}. These theories admit large initial data, but their continuation generally requires that the moving geometry remain nondegenerate: contact with a fixed boundary or self-contact of the structure remains an obstruction even for weak solutions. For the fluid--beam problem, Casanova, Grandmont, and Hillairet~\cite{CasanovaGrandmontHillairet2021} constructed global weak solutions in a formulation that accommodates possible contact, without asserting a positive distance from the wall.
	
	Strong solutions have been obtained locally in time for beams and plates~\cite{BeiraoDaVeiga2004,Lequeurre2011,GrandmontHillairetLequeurre2019,KukavicaTuffaha2024}, shells~\cite{ChengShkoller2007,ChengShkoller2010}, and bulk elastic solids~\cite{CoutandShkoller2005,CoutandShkoller2006,KukavicaTuffaha2012,AbelsLiu2023,RaymondVanninathan2014}. Global results for sufficiently small data, such as~\cite{Lequeurre2011,IgnatovaKukavicaLasieckaTuffaha2017,Remond-TiedrezTice2019}, preserve a nondegenerate configuration through smallness of the deformation. For large data, local and maximal-in-time theories~\cite{MaityRaymondRoy2020,Breit2024,GrandmontSabbagh2025}, including recent results for perfectly elastic shells~\cite{BreitMensahSchwarzacherSu2026}, leave open the possibility of geometric degeneration.
	
	A fundamental question is therefore whether contact can occur in finite time. This question has been studied extensively for rigid bodies immersed in viscous incompressible fluids. Under no-slip boundary conditions, the resistance of the fluid squeezed between approaching surfaces can prevent collision. Hillairet~\cite{Hillairet2007} established such a no-collision result in two dimensions under suitable assumptions on the shapes of the body and the cavity. The outcome depends on both the boundary regularity and the coupling conditions. G\'erard-Varet and Hillairet~\cite{Gerard-VaretHillairet2010} showed that bodies with sufficiently low boundary regularity can collide even under no-slip coupling. For a sphere falling towards a flat wall, G\'erard-Varet, Hillairet, and Wang~\cite{GerardVaretHillairetWang2015} proved finite-time contact when Navier slip is imposed on both approaching surfaces. For an elastic structure, the approaching geometry is itself an unknown, and a no-contact argument must control its deformation as the fluid gap narrows.
	
	For the beam-fluid system, in the seminal paper ~\cite{GrandmontHillairet2016}, Grandmont and Hillairet established the first global strong no-contact result for arbitrary compatible initial data in two dimensions. Their distance estimate, combined with propagation of regularity, yields global strong solutions with initial fluid velocity in $H^1$, beam height in $H^3$, and beam velocity in $H^1$. Note that structural damping is essential to their proof, which controls the beam velocity in $L^2(0,T;H^1)$ and promotes the distance estimate. They commented in \cite[Remark~4]{GrandmontHillairet2016}:
	\begin{quote}
		\emph{To prove our distance estimate we need to control $\partial_th$ in $L^2(0,T;H_\sharp^1(0,L))$. This is the reason why we assume $\gamma>0$. One may wonder whether the fluid dissipation would be sufficient.}
	\end{quote}
	Subsequent work has developed both regularity and contact results for elastic structures. Without structural damping, Schwarzacher and Su~\cite{SchwarzacherSu2025} proved strong existence and uniqueness up to a possible contact for initial data in $H^2\times H^4\times H^2$. Their coupled acceleration estimate propagates regularity for arbitrary times as long as the beam height remains positive. Breit and Roy~\cite{BreitRoy2025} obtained a conditional no-contact criterion for compressible fluids interacting with elastic plates, with an application to the incompressible perfectly elastic beam.
	In three dimensions, Bukal, Kukavica, Li, and Muha~\cite{BKLM2026} established a conditional no-contact criterion under additional regularity assumptions on weak solutions, with or without structural damping. Their separate global existence theorem in~\cite{BKLM2026JNS} constructs weak solutions that remain uniformly separated from the bottom when the initial kinetic and bending energy satisfies a quantitative bound relative to the square of the initial mean plate height, also allowing a perfectly elastic plate. This restriction permits large absolute energy but does not cover arbitrary data at a fixed mean height.
	Recently, for compressible fluids, Trifunovi\'c~\cite{Trifunovic2026} constructed global weak solutions allowing contact and proved detachment driven by pressure and suitable forces on the plate. Tawri and Ward~\cite[Theorem~3]{TawriWard2026} proved finite-time contact for a channel with a damped elastic beam and Navier-slip conditions at the moving interface and bottom wall, under their prescribed outlet flux and pressure data.
	
	In this contribution, we derive a \emph{quantitative distance estimate} directly from the coupled system at the level of the mechanical energy with only the fluid dissipation, cf.~Theorem \ref{thm:separation} below. Combining with new tangential estimates, we complete the construction of global solutions with initial regularity class $H^1\times H^3\times H^1$, cf.~Theorem \ref{thm:global}. This, therefore, answers the question of Grandmont and Hillairet~\cite[Remark~4]{GrandmontHillairet2016} that \emph{the fluid dissipation does suffice for the separation}. We also derive quadratic time-weighted bounds for the higher norms, which quantify the regularity at positive times, cf.~\eqref{eq:weighted-estimate}, and lead to a spatially Lipschitz estimate for fluid velocity, cf.~\eqref{eq:velocity-lipschitz}.
	
	\subsection{Problem setting}
	Let $L>0$, identify the endpoints of $(0,L)$, and denote the beam height by $h(t,x)$. All unknowns are $L$-periodic in the horizontal variable $x$. The subscript $\sharp$ indicates periodicity only, and zero-mean conditions are imposed separately. Norms of periodic functions are taken over one period. For fluid function spaces, periodicity in $x$ is understood and the subscript $\sharp$ is omitted. The fluid occupies
	\begin{align}
		\Omega_h(t)=\{(x,y):x\in(0,L),\ 0<y<h(t,x)\}.\label{eq:domain}
	\end{align}
	The beam moves only vertically, with velocity $\partial_t h$. For a depiction of the geometry, see Figure \ref{fig:geometry} below. The fluid velocity $\bu=(u_1,u_2)$ and pressure $p$ satisfy
	\begin{subequations}\label{eq:model}
		\begin{align}
			\rho_f(\partial_t\bu+\bu\cdot\nabla\bu)-\Div\bsigma&=0,
			\qquad\Div\bu=0\quad\text{in }\Omega_h(t),\label{eq:fluid}\\
			\bu(t,x,0)&=0,\qquad \bu(t,x,h(t,x))= \partial_t h (t,x)\be_2,\label{eq:boundary}\\
			\rho_s\partial_t^2 h-\beta \partial_x^2 h +\alpha \partial_x^4 h 
			&=-\be_2\cdot\bsigma(t,x,h(t,x))\mathbf N_h.\label{eq:beam}
		\end{align}
	\end{subequations}
	Here $ \bsigma=2\mu D\bu-p\mathbf I $, $D\bu=\frac12(\nabla\bu+(\nabla\bu)^\top)$ and $\mathbf N_h=(- \partial_x h,1)$. All unknowns are periodic in $x$. 
	Here $\rho_f,\rho_s > 0$ are the fluid and beam mass densities, $\mu>0$ is the fluid viscosity, $\alpha>0$ is the bending coefficient, and $\beta\geq 0$ is the tension coefficient. The vector $\mathbf N_h$ is the outward normal to the fluid multiplied by the surface Jacobian. Thus the right-hand side of \eqref{eq:beam} is the vertical fluid force per unit reference length.
	
	\begin{figure}[ht!]
		\begin{tikzpicture}[line width=0.75pt,scale=1.5]
			
			\definecolor{aqua}{rgb}{0.0, 1.0, 1.0}
			
			\begin{scope}[declare function={
					f(\x)=2+.7*sin(deg(\x*pi/2))
					+.3*sin(deg(\x*pi))
					+.15*(cos(deg(\x*pi))-cos(deg(\x*pi/2)));
				}]
				
				\fill[aqua!20] plot[domain = 0:4,samples=100](\x, {f(\x)}) -- (4,0) -- (0,0) -- (0,2);
				
				\draw[-stealth] (-.5,0) -- (4.5,0) node[right] {};
				\draw[-stealth] (0,-.25) -- (0,3.15);
				
				\draw[blue,domain = 0:4,samples=100] plot(\x, {f(\x)});
				
				\draw[dashed] (0,2) 
				-- (4,2);
				
				\node[below left] at (0,0) {$ 0 $};
				
				\node[below] at (4,0) {$ L $};
				
				\node[] at (3.5,0.5) {$ \Omega_h(t) $};
			\end{scope}
		\end{tikzpicture}
		\caption{Sketch of the geometry.}
		\label{fig:geometry}
	\end{figure}
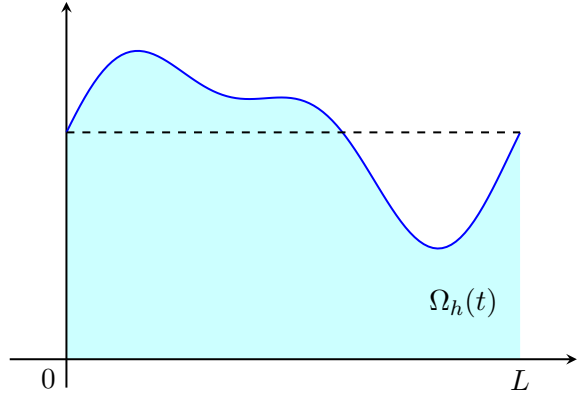
	
	Incompressibility and \eqref{eq:boundary} preserve the mean height
	\begin{align}
		H=\frac1L\int_0^L h(t,x)\dx=\frac1L\int_0^L h_0(x)\dx.
	\end{align}
	The spatially constant part of $p$ is determined by the condition that the right-hand side of \eqref{eq:beam} has zero mean. In particular, we do not independently impose zero spatial mean on the total pressure. The mechanical energy and fluid dissipation are
	\begin{align}
		\cE(t)&=\frac{\rho_f}{2}\norm{\bu(t)}_{L^2(\Omega_h(t))}^2
		+\frac{\rho_s}{2}\norm{\partial_t h (t)}_{L_\sharp^2(0,L)}^2
		+\frac\alpha2\norm{\partial_x^2 h (t)}_{L_\sharp^2(0,L)}^2
		+\frac\beta2\norm{\partial_x h (t)}_{L_\sharp^2(0,L)}^2,\\
		\cD(t)&=\norm{\nabla\bu(t)}_{L^2(\Omega_h(t))}^2.
	\end{align}
	Strong solutions satisfy
	\begin{align}
		\cE(t)+\mu\int_0^t\cD(s)\,\mathrm ds=\cE(0).\label{eq:energy-identity}
	\end{align}
	The equality $2\int_{\Omega_h}|D\bu|^2=\int_{\Omega_h}|\nabla\bu|^2$ used here follows from the particular no-slip traces in \eqref{eq:boundary}.
	
	\subsection{Main results}
	
	We use the standard notation $L^p$ and $W^{k,p}$ for Lebesgue and Sobolev spaces, and write $H^s=W^{s,2}$ for $s\geq0$, with noninteger orders understood in the fractional Sobolev sense. The same notation is used for scalar-, vector-, and matrix-valued functions. For integers $k\geq1$, $H_\sharp^k(0,L)$ consists of the functions in $H^k(0,L)$ whose derivatives of orders $0,\ldots,k-1$ have matching endpoint traces. Fractional periodic Sobolev spaces are defined by the inhomogeneous Fourier weights with frequencies in $(2\pi/L)\mathbb Z$. For $s>0$, we set $H_\sharp^{-s}(0,L)=(H_\sharp^s(0,L))'$.
	For a Banach space $X$, the spaces $L^q(0,T;X)$ and $W^{k,q}(0,T;X)$ have their usual Bochner and Bochner--Sobolev meanings, while $C([0,T];X)$ denotes continuity in the norm of $X$.
	
	For a positive height, define
	\begin{align*}
		\Omega_1=(0,L)\times(0,1),\qquad
		\chi_h(t,x,z)=(x,h(t,x)z).
	\end{align*}
	Thus $\chi_h(t,\cdot)$ maps $\Omega_1$ onto $\Omega_h(t)$. For a function $f$ on the moving domain, we write $\widehat f(t,x,z)=f(t,x,h(t,x)z)$. Measurability in time is understood after this pullback. For $1\leq q<\infty$ and $s\geq0$, we use the physical-domain norm
	\begin{align*}
		\norm{f}_{L^q(0,T;H^s(\Omega_h))}
		=
		\left(
		\int_0^T
		\norm{f(t)}_{H^s(\Omega_h(t))}^q
		\dt
		\right)^{1/q},
	\end{align*}
	with the essential-supremum convention when $q=\infty$. The analogous convention is used for $L^p$ and $W^{k,p}$ in place of $H^s$. We interpret $f\in C([0,T];H^1(\Omega_h))$ as $\widehat f\in C([0,T];H^1(\Omega_1))$.
	
	\begin{definition}\label{def:strong}
		A strong solution to \eqref{eq:model} on $[0,T]$ is a triple $(h,\bu,p)$ with positive continuous height and
		\begin{gather}
			\bu\in C([0,T];H^1(\Omega_h))\cap L^2(0,T;H^2(\Omega_h)),\ 
			\partial_t\bu\in L^2(0,T;L^2(\Omega_h)),\\
			p\in L^2(0,T;H^1(\Omega_h)),\nonumber\\
			h\in C([0,T];H_\sharp^3(0,L))\cap L^2(0,T;H_\sharp^{7/2}(0,L)),\\
			\partial_th\in C([0,T];H_\sharp^1(0,L))\cap L^2(0,T;H_\sharp^{3/2}(0,L)),\nonumber\\
			\partial_t^2h\in L^2(0,T;H_\sharp^{-1/2}(0,L)).\nonumber
		\end{gather}
		The fluid equations hold in $L^2(0,T;L^2(\Omega_h))$, the beam equation holds in $L^2(0,T;H_\sharp^{-1/2}(0,L))$, and the boundary conditions hold in the trace sense. The prescribed initial values are attained in the indicated continuous spaces. The pressure constant is fixed by the mean of the beam equation.
	\end{definition}
	
	The initial height and beam velocity are denoted by $h_0=h(0,\cdot)$ and $h_1=\partial_th(0,\cdot)$, respectively. The data satisfy
	\begin{subequations}\label{eq:data}
		\begin{gather}
			h_0\in H_\sharp^3(0,L),\quad h_1\in H_\sharp^1(0,L),\quad
			\bu_0\in H^1(\Omega_{h_0}),\quad m_0:=\min_{x\in[0,L]}h_0(x)>0,\label{eq:data-regularity}\\
			\int_0^Lh_1\dx=0,\qquad \Div\bu_0=0\quad\text{in }\Omega_{h_0},\\
			\bu_0(x,0)=0,\qquad \bu_0(x,h_0(x))=h_1(x)\be_2.
		\end{gather}
	\end{subequations}
	We first establish a quantitative lower bound for the gap on every finite interval of existence.
	\begin{theorem}[Quantitative distance estimate]
		\label{thm:separation}
		Let $(h,\bu,p)$ be a solution in Definition~\ref{def:strong} on $[0,T_*)$. Set
		\begin{align}
			\cS(t)=\int_0^L\frac{\dx}{h(t,x)}.
		\end{align}
		There exist $C_1,C_2>0$, depending only on $\cE(0),H,L,\rho_f,\rho_s,\mu,\alpha,\beta$, such that
		\begin{align}
			\cS(t)\leq \mathcal R(t):=C_1(1+\cS(0))e^{C_2t},
			\qquad 0\leq t<T_*, \label{eq:entropy-bound}
		\end{align}
		and
		\begin{align}
			\min_{x\in[0,L]}h(t,x)\geq
			\left(\frac{2\mathcal R(t)}L+\frac{2\cE(0)}{3\alpha}\mathcal R(t)^3\right)^{-1},
			\qquad 0\leq t<T_*.\label{eq:gap-bound}
		\end{align}
		These constants are independent of any assumed positive lower bound for the evolving height.
	\end{theorem}
	
	By the separation estimate in Theorem~\ref{thm:separation}, together with the global regularity estimate, we obtain the following result.
	\begin{theorem}[Global strong solution]\label{thm:global}
		Let the initial data $(\bu_0,h_0,h_1)$ satisfy \eqref{eq:data}. Then system~\eqref{eq:model} has a unique global solution in the class of Definition~\ref{def:strong}. It satisfies the energy identity~\eqref{eq:energy-identity} and the distance estimates~\eqref{eq:entropy-bound}--\eqref{eq:gap-bound}. For every finite $T>0$,
		\begin{align}
			&\sup_{0\leq t\leq T}\left(
			\norm{\bu(t)}_{H^1(\Omega_h(t))}^2+
			\norm{h(t)}_{H_\sharp^3(0,L)}^2+
			\norm{\partial_th(t)}_{H_\sharp^1(0,L)}^2\right)\nonumber\\
			&\quad+\int_0^T\left(
			\norm{\partial_t\bu}_{L^2(\Omega_h)}^2+
			\norm{\bu}_{H^2(\Omega_h)}^2+
			\norm{p}_{H^1(\Omega_h)}^2\right)\dt\nonumber\\
			&\quad+\int_0^T\left(
			\norm{h}_{H_\sharp^{7/2}(0,L)}^2+
			\norm{\partial_th}_{H_\sharp^{3/2}(0,L)}^2+
			\norm{\partial_t^2h}_{H_\sharp^{-1/2}(0,L)}^2\right)\dt\leq C_T.\label{eq:lower-estimate}
		\end{align}
		Here $C_T$ depends only on $T,L,\rho_f,\rho_s,\mu,\alpha,\beta$, $m_0^{-1}$, and the three initial norms in \eqref{eq:data-regularity}.
	\end{theorem}
	Moreover, we establish further regularity property for the global strong solutions, which might be of interest for its own. 
	\begin{theorem}[Partial parabolic smoothing]\label{thm:smoothing}
		Let $(h,\bu,p)$ be the global strong solution furnished by Theorem~\ref{thm:global}. For every finite $T>0$, the solution satisfies
		\begin{align}
			&\mathop{\mathrm{ess\,sup}}_{0<t\leq T}t^2\left(
			\norm{\partial_t\bu(t)}_{L^2(\Omega_h(t))}^2
			+\norm{\partial_t^2h(t)}_{L_\sharp^2(0,L)}^2
			+\norm{\partial_th(t)}_{H_\sharp^2(0,L)}^2\right.\nonumber\\
			&\hspace{41mm}\left.
			+\norm{\bu(t)}_{H^2(\Omega_h(t))}^2
			+\norm{p(t)}_{H^1(\Omega_h(t))}^2
			+\norm{h(t)}_{H_\sharp^4(0,L)}^2\right)\nonumber\\
			&\qquad
			+\int_0^Tt^2
			\norm{\nabla\partial_t\bu(t)}_{L^2(\Omega_h(t))}^2
			\dt
			\leq C_T.
			\label{eq:weighted-estimate}
		\end{align}
		Moreover, for every $1\leq q<2$,
		\begin{align}
			\norm{\bu}_{L^q(0,T;W^{1,\infty}(\Omega_h))}
			\leq C_{T,q}.
			\label{eq:velocity-lipschitz}
		\end{align}
		Here $C_T$ depends only on $T,L,\rho_f,\rho_s,\mu,\alpha,\beta$, $m_0^{-1}$, and the three initial norms in \eqref{eq:data-regularity}. The constant $C_{T,q}$ additionally depends on $q$.
	\end{theorem}
	Several comments are in order.
	
	\begin{remark}
		Theorem \ref{thm:separation} shows that the fluid dissipation solely is in fact strong enough in such beam-fluid, to prevent possible contact. Recently, Alazard, Kukavica, and Tuffaha \cite{AlazardKukavicaTuffaha2025} showed that in two dimensions, the inviscid fluid-beam system has a global weak solution, by the Craig--Sulem--Zakharov formulation of the problem. One may also ask if the distant estimate can still hold without fluid viscosity. 
	\end{remark}
	
	\begin{remark}
		For Theorem \ref{thm:global}, we adopt the initial regularity of the damped-beam theorem in \cite{GrandmontHillairet2016} without adding structural viscosity. The tension by $ \beta \geq 0 $ is unnecessary because bending controls all lower structural norms on the space of periodic functions with zero mean. Section~\ref{sec:global-proof} explains the corresponding zero-tension modification of the higher continuation theorem, which was stated for positive tension in \cite{SchwarzacherSu2025}.
	\end{remark}
	
	\begin{remark}
		Here the additional spatial Lipschitz estimate \eqref{eq:velocity-lipschitz} is derived as a byproduct, following the argument in \cite{DanchinMucha2019}.  For further applications, it may contribute in the stability issue with low regularity.
	\end{remark}
	
	\subsection{Obstruction and strategy}
	The main obstacle to derive a distance estimate is due to a lack of regularity. Since higher-order estimate would require nondegeneracy of the domain ($ \min h > 0 $), one can not exploit it without a uniform distance estimate.
	The standard energy controls $h$ in $H_\sharp^2(0,L)$ and $\partial_th$ in $L_\sharp^2(0,L)$, but gives neither separation from the bottom nor a full spatial derivative of the beam velocity. This loss appears in the classical distance test as in \cite{GrandmontHillairet2016}. Indeed, for the stream function
	\begin{align*}
		\psi(t,x,y)=F(t,x)\chi\left(\frac{y}{h(t,x)}\right),\qquad \chi(z)=3z^2-2z^3,
	\end{align*}
	the solenoidal field $\bw=(-\partial_y\psi,\partial_x\psi)$ has upper trace $\partial_xF\,\be_2$. The choice $F=\partial_xh$ produces the reciprocal-height derivative through the viscous term as in \cite{GrandmontHillairet2016}, but the structural inertia gives
	\begin{align*}
		\int_0^L\partial_t^2h\,\partial_x^2h\dx
		=\partial_t\int_0^L\partial_th\,\partial_x^2h\dx
		+\norm{\partial_t\partial_xh}_{L_\sharp^2(0,L)}^2.
	\end{align*}
	The last term has an unfavorable sign for the distance estimate and cannot be bounded by the mechanical energy and fluid dissipation.
	
	Instead, note that the fundamental theorem of calculus in each vertical fluid column yields
	\begin{align*}
		\int_0^L\frac{|\partial_t h|^2}{h}\,\mathrm dx = \int_0^L\frac{|u_2(t,x,h(t,x))|^2}{h}\,\mathrm dx
		\leq \int_0^L \int_0^{h(t,x)} |\partial_y u_2|^2 \dy \dx
		\leq \cD.
	\end{align*}
	Related weighted column estimates appear in~\cite[Section~4.2]{GrandmontHillairet2016} and~\cite[Lemma~A.3]{Trifunovic2026}. See also \cite[Theorem 5.1]{LeoniTice2019} for the trace of functions in strip-like domains. After integration by parts in $x$, the inertial remainder for a general $F$ therefore satisfies
	\begin{align*}
		\left|\int_0^L\partial_x\partial_th\,\partial_tF\dx\right|
		\leq\sqrt{\cD}\left(\int_0^Lh|\partial_x\partial_tF|^2\dx\right)^{1/2}.
	\end{align*}
	This allows a certain degree of freedom for us to close the argument. By the preceding inequality, a higher regularity of $F$, both in space and time, is required, which is not available a priori if $ F = \partial_x h $. Motivated by the observation, we therefore seek a regularization of $F = \partial_xh$ for which the weighted norm on the right-hand side can be controlled by the fluid dissipation, with a dependence on the minimum height compatible with the distance estimate. Such a regularization must also preserve the reciprocal-height derivative generated by the fluid viscosity. To make this it explicit, we introduce the horizontal flux
	\begin{align*}
		Q(t,x)=\int_0^{h(t,x)}u_1(t,x,y)\dy,
	\end{align*}
	which satisfies $\partial_th=-\partial_xQ$, and hence
	\begin{align*}
		\partial_t\cS(t)=2\int_0^L\frac{Q\partial_xh}{h^3}\dx.
	\end{align*}
	See also \cite[(79) \& (82)]{GrandmontHillairet2016} for similar calculations. The choice $F=\partial_xh$ recovers the above derivative exactly by the stream function test procedure that generates $\partial_t \cS$. Thus, a regularized $ F $ must retain this identity as well, up to a controllable error, cf.~Lemma \ref{lem:stress-identity} below. Because of the degeneracy as $ h \to 0 $, we use a ``degenerate elliptic regularization'' of $ \partial_x h $, by solving
	\begin{align*}
		F-\ell^2\partial_x^2F=\partial_xh,\qquad \ell=\varepsilon h^{2/3},
	\end{align*}
	where $\varepsilon>0$ is fixed by the initial energy. 
	The regularization length $\ell=\varepsilon h^{2/3}$ is chosen to balance the spatial variation of the coefficients with the singular weight in the inertial estimate, i.e.,
	\begin{align*}
		|\partial_x \ell| = \frac{2}{3} \eps h^{-1/3} |\partial_x h| \leq C \eps, 
	\end{align*}
	where $ |\partial_x h| \leq C h^{1/3} $ (cf.~\eqref{eq:slope-cubic}). With this construction, we are able to particularly justify (cf. Lemma \ref{lem:time} below)
	\begin{align*}
		\int_0^Lh|\partial_x\partial_tF|^2\dx
		\leq C(\min h)^{-2/3}\cD(t),
	\end{align*}
	The resulting inertial loss is $(\min h)^{-1/3}$, which one can control by $C(1+\cS)$, cf.~Lemma~\ref{lem:geometry}. This replaces the uncontrolled spatial derivative of the beam velocity by a weighted estimate compatible with a linear Gr\"onwall argument, cf. Section \ref{sec:separation-proof}. 
	We believe that the \emph{degenerate regularized multiplier} shall be adopted to other fluid-structure interaction problems to address the no-contact/contact issues, as long as its time derivative is compatible with the weighted trace controlled by fluid dissipation, and its spatial regularization preserve the distance identity up to an error controlled by elasticity.
	
	The next issue is to obtain a global $H^1$-level estimate. 
	The global regularity estimate of Grandmont and Hillairet~\cite{GrandmontHillairet2016} depends strongly on the structure damping ($\gamma > 0$), while the regularity theory of Schwarzacher and Su~\cite{SchwarzacherSu2025} is built on a coupled \textit{acceleration estimate}. Testing the time-differentiated system by the accelerations preserves the hyperbolic structure of the beam and yields the estimate for the spatial gradient of the fluid time derivative, from which continuation follows whenever the height stays positive. That estimate requires initial velocity in $H^2$, initial height in $H^4$ and initial beam velocity in $H^2$, and the constants that carry the local theory depend on the inverse gap. 
	The present contribution reaches the lower regularity $H^1\times H^3\times H^1$ setting as in Grandmont and Hillairet \cite{GrandmontHillairet2016} and proves that contact never occurs. At this level the acceleration is not available as a test field at time zero, and the corresponding constant becomes unbounded in the limit. Indeed, the Eulerian temporal derivative has the boundary trace
	\begin{align*}
		\partial_t\bu\big|_{y=h}=\partial_t^2h\,\be_2-\partial_th\,\partial_y\bu\big|_{y=h},
	\end{align*}
	so it does not directly give a coupled acceleration test. Testing the beam equation by $\partial_t^2h$ also produces an elastic cross term whose time derivative leaves a negative squared norm of $\partial_t\partial_x^2h$. Thus, acceleration argument, in the form used in \cite{SchwarzacherSu2025}, cannot be run on our data. Indeed, as Schwarzacher and Su mentioned in \cite{SchwarzacherSu2025}, second-order time derivative of a solution is not a good test function for hyperbolic equations.
	
	We circumvent it by tangential estimates and a negative half order accelerate estimate, which then control the final fluid acceleration and Stokes regularity via a solenoidal lifting. 
	On the fixed strip $\Omega_1$, the Piola velocity $\bv=B_h^\top(\bu\circ\chi_h)$ is solenoidal and has upper trace $\partial_th\,\be_2$, cf.~\eqref{eq:piola-velocity} below. Essentially tangential derivatives preserve both constraints. The coupled test $(-\partial_x^2\bv,-\partial_t\partial_x^2h)$ controls the horizontal derivatives of the fluid and the lower structural energy, while the fluid dissipation gives $\partial_th\in L^2(0,T;H_\sharp^{3/2}(0,L))$ through the upper trace. Combining with a subsequent solenoidal lifting test with upper trace $(1-\partial_x^2)^{3/2}(h-H)\be_2$, we are able to derive $h\in L^2(0,T;H_\sharp^{7/2}(0,L))$, cf.~Lemma \ref{lem:lower-height} below.
	Then formally by comparison in \eqref{eq:beam}, the beam acceleration is expected to stay in a negative half-order Sobolev space. This is further established with an admissible pair $((N^2-\partial_x^2)^{-1/2}\partial_t\bv,(N^2-\partial_x^2)^{-1/2}\partial_t^2h)$, where negative-order tangential multiplier is employed. The multiplier with $N$ that is chosen to be sufficiently large here filters the singular part of the viscous commutator, whose coefficient in the fluid $H^2$ norm is $N^{-2}$. The resulting estimate for $ \partial_t^2 h $ lies in $ L^2(0,T;H_\sharp^{-1/2}(0,L)) $, supplying the remaining boundary norm in the fluid estimate, \textit{without} requiring any initial acceleration, cf.~Lemma \ref{lem:lower-acceleration}.
	Note that subtracting a stationary solenoidal Stokes lift of the boundary velocity $\partial_th\,\be_2$ reduces the fluid equation to homogeneous boundary data. The time derivative of this lift is controlled in $L^2(0,T;L^2(\Omega_1))$ by $\norm{\partial_t^2h}_{L^2(0,T;H_\sharp^{-1/2}(0,L))}$. This gives a nonstationary fluid estimate on the entire chosen interval, which combined the tangential estimates above closes the argument by Gr\"onwall's inequality.
	
	Additionally, we would like to explore the regularity property for the global strong solutions. In particular, we aim to figure out what does the fluid dissipation provide in such a parabolic-hyperbolic coupled system. This motivates us to derive the \textit{time-weighted estimate}, cf. \eqref{eq:weighted-estimate} above. Following the scheme of Danchin and Mucha~\cite{DanchinMucha2019} (see also \cite{Li2017}), we carry out the time-weighted estimates on our fluid-structure interaction problem under the energy estimate, uniform separation, and $ H^1 $-estimate. The quadratic weight is then obtained by averaging the higher propagation estimate over positive starting times. Precisely the propagation estimate presupposes the higher norms at its starting time and therefore cannot be restarted at $t=0$ on our data, so the estimate holds for positive times only and controls the initial layer rather than supplying a new existence mechanism.
	
	Note that the \textit{time-weighted estimate} records the \textit{partial parabolic smoothing} of the coupled fluid-structure system (parabolic--hyperbolic coupling). The smoothing originates in the viscous fluid alone, since no structural damping is present and the beam equation is hyperbolic. The elastic beam gains regularity only through the coupling, which transmits the fluid regularity to the structure across the moving interface. 
	Nevertheless, the estimate is of independent interest in this form, since it quantifies a smoothing mechanism that is available to coupled parabolic--hyperbolic problems at low initial regularity, which is not visible in an unweighted estimate. Thus, we expect similar time-weighted estimates to be applicable to broader situations of fluid--structure interactions.
	
	\subsection{Outline}
	In Section~\ref{sec:separation}, the quantitative separation estimate is justified, including admissibility of the distance test in the lower regularity class. Section~\ref{sec:lower-estimates} establishes the direct global lower-order estimates, which is supplied by delicate tangential estimates for height, velocity and acceleration. In Section~\ref{sec:global-proof}, the global solution is constructed via higher-regularity approximation, and uniqueness is shown as well. Section~\ref{sec:time-weighted-estimates} then derives the higher regularity estimates with quadratic time weights.
	
	\section{Quantitative Separation}\label{sec:separation}
	In this section, we prove Theorem \ref{thm:separation}, which establishes the quantitative distance estimate. Before introducing the main argument, let us briefly justify the energy-dissipation law of \eqref{eq:model}, which is the basis of further estimates. Throughout this section, $C$ denotes a positive constant depending only on $\cE(0),H,L,\rho_f,\rho_s,\mu,\alpha,\beta$; its value may change between occurrences. These constants are independent of the minimum height. 
	\begin{lemma}\label{lem:mechanical-energy}
		Every solution in Definition~\ref{def:strong} conserves the mean height and satisfies \eqref{eq:energy-identity}.
	\end{lemma}
	\begin{proof}
		Integrating incompressibility and using the prescribed traces, we obtain $\int_0^L\partial_th\dx=0$. On a compact positive-gap interval, the two-dimensional product estimate ensures $\bu\cdot\nabla\bu\in L^2(0,T;L^2(\Omega_h))$. Then we use the test pair $(\bu, \partial_t h \be_2)$ for the coupled problem. Reynolds transport contributes $\frac12\int_0^L|\partial_th|^2\partial_th\dx$, which cancels the convective boundary flux. The beam equation can be paired with $\partial_th\in L^\infty(0,T;H_\sharp^1(0,L))$. The elastic product rule follows from $\partial_x^2h\in L^2(0,T;H_\sharp^1(0,L))$ and $\partial_t\partial_x^2h\in L^2(0,T;H_\sharp^{-1}(0,L))$. Fluid traction and beam work cancel because the traces match. We thus obtain the energy identity with dissipation $2\mu\int_{\Omega_h}|D\bu|^2\dA$. The cross-gradient integration by parts proved in Lemma~\ref{lem:stress-identity} identifies this dissipation with $\mu\int_{\Omega_h}|\nabla\bu|^2\dA$. Temporal regularization and the continuous initial traces justify the calculation from zero.
	\end{proof}
	
	\subsection{Geometry and estimates from the fluid dissipation}
	By \eqref{eq:energy-identity}, we know
	\begin{align}
		\norm{\partial_x^2 h (t)}_{L_\sharp^2(0,L)}\leq\sqrt{\frac{2\cE(0)}\alpha},\qquad
		0<h(t,x)\leq H+L^{3/2}\sqrt{\frac{2\cE(0)}\alpha}.\label{eq:energy-geometry}
	\end{align}
	Indeed, since $\partial_x h$ has zero mean, we have $\norm{\partial_x h}_{L_\sharp^\infty(0,L)}\leq\sqrt L\norm{\partial_x^2 h}_{L_\sharp^2(0,L)}$, and the oscillation of $h$ is at most $L\norm{\partial_x h}_{L_\sharp^\infty(0,L)}$. 
	
	We first establish a Glaeser-type estimate for the slope and a quantitative reciprocal-height estimate for the minimum. The former controls the spatial variation of height uniformly as the gap decreases (see also \cite[Proposition 6]{GrandmontHillairet2016} and \cite[Lemma A.1]{BKLM2026}), while the latter converts a bound on the reciprocal-height integral into a positive lower bound for the gap. 
	\begin{lemma}
		\label{lem:geometry}
		Every positive periodic $h\in H_\sharp^2(0,L)$ satisfies
		\begin{align}
			| \partial_x h (x)|^3&\leq\frac{9}{4}\norm{\partial_x^2 h}_{L_\sharp^2(0,L)}^2h(x).\label{eq:slope-cubic}
		\end{align}
		Moreover, it holds
		\begin{align}
			(\min h)^{-1}\leq\frac{2}L \int_0^L \frac1h\dx+\frac{1}{3}\norm{\partial_x^2 h}_{L_\sharp^2(0,L)}^2 \left(\int_0^L \frac1h\dx\right)^3,
			\label{eq:min-entropy0}
		\end{align}
		In particular,
		\begin{align}
			(\min h)^{-1/3}\leq C\left(1+\int_0^L \frac1h\dx\right).
			\label{eq:min-entropy}
		\end{align}
	\end{lemma}
	\begin{proof}
		The one-dimensional Sobolev embedding and positivity of $h$ justify the chain rules and integrations by parts below. We first observe that
		\begin{align*}
			\partial_x\left(\frac{(\partial_x h)^3}{h}\right)
			=3\frac{(\partial_x h)^2}{h}\partial_x^2h
			-\frac{(\partial_x h)^4}{h^2}
			=\frac94|\partial_x^2h|^2
			-\left(\frac{(\partial_x h)^2}{h}
			-\frac32\partial_x^2h\right)^2
			\leq\frac94|\partial_x^2h|^2.
		\end{align*}
		Since $\partial_x h$ has zero mean and is continuous, it vanishes somewhere in each period. Integrating the preceding inequality from such a zero to $x$ bounds $(\partial_x h(x))^3/h(x)$ from above. Integrating from $x$ to the next periodic copy of the same zero bounds it from below. Consequently,
		\begin{align*}
			\frac{|\partial_x h(x)|^3}{h(x)}
			\leq\frac94\int_0^L|\partial_x^2h|^2\dx,
		\end{align*}
		which implies \eqref{eq:slope-cubic}.
		
		For the reciprocal-height estimate, we apply periodic integration by parts and the Cauchy--Schwarz inequality to obtain
		\begin{align*}
			\int_0^L\frac{|\partial_x h|^2}{h^3}\dx
			=\frac12\int_0^L\frac{\partial_x^2h}{h^2}\dx
			& \leq\frac12
			\left(\int_0^L|\partial_x^2h|^2\dx\right)^{1/2}
			\left(\int_0^L\frac1{h^4}\dx\right)^{1/2} \\
			& \leq\frac{\norm{\partial_x^2 h}_{L_\sharp^2(0,L)}}{2} \left(\int_0^L \frac1h\dx\right)^{1/2} (\min h)^{-3/2}.
		\end{align*}
		The periodic $W^{1,1}$ estimate bounds the maximum of $h^{-1}$ by its mean plus one half of the $L^1$ norm of its derivative. Indeed, periodicity implies that its total variation is at least twice its oscillation. Applying this estimate and then the Cauchy--Schwarz and Young inequalities, we obtain
		\begin{align*}
			(\min h)^{-1}
			&\leq\frac 1L \int_0^L \frac1h\dx+\frac12\int_0^L
			\frac{|\partial_x h|}{h^2}\dx\\
			&\leq\frac 1L\int_0^L \frac1h\dx+\frac12
			\left(\int_0^L\frac{|\partial_x h|^2}{h^3}\dx\right)^{1/2}
			\left(\int_0^L\frac1h\dx\right)^{1/2}\\
			&\leq\frac 1L\int_0^L \frac1h\dx
			+ \frac1{2} (\min h)^{-1}
			+ \frac{\norm{\partial_x^2 h}_{L_\sharp^2(0,L)}^2}{6} \left(\int_0^L \frac1h\dx\right)^3,
		\end{align*}
		which shows \eqref{eq:min-entropy0}. Taking cube roots, using $(a+b)^{1/3}\leq a^{1/3}+b^{1/3}$ for $a,b\geq0$, and recalling $\norm{\partial_x^2h}_{L_\sharp^2(0,L)}^2\leq2\cE(0)/\alpha$, we obtain
		\begin{align*}
			(\min h)^{-1/3}
			&\leq\left(\frac2L\right)^{1/3}\left(\int_0^L\frac1h\dx\right)^{1/3}
			+\left(\frac{2\cE(0)}{3\alpha}\right)^{1/3}\int_0^L\frac1h\dx
			\leq C\left(1+\int_0^L\frac1h\dx\right),
		\end{align*}
		which proves \eqref{eq:min-entropy}.
	\end{proof}
	
	For the rest of this section, $h$ satisfies the energy bounds in \eqref{eq:energy-geometry}. Fix the following parameters:
	\begin{align}
		\varepsilon=\frac{3}{32\left(1+\left(\frac{9\cE(0)}{2\alpha}\right)^{1/3}\right)},\qquad
		\ell=\varepsilon h^{2/3}.
	\end{align}
	Then it follows from Lemma~\ref{lem:geometry} that
	\begin{align}
		|\partial_x h|\leq\left(\frac{9\cE(0)}{2\alpha}\right)^{1/3}h^{1/3},\qquad
		\norm{\partial_x\ell}_{L_\sharp^\infty(0,L)}\leq\frac1{16}.\label{eq:length-lipschitz}
	\end{align}
	
	The horizontal flux links the beam motion to the fluid dissipation. More precisely, incompressibility implies $\partial_t h=-\partial_xQ$, so the derivative of the reciprocal-height integral can be expressed in terms of $Q$. The no-slip condition at both ends of each vertical column controls $Q$ with weight $h^{-3}$, which is precisely the weight in the viscous distance identity.
	
	\begin{lemma}
		\label{lem:flux}
		Define the horizontal flux by
		\begin{align*}
			Q(t,x)=\int_0^{h(t,x)}u_1(t,x,y)\dy.
		\end{align*}
		Then
		\begin{align}
			\partial_t h =- \partial_x Q,\qquad
			12\int_0^L\frac{Q^2}{h^3}\dx\leq\cD,\qquad
			\int_0^L\frac{(\partial_t h)^2}{h}\dx\leq\cD.\label{eq:flux-estimates}
		\end{align}
		Moreover,
		\begin{align}
			\norm{\bu(x,\cdot)}_{L^2(0,h(x))}
			\leq h(x)\norm{\partial_y\bu(x,\cdot)}_{L^2(0,h(x))}.\label{eq:column-poincare}
		\end{align}
	\end{lemma}
	
	\begin{proof}
		Differentiating $Q$ and using $u_1(x,h)=0$, we find
		\begin{align*}
			\partial_x Q =\int_0^h \partial_x u_1 \dy=-\int_0^h \partial_y u_2 \dy=- \partial_t h.
		\end{align*}
		The two zero endpoint values of $u_1$ imply
		\begin{align*}
			Q=-\int_0^h\left(y-\frac h2\right) \partial_y u_1 \dy,
			\qquad |Q|^2\leq\frac{h^3}{12}\int_0^h| \partial_y u_1 |^2\dy.
		\end{align*}
		Also $\partial_t h =\int_0^h \partial_y u_2 \dy$, so $(\partial_t h)^2\leq h\int_0^h| \partial_y u_2 |^2\dy$. Integration in $x$ proves \eqref{eq:flux-estimates}. Finally, \eqref{eq:column-poincare} follows from the bottom condition $\bu(x,0)=0$ and the fundamental theorem of calculus.
	\end{proof}
	\begin{remark}
		The powers $h^{-3}$ and $h^{-1}$ in \eqref{eq:flux-estimates} follow from the squared $L^2(0,h)$ norms of $y-h/2$ and $1$, respectively. Thus fluid dissipation controls the flux and the beam velocity with the height weights needed to close the coupled distance identity. Although these estimates are derived while $h>0$, their constants are independent of $\min h$. The weighted trace estimate underlying the last inequality in \eqref{eq:flux-estimates}, together with a weighted column estimate, also appears in \cite[Lemma~A.3]{Trifunovic2026}.
	\end{remark}
	
	\subsection{The regularized multiplier}
	
	For each positive height $h$, let $F$ be the periodic solution of
	\begin{align}
		F-\ell^2\partial_x^2F=\partial_xh, \label{eq:elliptic}
	\end{align}
	where $\ell=\varepsilon h^{2/3}$ is the local smoothing length for $\partial_xh$. The exponent $2/3$ is consistent with both the slope estimate, which bounds $\partial_x\ell$ uniformly, and the height-weighted time-derivative estimates below, which produce a loss of at most $(\min h)^{-1/3}$ in the structural inertial term. Lemma~\ref{lem:geometry} bounds this loss by $C(1+\cS)$.
	After division by $\ell^2$, equation~\eqref{eq:elliptic} reads $- \partial_x^2 F +\ell^{-2}F=\ell^{-2} \partial_x h$. For a fixed positive $h$, the bilinear form
	\begin{align*}
		(F,\varphi)\mapsto\int_0^L \partial_x F \partial_x \varphi \dx+\int_0^L \ell^{-2}F\varphi\dx
	\end{align*}
	is coercive on $H_\sharp^1(0,L)$. The Lax--Milgram theorem and one-dimensional elliptic regularity therefore ensure a unique solution. This construction may depend on the positive minimum of the fixed height, while the following estimates have constants uniform as that minimum decreases. 
	
	\begin{lemma}[Pointwise bounds]
		The solution of \eqref{eq:elliptic} satisfies
		\begin{align}
			|F|\leq Ch^{1/3},\qquad
			| \partial_x F |\leq Ch^{-1/3},\qquad
			| \partial_x^2 F |\leq Ch^{-1},\label{eq:pointwise}
		\end{align}
		where $C > 0$ does not depend on $\min h$.
	\end{lemma}
	\begin{proof}
		Since $h\in H_\sharp^2(0,L)$ is positive, \eqref{eq:elliptic} implies $F\in H_\sharp^3(0,L)\hookrightarrow C^2([0,L])$. We estimate $F$ and $\ell\partial_xF$ simultaneously, using
		\begin{align*}
			|F|+\ell|\partial_xF|
			=\max\bigl\{|F+\ell\partial_xF|,\,
			|F-\ell\partial_xF|\bigr\}.
		\end{align*}
		The combinations on the right satisfy first-order identities involving only $\partial_xh$ and $\partial_x\ell$. Indeed, since
		\begin{align*}
			\ell^2\partial_x^2F=F-\partial_xh,
			\qquad
			\ell\partial_x(h^{-1/3})
			=-\frac12(\partial_x\ell)h^{-1/3},
		\end{align*}
		direct differentiation yields
		\begin{align*}
			\ell\partial_x\left(
			\frac{F\pm\ell\partial_xF}{h^{1/3}}\right)
			&=
			\pm\frac{F\pm\ell\partial_xF}{h^{1/3}}
			-\frac{\partial_x\ell}{2}
			\frac{F\mp\ell\partial_xF}{h^{1/3}}
			\mp\frac{\partial_xh}{h^{1/3}}.
		\end{align*}
		
		Choose the sign for which the absolute maximum of these two periodic functions is larger, and take a point where this maximum is attained. The derivative of the selected function vanishes there, while the absolute value of the other function is no larger. Taking absolute values in the preceding identity and using \eqref{eq:length-lipschitz}, we obtain
		\begin{align*}
			\max_{\pm}
			\left\|\frac{F\pm\ell\partial_xF}{h^{1/3}}\right\|
			_{L_\sharp^\infty(0,L)}
			&\leq \frac1{32}\max_{\pm}
			\left\|\frac{F\pm\ell\partial_xF}{h^{1/3}}\right\|
			_{L_\sharp^\infty(0,L)}
			+\left\|\frac{\partial_xh}{h^{1/3}}\right\|
			_{L_\sharp^\infty(0,L)}.
		\end{align*}
		Absorbing the first term and recalling the slope bound in \eqref{eq:length-lipschitz}, we conclude that, at every point,
		\begin{align*}
			\frac{|F|+\ell|\partial_xF|}{h^{1/3}}
			&\leq \frac{32}{31}
			\norm{\frac{\partial_xh}{h^{1/3}}}
			_{L_\sharp^\infty(0,L)}
			\leq C.
		\end{align*}
		Consequently, since $\ell=\varepsilon h^{2/3}$,
		\begin{align*}
			|F|\leq Ch^{1/3},
			\qquad
			|\partial_xF|\leq C\varepsilon^{-1}h^{-1/3}.
		\end{align*}
		Finally, the equation itself implies
		\begin{align*}
			|\partial_x^2F|
			\leq \ell^{-2}\bigl(|F|+|\partial_xh|\bigr)
			\leq C\varepsilon^{-2}h^{-1}.
		\end{align*}
		This completes the proof.
	\end{proof}
	
	The next identity provides the positive term that will absorb the defect in the viscous distance calculation.
	\begin{lemma}\label{lem:elastic}
		For a positive $h\in H_\sharp^3(0,L)$, it holds
		\begin{align}
			\beta& \int_0^L \partial_x^2 h \partial_x F \dx+\alpha\int_0^L \partial_x^3 h \partial_x^2 F \dx \nonumber\\
			& \geq\frac{15\alpha}{16}
			\bigl(\norm{\partial_x^2 F}_{L_\sharp^2(0,L)}^2+\norm{\ell \partial_x^3 F}_{L_\sharp^2(0,L)}^2\bigr)
			+\beta\bigl(\norm{\partial_x F}_{L_\sharp^2(0,L)}^2+\norm{\ell \partial_x^2 F}_{L_\sharp^2(0,L)}^2\bigr)\geq0.
			\label{eq:elastic-coercivity}
		\end{align}
	\end{lemma}
	
	\begin{proof}
		By \eqref{eq:elliptic}, we have 
		\begin{align*}
			\partial_x^2 h = \partial_x F -\partial_x(\ell^2 \partial_x^2 F),\qquad
			\partial_x^3 h = \partial_x^2 F -\partial_x^2(\ell^2 \partial_x^2 F).
		\end{align*}
		Then $F\in H_\sharp^4(0,L)$ because $h\in H_\sharp^3(0,L)$. Integration by parts yields
		\begin{align*}
			\int_0^L \partial_x^2 h \partial_x F \dx
			&=\norm{\partial_x F}_{L_\sharp^2(0,L)}^2+\norm{\ell \partial_x^2 F}_{L_\sharp^2(0,L)}^2,\\
			\int_0^L \partial_x^3 h \partial_x^2 F \dx
			&=\norm{\partial_x^2 F}_{L_\sharp^2(0,L)}^2+\norm{\ell \partial_x^3 F}_{L_\sharp^2(0,L)}^2
			+\int_0^L \partial_x(\ell^2) \partial_x^2 F \partial_x^3 F.
		\end{align*}
		Moreover, 
		\begin{align*}
			\left|\int_0^L \partial_x(\ell^2) \partial_x^2 F \partial_x^3 F \dx\right|
			\leq\frac18\norm{\partial_x^2 F}_{L_\sharp^2(0,L)}\norm{\ell \partial_x^3 F}_{L_\sharp^2(0,L)}
			\leq\frac1{16}\bigl(\norm{\partial_x^2 F}_{L_\sharp^2(0,L)}^2+\norm{\ell \partial_x^3 F}_{L_\sharp^2(0,L)}^2\bigr).
		\end{align*}
		Adding them together with $\alpha > 0,\beta \geq 0$ proves \eqref{eq:elastic-coercivity}.
	\end{proof}
	
	Now we turn to the temporal regularity. Taking the time derivative of \eqref{eq:elliptic} leads to
	\begin{align}
		\partial_t F -\ell^2 \partial_x^2\partial_t F = \partial_x\partial_t h +\frac43\frac{F-\partial_x h}{h}\partial_t h,\label{eq:time-resolvent}
	\end{align}
	where $\left|\frac43\frac{F-\partial_x h}{h}\right|\leq Ch^{-2/3}$.
	\begin{lemma}
		\label{lem:time}
		For a strong solution, it follows
		\begin{gather}
			\int_0^L h(\partial_t F)^2\dx\leq C\cD,\qquad
			\int_0^L h(\partial_x\partial_t F)^2\dx\leq C (\min h)^{-2/3}\cD,\label{eq:time-dissipation}\\
			\left|\int_0^L \partial_x\partial_t h \partial_t F \dx\right|\leq C(1+\cS)\cD.
		\end{gather}
	\end{lemma}
	
	\begin{proof}
		Testing \eqref{eq:time-resolvent} by $h \partial_t F$ and integrating by parts, we obtain the exact identity
		\begin{align}
			&\int_0^L h(\partial_t F)^2\dx+\int_0^L \ell^2 h(\partial_x\partial_t F)^2\dx
			\nonumber\\
			&\quad=-\int_0^L\partial_x(\ell^2 h)\,\partial_t F\,\partial_x\partial_t F\dx-\int_0^L h\,\partial_t h\,\partial_x\partial_t F\dx
			+\int_0^L\left(\frac43F-\frac73\partial_x h\right)\partial_t h\,\partial_t F\dx.
			\label{eq:first-weighted-identity}
		\end{align}
		Since $\partial_x(\ell^2 h) = \frac72 \partial_x \ell h \ell$, we have 
		\begin{align*}
			\left|\int_0^L\partial_x(\ell^2 h) \partial_t F \partial_x\partial_t F \dx\right|
			\leq\frac7{64}\left(\int_0^L h(\partial_t F)^2\dx+\int_0^L \ell^2 h(\partial_x\partial_t F)^2\dx\right),
		\end{align*}
		where we used $\|\partial_x\ell\|_{L_\sharp^\infty(0,L)} \leq 1/16$.
		Moreover, by Young's inequality, $\ell = \eps h^{2/3}$, \eqref{eq:slope-cubic}, and \eqref{eq:pointwise} one records
		\begin{align*}
			\left|\int_0^L h \partial_t h \partial_x\partial_t F \dx\right|
			&\leq\frac14\int_0^L \ell^2 h(\partial_x\partial_t F)^2\dx+C\int_0^L h^{-1/3}(\partial_t h)^2\dx,\\
			\left|\int_0^L\left(\frac43F-\frac73 \partial_x h \right) \partial_t h \partial_t F \dx\right|
			&\leq\frac14\int_0^L h(\partial_t F)^2\dx+C\int_0^L h^{-1/3}(\partial_t h)^2\dx.
		\end{align*}
		Therefore, by \eqref{eq:energy-geometry} and \eqref{eq:flux-estimates}, we infer from \eqref{eq:first-weighted-identity} that
		\begin{align*}
			\int_0^L h(\partial_t F)^2\dx+\int_0^L \ell^2h(\partial_x\partial_t F)^2\dx
			\leq C\int_0^L h^{-1/3}(\partial_t h)^2\dx
			\leq C\int_0^L h^{-1}(\partial_t h)^2\dx
			\leq C\cD,
		\end{align*}
		which proves the first inequality of \eqref{eq:time-dissipation}.
		
		Now testing instead by $(h/\ell^2) \partial_t F$, one has
		\begin{align}
			&\int_0^L\frac{h}{\ell^2}(\partial_t F)^2\dx+\int_0^L h(\partial_x\partial_t F)^2\dx
			\nonumber\\
			&\qquad=
			-\int_0^L\partial_x h\,\partial_t F\,\partial_x\partial_t F\dx
			-\int_0^L\frac{h}{\ell^2}\,\partial_t h\,\partial_x\partial_t F\dx
			+\int_0^L\frac{\frac43F-\partial_x h}{\ell^2}\,\partial_t h\,\partial_t F\dx.
			\label{eq:second-weighted-identity}
		\end{align}
		Here $h/\ell^2=\varepsilon^{-2}h^{-1/3}$ and $\ell \partial_x h /h=\frac32 \partial_x \ell$ and $\|\partial_x\ell\|_{L^\infty}\leq 1/16$. Thus the first term on the right-hand side is controlled by
		\begin{align*}
			\left|\int_0^L
			(\partial_x h)(\partial_t F)(\partial_x\partial_t F)\dx\right|
			&\leq
			\frac32\|\partial_x\ell\|_{L^\infty}
			\int_0^L\left|\frac{\sqrt h}{\ell}\,\partial_t F\right|\left|\sqrt h\,\partial_x\partial_t F\right|\dx
			\\
			&\quad\leq
			\frac{3}{64}
			\int_0^L\left(\frac{h}{\ell^2}|\partial_t F|^2+h|\partial_x\partial_t F|^2\right)\dx.
		\end{align*}
		Furthermore, by Young's inequality, $\ell = \eps h^{2/3}$, \eqref{eq:slope-cubic}, and \eqref{eq:pointwise}, it holds
		\begin{align*}
			\left|\int_0^L\frac{h}{\ell^2} \partial_t h \partial_x\partial_t F \dx\right|
			&\leq\frac14\int_0^L h(\partial_x\partial_t F)^2\dx+C\int_0^L h^{-5/3}(\partial_t h)^2\dx,\\
			\left|\int_0^L\frac{\frac43F- \partial_x h}{\ell^2}\, \partial_t h \partial_t F \dx\right|
			&\leq\frac14\int_0^L\frac{h}{\ell^2}(\partial_t F)^2\dx+C\int_0^L h^{-5/3}(\partial_t h)^2\dx,
		\end{align*}
		Combining these bounds with the preceding estimates and \eqref{eq:second-weighted-identity}, we obtain
		\begin{align*}
			\int_0^L\frac {h}{\ell^2}(\partial_t F)^2\dx+\int_0^L h(\partial_x\partial_t F)^2\dx
			\leq C\int_0^L h^{-5/3}(\partial_t h)^2\dx
			\leq C (\min h)^{-2/3}\cD.
		\end{align*}
		Here we have used \eqref{eq:flux-estimates}. This shows the second inequality in \eqref{eq:time-dissipation}. 
		
		Finally, integrating by parts with Lemma~\ref{lem:geometry} leads to
		\begin{align*}
			\left|\int_0^L \partial_x\partial_t h \partial_t F \dx\right|
			=\left|\int_0^L \partial_t h \partial_x\partial_t F \dx\right|
			&\leq\left(\int_0^L\frac{(\partial_t h)^2}{h}\dx\right)^{1/2}
			\left(\int_0^L h(\partial_x\partial_t F)^2\dx\right)^{1/2}\\
			&\leq C (\min h)^{-1/3}\cD\leq C(1+\cS)\cD.\qedhere
		\end{align*}
	\end{proof}
	
	\subsection{The distance estimate}
	Now following the stream function testing procedure as in \cite{GrandmontHillairet2016}, we lift the beam test $\partial_x F$ into the fluid by the stream function
	\begin{align}
		\psi(t,x,y)=F(t,x)\chi\left(\frac{y}{h(t,x)}\right),\qquad
		\chi(z)=3z^2-2z^3,\qquad
		\bw=(- \partial_y \psi, \partial_x \psi).\label{eq:stream-function}
	\end{align}
	Because $\chi(0)=\chi'(0)=\chi'(1)=0$ and $\chi(1)=1$, we infer
	\begin{align}
		\Div\bw=0,\qquad
		\bw(x,0)=0,\qquad
		\bw(x,h(x))= \partial_x F (x)\be_2.\label{eq:test-traces}
	\end{align}
	The mean of $\partial_x F$ is zero, so this is an admissible coupled test. The lifting follows the distance construction of \cite[Section~4.2]{GrandmontHillairet2016}, while the new choice of $F$ changes the viscous identity by a term that we compute explicitly.
	
	\begin{lemma}[Regularity of test functions]\label{lem:distance-admissibility}
		Let $(h,\bu,p)$ satisfy Definition~\ref{def:strong} on $[0,T]$. The multiplier given by \eqref{eq:elliptic} satisfies
		\begin{align}
			F\in C([0,T];H_\sharp^4(0,L))\cap C^1([0,T];H_\sharp^2(0,L)).\label{eq:multiplier-lower-class}
		\end{align}
		The field in \eqref{eq:stream-function} satisfies
		\begin{align}
			\bw\in C([0,T];H^2(\Omega_h)),\qquad
			\partial_t\bw\in C([0,T];L^2(\Omega_h)).\label{eq:test-lower-class}
		\end{align} 
	\end{lemma}
	\begin{proof}
		The height has a positive minimum on the compact interval. Using the equation $\partial_x^2F=\ell^{-2}(F-\partial_xh)$, the coercive weak formulation, and one-dimensional Sobolev multiplication, we first obtain $F\in H_\sharp^2(0,L)$ and then $F\in H_\sharp^4(0,L)$ because $h\in H_\sharp^3(0,L)$. Resolvent estimates for the difference of two coefficient fields imply continuity with values in $H_\sharp^4(0,L)$.
		
		Subtracting \eqref{eq:elliptic} at $t+\tau$ and $t$, and dividing by $\tau$, we obtain
		\begin{align*}
			&\frac{F(t+\tau)-F(t)}{\tau}
			-\ell(t)^2\partial_x^2\frac{F(t+\tau)-F(t)}{\tau}
			=\partial_x\frac{h(t+\tau)-h(t)}{\tau}
			+\frac{\ell(t+\tau)^2-\ell(t)^2}{\tau}\partial_x^2F(t+\tau).
		\end{align*}
		Since $\partial_th\in C([0,T];H_\sharp^1(0,L))$ and $F\in C([0,T];H_\sharp^4(0,L))$, the right-hand side converges in $L_\sharp^2(0,L)$, uniformly on compact time intervals. Elliptic estimates then show that the time derivative belongs to $C([0,T];H_\sharp^2(0,L))$, which gives \eqref{eq:multiplier-lower-class}. 
		
		The stream function consequently has three spatial derivatives in $L^2$, while its Eulerian time derivative is
		\begin{align*}
			\partial_t\psi=\partial_tF\,\chi
			-F\frac{\partial_th}{h}\frac yh\chi',
		\end{align*}
		which has one spatial derivative in $L^2$ continuously in time. This proves \eqref{eq:test-lower-class}. Also $\partial_xF\in C([0,T];H_\sharp^3(0,L))\cap C^1([0,T];H_\sharp^1(0,L))$, so
		\begin{align*}
			\left\langle\partial_t^2h,\partial_xF\right\rangle_{H_\sharp^{-1/2},H_\sharp^{1/2}}
		\end{align*}
		is integrable and the product rule for $\int_0^L\partial_th\,\partial_xF\dx$ is valid. After pullback by $\chi_h$, the fluid product rule follows from the identity $\partial_t(\bu\circ\chi_h)=(\partial_t\bu)\circ\chi_h+z\partial_th\,(\partial_y\bu)\circ\chi_h\in L^2(0,T;L^2(\Omega_1))$, the preceding regularity of $\bw$, and the Jacobian $h$. The displayed regularity follows from the positive gap and $\partial_th\in C([0,T];H_\sharp^1(0,L))$. Thus Reynolds transport and the integrations by parts used below hold down to zero, by temporal regularization if necessary. In the beam acceleration term of the coupled identity, the integral is understood as the displayed duality pairing.
		
		The positive minimum used here justifies the operations on an existing solution. Every inequality used to estimate the distance is proved separately with the gap dependence displayed explicitly.
	\end{proof}
	
	Next, we introduce the corresponding test with respect to the fluid viscosity which generates the distance estimate essentially. 
	\begin{lemma}\label{lem:stress-identity}
		For the velocity and test above, we have
		\begin{align}
			2\mu& \int_{\Omega_h}D\bu:D\bw\dA \nonumber\\
			& =-6\mu \frac{\mathrm d}{\mathrm dt} \int_0^L \frac{\dx}h 
			- 12\mu\varepsilon^2\int_0^L \frac{Q}{h^{5/3}} \partial_x^2 F \dx+\mu\int_{\Omega_h} \partial_x^2 \psi (\partial_x u_2 -2 \partial_y u_1)\dA.
			\label{eq:distance-viscosity}
		\end{align}
	\end{lemma}
	
	\begin{proof}
		First, integrating by parts and recalling $\Div\bu=0$, we find
		\begin{align*}
			\int_{\Omega_h}(\nabla\bu)^\top:\nabla\bw\dA
			=\int_0^L \partial_x F (- \partial_x h \partial_y u_1 + \partial_y u_2)(x,h(x))\dx=0.
		\end{align*}
		Indeed, $\partial_y u_2 =- \partial_x u_1$, and differentiation of the trace $u_1(x,h(x))=0$ implies $\partial_x u_1 + \partial_x h \partial_y u_1 =0$. The lower-boundary term vanishes by \eqref{eq:test-traces}, and the lateral terms cancel by periodicity. The same computation with $\bw=\bu$ proves $2\int_{\Omega_h}|D\bu|^2\dA=\int_{\Omega_h}|\nabla\bu|^2\dA$.
		
		By incompressibility and the definition of $\bw$, it holds
		\begin{align*}
			\nabla\bu:\nabla\bw
			=2 \partial_y u_2 \partial_x\partial_y \psi - \partial_y u_1 \partial_y^2 \psi + \partial_x u_2 \partial_x^2 \psi.
		\end{align*}
		Since $u_1|_{y=0} = u_1|_{y=h} = 0$, $\partial_y\chi(z)=\chi'(z)/h$ and $\partial_y^3\chi(z)=\chi'''(z)/h^3=-12/h^3$,
		\begin{align*}
			-\int_{\Omega_h} \partial_y u_1 \partial_y^2 \psi \dA
			=\int_{\Omega_h}u_1 \partial_y^3 \psi \dA
			=-12\int_0^L\frac{QF}{h^3}\dx.
		\end{align*}
		The first term satisfies
		\begin{align*}
			2\int_{\Omega_h} \partial_y u_2 \partial_x\partial_y \psi \dA
			=-2\int_{\Omega_h} \partial_x u_1 \partial_x\partial_y \psi \dA
			=2\int_{\Omega_h}u_1 \partial_x^2\partial_y \psi \dA
			=-2\int_{\Omega_h} \partial_y u_1 \partial_x^2 \psi \dA.
		\end{align*}
		Invoking identities above, we arrive at
		\begin{align}
			2\int_{\Omega_h}D\bu:D\bw\dA
			=-12\int_0^L\frac{QF}{h^3}\dx
			+\int_{\Omega_h} \partial_x^2 \psi (\partial_x u_2 -2 \partial_y u_1)\dA.\label{eq:stress-identity}
		\end{align}
		By $\partial_t h =- \partial_x Q$, we compute
		\begin{align}
			\frac{\mathrm d}{\mathrm dt} \int_0^L \frac{\dx}h  =-\int_0^L\frac{\partial_t h}{h^2}\dx
			=2\int_0^L\frac{Q \partial_x h}{h^3}\dx.\label{eq:entropy-derivative}
		\end{align}
		Substituting $F= \partial_x h +\varepsilon^2h^{4/3} \partial_x^2 F$ into \eqref{eq:stress-identity} with \eqref{eq:entropy-derivative} yields \eqref{eq:distance-viscosity}.
	\end{proof}
	
	\begin{remark}
		The proof of Lemma~\ref{lem:stress-identity} does not introduce a periodic auxiliary pressure with derivative $12F/h^3$. Such a primitive need not exist, since the resolvent does not imply $\int_0^L F/h^3\dx=0$. The direct flux calculation also retains a possible nonzero mean of $Q$.
	\end{remark}
	By Lemma \ref{lem:elastic}, we know
	\begin{align}
		\cB\coloneqq\beta\int_0^L \partial_x^2 h \partial_x F \dx+\alpha\int_0^L \partial_x^3 h \partial_x^2 F \dx,\label{eq:elastic-form}
	\end{align}
	is coercive in suitable spaces. Then one shall control the additional perturbation due to elliptic regularization by the following.
	\begin{lemma}\label{lem:defect}
		It holds
		\begin{align}
			\left|12\mu\varepsilon^2\int_0^L \frac{Q}{h^{5/3}}\partial_x^2F\dx\right|
			\leq\frac12\cB+C(1+\cS)\cD.
		\end{align}
	\end{lemma}
	
	\begin{proof}
		By the flux bound in \eqref{eq:flux-estimates}, we have
		\begin{align*}
			\left|12\mu\varepsilon^2\int_0^L \frac{Q}{h^{5/3}}\partial_x^2F\dx\right|
			&\leq\sqrt{12}\mu\varepsilon^2\sqrt\cD
			\left(\int_0^L h^{-1/3}(\partial_x^2 F)^2\dx\right)^{1/2}\\
			&\leq\sqrt{12}\mu\varepsilon^2(\min h)^{-1/6}\sqrt\cD\norm{\partial_x^2 F}_{L_\sharp^2(0,L)}.
		\end{align*}
		By Young's inequality and \eqref{eq:elastic-coercivity}, we have
		\begin{align*}
			\left|12\mu\varepsilon^2\int_0^L \frac{Q}{h^{5/3}}\partial_x^2F\dx\right|\leq\frac12\frac{15\alpha}{16}\norm{\partial_x^2 F}_{L_\sharp^2(0,L)}^2
			+C (\min h)^{-1/3}\cD
			\leq\frac12\cB+C(1+\cS)\cD,
		\end{align*}
		where the last step uses \eqref{eq:min-entropy}.
	\end{proof}
	
	\begin{lemma}\label{lem:fluid-remainders}
		The velocity test holds
		\begin{align}
			\left|\int_{\Omega_h}\bu\cdot
			(\partial_t \bw +\bu\cdot\nabla\bw)\dA\right|
			\leq C\cD.\label{eq:fluid-inertia-bound}
		\end{align}
	\end{lemma}
	
	\begin{proof}
		We first record the second spatial derivative of the stream function. Differentiating at fixed physical coordinates, with the profile and its derivatives evaluated at $z=y/h\in(0,1)$, we obtain
		\begin{align}
			\partial_x^2 \psi &= \partial_x^2 F \chi
			-\left(\frac{2 \partial_x F \partial_x h}{h}+\frac{F \partial_x^2 h}h\right)\frac yh\chi'
			+\frac{F(\partial_x h)^2}{h^2}\Big(2\frac yh\chi'+\frac{y^2}{h^2}\chi''\Big).\label{eq:psi-second}
		\end{align}
		For the time derivative, we use the formula established in the proof of Lemma~\ref{lem:distance-admissibility}. Since the profile and its derivatives are bounded on $[0,1]$, \eqref{eq:time-dissipation} implies
		\begin{align*}
			\norm{\partial_t \psi}_{L^2(\Omega_h)}^2
			\leq C\int_0^L h(\partial_t F)^2\dx+C\int_0^L\frac{F^2(\partial_t h)^2}{h}\dx
			\leq C\cD,
		\end{align*}
		where we used $\norm{F}_{L_\sharp^\infty(0,L)}\leq C$ and \eqref{eq:flux-estimates} for the second term.
		
		By \eqref{eq:length-lipschitz} and \eqref{eq:pointwise}, every term in \eqref{eq:psi-second} other than the one containing $\partial_x^2 h$ is bounded by $Ch^{-1}$. The remaining term is controlled by $C| \partial_x^2 h |h^{-2/3}$. Therefore invoking \eqref{eq:min-entropy}, one has
		\begin{align*}
			\norm{\partial_x^2 \psi}_{L^2(\Omega_h)}^2
			\leq C\left(\cS+\int_0^L \frac{(\partial_x^2 h)^2}{h^{1/3}}\dx\right)
			\leq C(\cS+(\min h)^{-1/3})\leq C(1+\cS).
		\end{align*}
		In view of Cauchy--Schwarz inequality, one further shows 
		\begin{gather}
			\left|\mu\int_{\Omega_h}\partial_x^2\psi(\partial_xu_2-2\partial_yu_1)\dA\right|\leq C\sqrt{(1+\cS)\cD}.
		\end{gather}
		
		For the time contribution, spatial integration by parts leads to the exact identity
		\begin{align*}
			\int_{\Omega_h}\bu\cdot \partial_t \bw \dA
			=-\int_0^L \partial_t h \partial_x h \partial_t F \dx
			+\int_{\Omega_h}(\partial_y u_1 - \partial_x u_2) \partial_t \psi \dA.
		\end{align*}
		Here the top trace of $\partial_t \psi$ is $\partial_t F$, since $\chi'(1)=0$. By \eqref{eq:flux-estimates} and \eqref{eq:time-dissipation}, one derives
		\begin{align*}
			\left|\int_0^L \partial_t h \partial_x h \partial_t F \dx\right|
			&\leq\norm{\partial_x h}_{L_\sharp^\infty(0,L)}
			\left(\int_0^L\frac{(\partial_t h)^2}{h}\dx\right)^{1/2}
			\left(\int_0^L h(\partial_t F)^2\dx\right)^{1/2}\leq C\cD.
		\end{align*}
		Moreover, it holds
		\begin{align*}
			\left|\int_{\Omega_h}(\partial_y u_1 - \partial_x u_2) \partial_t \psi \dA\right|
			&\leq C\sqrt{\cD} \norm{\partial_t \psi}_{L^2(\Omega_h)} \leq C \cD.
		\end{align*}
		Then 
		\begin{align*}
			\left|\int_{\Omega_h}\bu\cdot \partial_t \bw \dA\right|
			\leq C \cD.
		\end{align*}
		
		For convection, incompressibility and the top traces yield
		\begin{align*}
			\int_{\Omega_h}\bu\cdot(\bu\cdot\nabla\bw)\dA
			=\int_0^L (\partial_t h)^2 \partial_x F \dx
			-\int_{\Omega_h}(\bu\cdot\nabla\bu)\cdot\bw\dA.
		\end{align*}
		The first term is bounded by $\norm{h \partial_x F}_{L_\sharp^\infty(0,L)}\int_0^L (\partial_t h)^2/h\dx\leq C\cD$. From \eqref{eq:stream-function} and \eqref{eq:pointwise},
		\begin{align}
			|\bw(x,y)|\leq C\bigl(h(x)^{-2/3}+h(x)^{-1/3}\bigr).\label{eq:test-pointwise}
		\end{align}
		Applying \eqref{eq:column-poincare} to the other term, we find
		\begin{align*}
			&\left|\int_{\Omega_h}(\bu\cdot\nabla\bu)\cdot\bw\dA\right|
			\leq C\int_0^L
			\norm{\partial_y\bu(x,\cdot)}_{L^2(0,h)}
			\norm{\nabla\bu(x,\cdot)}_{L^2(0,h)}\dx
			\leq C\cD.
		\end{align*}
		Indeed, multiplying the right-hand side of \eqref{eq:test-pointwise} by $h$ leads to a bounded coefficient by \eqref{eq:energy-geometry}. Combining the estimates for time and convection proves \eqref{eq:fluid-inertia-bound}.
	\end{proof}
	
	Next, we introduce the functional arising from the coupled test. The natural energy bound makes it equivalent, up to an additive constant, to $1+\cS$.
	\begin{lemma}
		Let
		\begin{align}
			\mathcal Z=6\mu \int_0^L \frac{\dx}{h}
			-\rho_s\int_0^L \partial_t h \partial_x F \dx
			-\rho_f\int_{\Omega_h}\bu\cdot\bw\dA.
		\end{align}
		There is a constant $C_*>0$, depending only on the fixed energy parameters, such that
		\begin{align}
			3\mu(1+\cS)\leq \mathcal Z+C_*\leq C(1+\cS).\label{eq:functional-equivalence}
		\end{align}
	\end{lemma}
	
	\begin{proof}
		By \eqref{eq:pointwise} and H\"older's inequality,
		\begin{align*}
			\norm{\partial_x F}_{L_\sharp^2(0,L)}^2\leq C\int_0^L h^{-2/3}\dx
			\leq C L^{1/3}\cS^{2/3}.
		\end{align*}
		The energy bound therefore implies
		\begin{align}
			\left|\rho_s\int_0^L \partial_t h \partial_x F \dx\right|\leq C\cS^{1/3}.\label{eq:beam-pairing}
		\end{align}
		Further by vertical integration of \eqref{eq:test-pointwise}, we obtain
		\begin{align*}
			\norm{\bw}_{L^2(\Omega_h)}^2
			\leq C\left(1+\int_0^L h^{-1/3}\dx\right)
			\leq C(1+\cS^{1/3}),
		\end{align*}
		and hence
		\begin{align}
			\left|\rho_f\int_{\Omega_h}\bu\cdot\bw\dA\right|
			\leq C(1+\cS^{1/6}).\label{eq:fluid-pairing}
		\end{align}
		Since both exponents are strictly below one, the sum of the absolute values in \eqref{eq:beam-pairing}--\eqref{eq:fluid-pairing} is at most $3\mu \cS+C_*-3\mu$ for a sufficiently large fixed $C_*\geq3\mu$. Consequently, $\mathcal Z+C_*\geq3\mu(1+\cS)$. The upper bound follows from the same estimates.
	\end{proof}
	
	\subsection{Proof of Theorem~\ref{thm:separation}}\label{sec:separation-proof}
	We first work on a compact interval before the maximal existence time. The coupled test $(\bw,\partial_x F)$ is admissible by Lemma~\ref{lem:distance-admissibility}. Testing the fluid equation \eqref{eq:fluid} by $\bw$ and the beam equation \eqref{eq:beam} by $\partial_x F$ yields
	\begin{align}
		&\rho_f\int_{\Omega_h}(\partial_t \bu +\bu\cdot\nabla\bu)\cdot\bw\dA
		+2\mu\int_{\Omega_h}D\bu:D\bw\dA
		+\rho_s\int_0^L \partial_t^2 h \partial_x F \dx \nonumber\\
		&\qquad =\beta\int_0^L \partial_x^2 h \partial_x F \dx+\alpha\int_0^L \partial_x^3 h \partial_x^2 F \dx,\label{eq:coupled-test}
	\end{align}
	where the traction term cancels after integration by parts. For the beam pairing,
	\begin{align}
		\int_0^L \partial_t^2 h \partial_x F \dx
		=\partial_t\int_0^L \partial_t h \partial_x F \dx+\int_0^L \partial_x\partial_t h \partial_t F \dx.\label{eq:beam-pairing-time}
	\end{align}
	By Reynolds transport,
	\begin{align*}
		\partial_t\int_{\Omega_h}\bu\cdot\bw\dA
		&=\int_{\Omega_h}(\partial_t \bu \cdot\bw+\bu\cdot \partial_t \bw)\dA
		+\int_0^L (\partial_t h)^2 \partial_x F \dx.
	\end{align*}
	Combining this identity with the convective integration by parts in the proof of Lemma~\ref{lem:fluid-remainders}, the boundary contributions cancel in the material-derivative pairing, and we obtain
	\begin{align}
		\int_{\Omega_h}(\partial_t \bu +\bu\cdot\nabla\bu)\cdot\bw\dA
		=\partial_t\int_{\Omega_h}\bu\cdot\bw\dA
		-\int_{\Omega_h}\bu\cdot(\partial_t \bw +\bu\cdot\nabla\bw)\dA.
		\label{eq:fluid-pairing-time}
	\end{align}
	On every compact positive-gap interval, the admissibility and the estimates below show that $\mathcal Z$ is absolutely continuous. The following identities hold almost everywhere and in their time-integrated form. Substituting \eqref{eq:distance-viscosity}, \eqref{eq:beam-pairing-time}, and \eqref{eq:fluid-pairing-time} into \eqref{eq:coupled-test} yields
	\begin{align}
		\partial_t \mathcal Z+\cB={}&\rho_s\int_0^L\partial_x\partial_t h\,\partial_tF\dx
		+\mu\int_{\Omega_h}\partial_x^2\psi(\partial_xu_2-2\partial_yu_1)\dA\nonumber\\
		&-12\mu\varepsilon^2\int_0^L\frac{Q}{h^{5/3}}\partial_x^2F\dx
		-\rho_f\int_{\Omega_h}\bu\cdot(\partial_t\bw+\bu\cdot\nabla\bw)\dA.
	\end{align}
	Here $\cB$ is defined in \eqref{eq:elastic-form} as the right-hand side of \eqref{eq:coupled-test}, which is coercive by Lemma \ref{lem:elastic}. Then recalling Lemmata~\ref{lem:time}, \ref{lem:defect}, and \ref{lem:fluid-remainders}, one has
	\begin{align*}
		\partial_t (\mathcal Z+C_*) +\frac12\cB
		\leq C(1+\cS)\cD+C\sqrt{(1+\cS)\cD}+C\cD.
	\end{align*}
	Using \eqref{eq:functional-equivalence}, we conclude that
	\begin{align}
		\partial_t (\mathcal Z+C_*) +\frac12\cB\leq C(1+\cD)(\mathcal Z+C_*).
	\end{align}
	Every coefficient on the right is controlled by the natural energy. Since $\cB\geq0$, Gronwall's inequality and \eqref{eq:energy-identity} imply
	\begin{align*}
		(\mathcal Z+C_*)(t)&\leq (\mathcal Z+C_*)(0)\exp\left(C\left[t+\int_0^t\cD(s)\,\mathrm ds\right]\right)
		\leq C(1+\cS(0))\exp\left(C\left[t+\frac{\cE(0)}\mu\right]\right).
	\end{align*}
	Absorbing the energy-dependent factor into $C_1$ proves \eqref{eq:entropy-bound}. Lemma~\ref{lem:geometry} then further implies \eqref{eq:gap-bound}. The constants are independent of the compact interval and of its positive minimum height, so the estimates hold for all $t<T_*$. \qed
	\begin{remark}
		Theorem~\ref{thm:separation} is an energy-controlled estimate for solutions in Definition~\ref{def:strong}. The regularity of that class justifies the coupled test through Lemma~\ref{lem:distance-admissibility}, while the resulting bound on each finite interval $[0,T]$ depends only on $T$, the mechanical energy, the conserved mean height, the period, the physical parameters, and the initial reciprocal-height integral. In Section~\ref{sec:global-proof}, these estimates hold uniformly for the exact higher strong approximations and pass to the constructed solution through the uniform convergence of the heights and their common positive lower bound on each finite interval. Therefore, one may probably ask if the uniform separation property can hold true for weak solutions fulfilling the coupled weak formulation. 
	\end{remark} 
	
	\section{Uniform Regularity}\label{sec:lower-estimates}
	Throughout this section, we work on a fixed interval $[0,T]$ with an exact higher strong solution of \eqref{eq:model}, in the class supplied by \cite[Theorems~1.1 and~4.7]{SchwarzacherSu2025} and used in Section~\ref{sec:global-proof}. We assume that $h(t,x)\geq \min h>0$ on this interval. Higher regularity is used to justify the coupled tests, with regularization of the test functions where needed. Apart from their stated dependence on $T$, the constants below depend only on the mechanical energy, $(\min h)^{-1}$, the conserved mean height, the period, and the physical parameters, together with the lower initial norms when initial values occur. We retain the physical dissipation $\cD$ as an integrable coefficient and combine tangential and filtered tests with a zero-trace fluid acceleration test. Section~\ref{sec:global-proof} transfers these bounds to the lower strong solution.
	
	\subsection{The transformed system and main estimate}
	To study the regularity property, it is useful to transform the Eulerian moving domain to a fixed reference configuration.
	We employ the change of variables $\chi_h(t,x,z)=(x,h(t,x)z)$, with $x$ periodic and $0<z<1$, and write $\widehat\bu=\bu\circ\chi_h$ and $\widehat p=p\circ\chi_h$. The geometric matrices are
	\begin{align}
		B_h=\begin{pmatrix}h&-z\partial_xh\\0&1\end{pmatrix},
		\qquad
		A_h=\begin{pmatrix}h&-z\partial_xh\\-z\partial_xh&(1+z^2|\partial_xh|^2)/h\end{pmatrix}.
	\end{align}
	By $\chi_h$ we know the transformed velocity is not necessarily solenoidal and the nonhomogeneous boundary data varies. To preserve the properties, one adopts the so-called \emph{Piola transform} such that 
	\begin{align}
		\bv=B_h^\top\widehat\bu,
		\label{eq:piola-velocity}
	\end{align}
	fulfilling
	\begin{align}
		\Div\bv=0,\qquad
		\bv|_{z=1}=\partial_th\be_2,\qquad
		\bv|_{z=0}=0,\qquad
		\widehat\bu=B_h^{-\top}\bv.
	\end{align}
	For every smooth periodic test pair $(\bphi,\zeta)$ satisfying $\Div\bphi=0$, $\bphi|_{z=0}=0$, and $\bphi|_{z=1}=\zeta\be_2$, the coupled equations then take the form
	\begin{align}
		&\rho_f\int_0^L\int_0^1 h(B_h^{-\top}\partial_t\bv)\cdot(B_h^{-\top}\bphi)\dz\,\mathrm dx
		+\rho_s\int_0^L\partial_t^2h\,\zeta\dx\nonumber\\
		&\quad+\mu\sum_{i,j=1}^2\int_0^L\int_0^1(A_h)_{ij}
		\partial_i\widehat\bu\cdot\partial_j(B_h^{-\top}\bphi)\dz\,\mathrm dx
		+\alpha\int_0^L\partial_x^2h\,\partial_x^2\zeta\dx
		+\beta\int_0^L\partial_xh\,\partial_x\zeta\dx
		\label{eq:lower-coupled}\\
		&\quad+\rho_f\int_0^L\int_0^1h(\partial_tB_h^{-\top})\bv\cdot B_h^{-\top}\bphi\dz\,\mathrm dx
		+\rho_f\int_0^L\int_0^1
		\big[(\widehat\bu-\partial_t\chi_h)\cdot(B_h\nabla)\big]\widehat\bu
		\cdot B_h^{-\top}\bphi\dz\,\mathrm dx=0.\nonumber
	\end{align}
	Here $\partial_1=\partial_x$ and $\partial_2=\partial_z$. The transport field satisfies
	\begin{align}
		B_h^\top(\widehat\bu-\partial_t\chi_h)
		=\bv-z\partial_th\be_2.
	\end{align}
	and the transformed fluid momentum equation is
	\begin{align}
		&\rho_fhB_h^{-1}B_h^{-\top}\partial_t\bv
		-\mu B_h^{-1}\Div\big(A_h\nabla(B_h^{-\top}\bv)\big)
		+\nabla\widehat p\nonumber\\
		&\quad=-\rho_fhB_h^{-1}(\partial_tB_h^{-\top})\bv
		-\rho_fB_h^{-1}
		\big[(\widehat\bu-\partial_t\chi_h)\cdot(B_h\nabla)\big]\widehat\bu.
		\label{eq:lower-strong-fluid}
	\end{align}
	Here, we have used by the chain rule that
	\begin{align*}
		(\nabla p)\circ\chi_h&=h^{-1}B_h\nabla\widehat p,\qquad
		(\Delta\bu)\circ\chi_h=h^{-1}\Div(A_h\nabla\widehat\bu),\\
		(\partial_t\bu+\bu\cdot\nabla\bu)\circ\chi_h
		&=\partial_t\widehat\bu+h^{-1}\big[(\widehat\bu-\partial_t\chi_h)\cdot(B_h\nabla)\big]\widehat\bu.
	\end{align*}
	with the Jacobian of $\chi_h$ that is $h$.

	Then we have the regularity estimate.	
	\begin{proposition}[Lower-order regularity estimate]\label{prop:lower-estimates}
		Let $0<T<\infty$ and let $(h,\bu,p)$ be an exact higher strong solution satisfying the standing assumptions of this section, with $h\geq \min h>0$ on $[0,T]\times[0,L]$. Then
		\begin{align}
			&\sup_{0\leq t\leq T}\left(
			\norm{\bv(t)}_{H^1(\Omega_1)}^2
			+\norm{h(t)}_{H_\sharp^3(0,L)}^2
			+\norm{\partial_th(t)}_{H_\sharp^1(0,L)}^2\right)\nonumber\\
			&\quad+\int_0^T\left(
			\norm{\partial_t\bv}_{L^2(\Omega_1)}^2
			+\norm{\bv}_{H^2(\Omega_1)}^2
			+\norm{\widehat p}_{H^1(\Omega_1)}^2\right)\dt\nonumber\\
			&\quad+\int_0^T\left(
			\norm{h}_{H_\sharp^{7/2}(0,L)}^2
			+\norm{\partial_th}_{H_\sharp^{3/2}(0,L)}^2
			+\norm{\partial_t^2h}_{H_\sharp^{-1/2}(0,L)}^2\right)\dt\leq C_T.
			\label{eq:lower-bound}
		\end{align}
		The constant $C_T > 0$ depends on $T$, $(\min h)^{-1}$, the initial mechanical energy, the mean height, the period, the physical parameters, and the three initial norms in \eqref{eq:data-regularity}.   
	\end{proposition}
	
	\subsection{Geometric coefficients estimates}\label{sec:lower-coefficients}
	We first make the coefficient and product estimates explicit.
	Direct calculation gives
	\begin{align*}
		B_h^{-\top}
		&=\begin{pmatrix}
			h^{-1}&0\\
			z\partial_xh/h&1
		\end{pmatrix},\quad
		\partial_xB_h^{-\top}
		=\begin{pmatrix}
			-\partial_xh/h^2&0\\
			z\big(\partial_x^2h/h-|\partial_xh|^2/h^2\big)&0
		\end{pmatrix},\\
		\partial_x^2B_h^{-\top}
		&=\begin{pmatrix}
			-\partial_x^2h/h^2+2|\partial_xh|^2/h^3&0\\
			z\big(\partial_x^3h/h
			-3\partial_xh\,\partial_x^2h/h^2
			+2(\partial_xh)^3/h^3\big)&0
		\end{pmatrix},\quad
		\partial_z\partial_xB_h^{-\top}
		=\begin{pmatrix}
			0&0\\
			\partial_x^2h/h-|\partial_xh|^2/h^2&0
		\end{pmatrix}.
	\end{align*}
	Moreover,
	\begin{align*}
		\partial_xA_h
		=\begin{pmatrix}
			\partial_xh&-z\partial_x^2h\\
			-z\partial_x^2h&
			\displaystyle
			\frac{2z^2\partial_xh\,\partial_x^2h}{h}
			-\frac{(1+z^2|\partial_xh|^2)\partial_xh}{h^2}
		\end{pmatrix}.
	\end{align*}
	Since the energy and separation bounds control $h$, $h^{-1}$, and $\partial_xh$ in $L_\sharp^\infty(0,L)$, these formulas imply
	\begin{gather}
		\norm{A_h}_{L^\infty(\Omega_1)}
		+\norm{B_h^{-\top}}_{L^\infty(\Omega_1)}
		+\norm{\partial_zB_h^{-\top}}_{L^\infty(\Omega_1)}
		\leq C,\\
		|\partial_xA_h|+|\partial_xB_h^{-\top}|
		\leq C(1+|\partial_x^2h|),\quad
		|\nabla\partial_xB_h^{-\top}|
		\leq C(1+|\partial_x^2h|+|\partial_x^3h|). \label{eq:geometric-bound}
	\end{gather}
	In particular,
	\begin{align*}
		\norm{\partial_xA_h}_{L_\sharp^2(0,L;L^\infty(0,1))}
		+\norm{\partial_xB_h^{-\top}}_{L_\sharp^2(0,L;L^\infty(0,1))}
		\leq C.
	\end{align*}
	Notice that $\nabla\partial_xB_h^{-\top}$ contains $\partial_x^3h$ only linearly.
	
	We next record the estimates used below. By \eqref{eq:lower-energy-products} and \eqref{eq:lower-anisotropic},
	\begin{align*}
		\norm{\mathbf v}_{L_\sharp^\infty(0,L;L^2(0,1))}
		&\leq C(1+\cD)^{1/4},\\
		\norm{\partial_x\mathbf v}_{L_\sharp^\infty(0,L;L^2(0,1))}
		&\leq C\norm{\partial_x\mathbf v}_{L^2(\Omega_1)}^{1/2}
		\norm{\partial_x\mathbf v}_{H^1(\Omega_1)}^{1/2}
		\leq C(1+\cD)^{1/4}
		\norm{\partial_x\mathbf v}_{H^1(\Omega_1)}^{1/2},\\
		\norm{\nabla\mathbf v}_{L_\sharp^\infty(0,L;L^2(0,1))}
		&\leq C\norm{\nabla\mathbf v}_{L^2(\Omega_1)}^{1/2}
		\left(
		\norm{\nabla\mathbf v}_{L^2(\Omega_1)}
		+\norm{\partial_x\nabla\mathbf v}_{L^2(\Omega_1)}
		\right)^{1/2}\\
		&\leq C(1+\cD)^{1/2}
		+C(1+\cD)^{1/4}
		\norm{\partial_x\mathbf v}_{H^1(\Omega_1)}^{1/2}.
	\end{align*}
	Then the products involving a coefficient depending only on $x$ satisfy
	\begin{align*}
		\norm{\partial_x^2h\,\nabla\mathbf v}_{L^2(\Omega_1)}
		&\leq
		\norm{\partial_x^2h}_{L_\sharp^2(0,L)}
		\norm{\nabla\mathbf v}_{L_\sharp^\infty(0,L;L^2(0,1))}
		\leq C(1+\cD)^{1/2}
		+C(1+\cD)^{1/4}
		\norm{\partial_x\mathbf v}_{H^1(\Omega_1)}^{1/2},\\
		\norm{\partial_x^3h\,\mathbf v}_{L^2(\Omega_1)}
		&\leq
		\norm{\partial_x^3h}_{L_\sharp^2(0,L)}
		\norm{\mathbf v}_{L_\sharp^\infty(0,L;L^2(0,1))}
		\leq C\norm{\partial_x^3h}_{L_\sharp^2(0,L)}
		(1+\cD)^{1/4},\\
		\norm{\partial_x^2h\,\mathbf v}_{L^2(\Omega_1)}
		&\leq C(1+\cD)^{1/4}.
	\end{align*}
	Expanding the derivative of $(\partial_xB_h^{-\top})\mathbf v$ and applying the preceding product estimates, we obtain
	\begin{align}
		\norm{\nabla((\partial_xB_h^{-\top})\mathbf v)}_{L^2(\Omega_1)}
		&\leq C\left(
		\norm{\mathbf v}_{H^1(\Omega_1)}
		+\norm{\partial_x^2h\,\mathbf v}_{L^2(\Omega_1)}
		+\norm{\partial_x^3h\,\mathbf v}_{L^2(\Omega_1)}
		+\norm{\partial_x^2h\,\nabla\mathbf v}_{L^2(\Omega_1)}
		\right)\nonumber\\
		&\leq
		C\left(1+\norm{\partial_x^3h}_{L_\sharp^2(0,L)}\right)
		(1+\cD)^{1/2}
		+C(1+\cD)^{1/4}
		\norm{\partial_x\mathbf v}_{H^1(\Omega_1)}^{1/2}.
		\label{eq:lower-matrix-product}
	\end{align}
	
	To estimate the column supremum of $\nabla\widehat\bu$, first observe that the change of variables and the multiplication bounds from Lemma~\ref{lem:lower-coefficients} yield
	\begin{align*}
		\norm{\nabla\widehat\bu}_{L^2(\Omega_1)}
		&\leq C\cD^{1/2},\\
		\norm{\nabla(B_h^{-\top}\partial_x\mathbf v)}_{L^2(\Omega_1)}
		&\leq C\norm{\partial_x\mathbf v}_{H^1(\Omega_1)}.
	\end{align*}
	The second estimate uses only the energy bound for $\partial_x^2h$. Namely, the coefficient product arising when the derivative falls on $B_h^{-\top}$ satisfies
	\begin{align*}
		\norm{\partial_x^2h\,\partial_x\mathbf v}_{L^2(\Omega_1)}
		\leq
		\norm{\partial_x^2h}_{L_\sharp^2(0,L)}
		\norm{\partial_x\mathbf v}_{L_\sharp^\infty(0,L;L^2(0,1))}
		\leq C\norm{\partial_x\mathbf v}_{H^1(\Omega_1)}.
	\end{align*}
	Since $\widehat\bu=B_h^{-\top}\mathbf v$, we have
	\begin{align*}
		\partial_x\nabla\widehat\bu
		=\nabla(B_h^{-\top}\partial_x\mathbf v)
		+\nabla((\partial_xB_h^{-\top})\mathbf v).
	\end{align*}
	Using \eqref{eq:lower-matrix-product} and Young's inequality, we obtain
	\begin{align*}
		\norm{\partial_x\nabla\widehat\bu}_{L^2(\Omega_1)}
		\leq C\norm{\partial_x\mathbf v}_{H^1(\Omega_1)}
		+C\left(1+\norm{\partial_x^3h}_{L_\sharp^2(0,L)}\right)
		(1+\cD)^{1/2}.
	\end{align*}
	By \eqref{eq:lower-anisotropic}, applied to $\nabla\widehat\bu$, one concludes
	\begin{align}
		\norm{\nabla\widehat\bu}_{L_\sharp^\infty(0,L;L^2(0,1))}
		&\leq C\norm{\nabla\widehat\bu}_{L^2(\Omega_1)}^{1/2}
		\left(
		\norm{\nabla\widehat\bu}_{L^2(\Omega_1)}
		+\norm{\partial_x\nabla\widehat\bu}_{L^2(\Omega_1)}
		\right)^{1/2}\nonumber\\
		&\quad\leq
		C(1+\cD)^{1/4}
		\norm{\partial_x\mathbf v}_{H^1(\Omega_1)}^{1/2}
		+C(1+\cD)^{1/2}
		\left(1+\norm{\partial_x^3h}_{L_\sharp^2(0,L)}\right)^{1/2}.
		\label{eq:lower-gradient-column}
	\end{align}
	
	Before deriving the estimates, we define the $H^1$-level functional
	\begin{align}
		\cE_1(t)\coloneqq1+\norm{\bv(t)}_{H^1(\Omega_1)}^2
		+\norm{h(t)}_{H_\sharp^3(0,L)}^2
		+\norm{\partial_th(t)}_{H_\sharp^1(0,L)}^2.
	\end{align}
	Moreover, we would like to establish certain equivalency between physical norms and reference norms, by controlling the geometric quantities from transformation.
	\begin{lemma}\label{lem:lower-coefficients}
		Under the standing assumptions of this section, the matrices $A_h$ and $B_h^{-1}B_h^{-\top}$ are uniformly elliptic and have uniformly bounded entries in $L^\infty(0,1;H_\sharp^1(0,L))$. Moreover,
		\begin{gather}
			C^{-1}\norm\bu_{H^1(\Omega_h)}
			\leq\norm{\bv}_{H^1(\Omega_1)}
			\leq C\norm\bu_{H^1(\Omega_h)},\nonumber\\
			\norm{\bv}_{H^1(\Omega_1)}^2\leq C\cD,
			\qquad
			\norm{\bv}_{L_\sharp^\infty(0,L;L^2(0,1))}^2\leq C(1+\cD)^{1/2}.
			\label{eq:lower-energy-products}
		\end{gather}
		For every $x$-periodic $\mathbf z\in H^1(\Omega_1)$ with zero bottom trace,
		\begin{align}
			C^{-1}\norm{\mathbf z}_{H^1(\Omega_1)}^2
			&\leq\sum_{i,j=1}^2\int_{\Omega_1}(A_h)_{ij}
			\partial_i(B_h^{-\top}\mathbf z)\cdot\partial_j(B_h^{-\top}\mathbf z)\dx\dz
			\leq C\norm{\mathbf z}_{H^1(\Omega_1)}^2.
			\label{eq:lower-uniform-coercivity}
		\end{align}
		The physical and transformed second-order spatial norms satisfy
		\begin{align}
			\norm{\bv}_{H^2(\Omega_1)}^2
			\leq C\norm\bu_{H^2(\Omega_h)}^2+C(1+\cD)\cE_1,\quad
			\norm\bu_{H^2(\Omega_h)}^2
			\leq C\norm{\bv}_{H^2(\Omega_1)}^2+C(1+\cD)\cE_1.
			\label{eq:lower-norm-spatial}
		\end{align}
		and their time derivatives satisfy
		\begin{align}
			\norm{\partial_t\bu}_{L^2(\Omega_h)}^2
			\leq C\norm{\partial_t\bv}_{L^2(\Omega_1)}^2+C(1+\cD)\cE_1,\quad
			\norm{\partial_t\bv}_{L^2(\Omega_1)}^2
			\leq C\norm{\partial_t\bu}_{L^2(\Omega_h)}^2+C(1+\cD)\cE_1.
			\label{eq:lower-norm-time}
		\end{align}
		For every $\delta>0$, the two transport terms obey
		\begin{align}
			&\norm{h(\partial_tB_h^{-\top})\bv}_{L^2(\Omega_1)}^2
			+\norm{[(\widehat\bu-\partial_t\chi_h)\cdot(B_h\nabla)]\widehat\bu}_{L^2(\Omega_1)}^2
			\leq\delta\norm{\bv}_{H^2(\Omega_1)}^2+C_\delta(1+\cD)\cE_1.
			\label{eq:lower-nonlinear}
		\end{align}
	\end{lemma}
	\begin{proof}
		The mechanical energy bounds $h$ in $H_\sharp^2(0,L)$, $\partial_th$ in $L_\sharp^2(0,L)$, and $\bv$ in $L^2(\Omega_1)$. Together with $h\geq \min h$, it controls $h$, $h^{-1}$, and $\partial_xh$ in $L_\sharp^\infty(0,L)$. The matrices $A_h$ and $B_h^{-1}B_h^{-\top}$ are therefore uniformly elliptic, and their entries are uniformly bounded in $L^\infty(0,1;H_\sharp^1(0,L))$.
		
		We shall repeatedly use the one-dimensional inequalities with values in $L^2(0,1)$,
		\begin{align}
			\norm{f}_{L_\sharp^\infty(0,L;L^2(0,1))}^2
			&\leq C\norm{f}_{L^2(\Omega_1)}\norm{f}_{L^2(0,1;H_\sharp^1(0,L))},\nonumber\\
			\norm{f}_{L_\sharp^4(0,L;L^2(0,1))}^2
			&\leq C\norm{f}_{L^2(\Omega_1)}^{3/2}\norm{f}_{L^2(0,1;H_\sharp^1(0,L))}^{1/2}.
			\label{eq:lower-anisotropic}
		\end{align}
		Here $\norm{f}_{L^2(0,1;H_\sharp^s(0,L))}^2=\int_0^1\norm{f(\cdot,z)}_{H_\sharp^s(0,L)}^2\dz$. Multiplication by $B_h^\top$ and $B_h^{-\top}$ is bounded on $H^1(\Omega_1)$ with the stated constants. The only coefficient derivative requiring an estimate is a bounded factor times $\partial_x^2h$, and
		\begin{align*}
			\norm{\partial_x^2h\,\mathbf z}_{L^2(\Omega_1)}
			\leq\norm{\partial_x^2h}_{L_\sharp^2(0,L)}
			\norm{\mathbf z}_{L_\sharp^\infty(0,L;L^2(0,1))}
			\leq C\norm{\mathbf z}_{H^1(\Omega_1)}.
		\end{align*}
		Combining the coordinate change with the fixed bottom trace, we obtain \eqref{eq:lower-energy-products}.
		
		For a field with zero bottom trace, ellipticity, Poincare's inequality, and the same multiplication bounds imply \eqref{eq:lower-uniform-coercivity}. Thus the coercivity constant also depends only on the energy and separation.
		
		Differentiating the Piola transformation twice yields \eqref{eq:lower-norm-spatial}. Indeed, the scalar composition with $\chi_h$ and its inverse is bounded on $H^2$ under an $H_\sharp^2(0,L)$ bound for $h$ and a positive gap, cf.~\cite[Proposition~2.1]{SchwarzacherSu2025}. The additional highest derivative in the Piola multiplication satisfies
		\begin{align*}
			\norm{\partial_x^3h\,\widehat\bu}_{L^2(\Omega_1)}^2
			\leq\norm{\partial_x^3h}_{L_\sharp^2(0,L)}^2
			\norm{\widehat\bu}_{L_\sharp^\infty(0,L;L^2(0,1))}^2
			\leq C(1+\cD)\cE_1.
		\end{align*}
		Products with $\partial_x^2h$ multiplying a first fluid derivative are controlled by \eqref{eq:lower-anisotropic} and Young's inequality. Products with $(\partial_x^2h)^2$ multiplying an undifferentiated fluid factor use
		$\norm{\partial_x^2h}_{L_\sharp^4(0,L)}^4\leq C(1+\norm{\partial_x^3h}_{L_\sharp^2(0,L)})$. These observations prove both directions of \eqref{eq:lower-norm-spatial}.
		
		The exact temporal conversion is
		\begin{align}
			(\partial_t\bu)\circ\chi_h
			=B_h^{-\top}\partial_t\bv+(\partial_tB_h^{-\top})\bv
			-\frac{z\partial_th}{h}\partial_z(B_h^{-\top}\bv).
			\label{eq:lower-eulerian-}
		\end{align}
		The matrix $\partial_tB_h^{-\top}$ contains bounded factors multiplying $\partial_th$ and $\partial_t\partial_xh$. Estimating the latter coefficient in $L_\sharp^2(0,L)$ against the column supremum of $\bv$ and using the one-dimensional interpolation inequality for $\partial_th$ in $L_\sharp^\infty(0,L)$, we derive \eqref{eq:lower-norm-time}.
		
		The same estimates control the geometric transport terms. For convection, we invoke the two-dimensional interpolation inequality to obtain
		\begin{align*}
			\norm{\bu\cdot\nabla\bu}_{L^2(\Omega_h)}^2
			\leq C\norm\bu_{L^2(\Omega_h)}\norm\bu_{H^1(\Omega_h)}^2\norm\bu_{H^2(\Omega_h)}.
		\end{align*}
		Using \eqref{eq:lower-norm-spatial} and $\cD\leq C\cE_1$, we obtain \eqref{eq:lower-nonlinear} for every $\delta>0$.
		
		This completes the proof. 
	\end{proof}
	
	\subsection{The tangential estimate} 
	We begin with a tangential estimate, using the fact that horizontal differentiation preserves the coupling at the top boundary.
	\begin{lemma}
		Under the standing assumptions of this section, there is a constant $c>0$ such that, for every $\delta>0$ and $0\leq t\leq T$,
		\begin{align}
			&1+\rho_f\int_{\Omega_1}h(t)|B_{h(t)}^{-\top}\partial_x\bv(t)|^2\,\mathrm dx\,\mathrm dz
			+\rho_s\norm{\partial_t\partial_xh(t)}_{L_\sharp^2(0,L)}^2\nonumber\\
			&\quad+\beta\norm{\partial_x^2h(t)}_{L_\sharp^2(0,L)}^2
			+\alpha\norm{\partial_x^3h(t)}_{L_\sharp^2(0,L)}^2
			+c\int_0^t\norm{\partial_x\bv}_{H^1(\Omega_1)}^2\,\mathrm ds\nonumber\\
			&\leq1+\rho_f\int_{\Omega_1}h_0|B_{h_0}^{-\top}\partial_x\bv_0|^2\,\mathrm dx\,\mathrm dz
			+\rho_s\norm{\partial_xh_1}_{L_\sharp^2(0,L)}^2\nonumber\\
			&\quad+\beta\norm{\partial_x^2h_0}_{L_\sharp^2(0,L)}^2
			+\alpha\norm{\partial_x^3h_0}_{L_\sharp^2(0,L)}^2\nonumber\\
			&\quad+\delta\int_0^t\left(\norm{\partial_t\bv}_{L^2(\Omega_1)}^2
			+\norm{\bv}_{H^2(\Omega_1)}^2\right)\,\mathrm ds
			+C_\delta\int_0^t(1+\cD)\cE_1\,\mathrm ds.
			\label{eq:lower-horizontal-bound}
		\end{align}
		Moreover, for almost every $t\in(0,T)$,
		\begin{align}
			\norm{\partial_th}_{H_\sharp^{3/2}(0,L)}^2
			\leq C\norm{\partial_x\bv}_{H^1(\Omega_1)}^2.
			\label{eq:lower-velocity-trace}
		\end{align}
		All constants are independent of an assumed upper bound for $\cE_1$.
	\end{lemma}
	\begin{proof}
		We take $(\bphi,\zeta)=(-\partial_x^2\bv,-\partial_t\partial_x^2h)$ in \eqref{eq:lower-coupled}, since $\partial_x$ is the tangential direction and does not change the coupling of kinetic condition. Integration by parts in $x$ yields
		\begin{align}
			&\frac12\frac{\mathrm d}{\mathrm dt}\left[
			\int_0^L\int_0^1\rho_fh|B_h^{-\top}\partial_x\bv|^2\dz\,\mathrm dx
			+\rho_s\norm{\partial_t\partial_xh}_{L_\sharp^2(0,L)}^2
			+\beta\norm{\partial_x^2h}_{L_\sharp^2(0,L)}^2+\alpha\norm{\partial_x^3h}_{L_\sharp^2(0,L)}^2\right]\nonumber\\
			&\quad+\mu\sum_{i,j}\int_0^L\int_0^1(A_h)_{ij}
			\partial_i(B_h^{-\top}\partial_x\bv)\cdot\partial_j(B_h^{-\top}\partial_x\bv)\dz\,\mathrm dx\nonumber\\
			&=\frac12\int_0^L\int_0^1(\rho_f\partial_t(hB_h^{-1}B_h^{-\top}))\partial_x\bv\cdot\partial_x\bv\dz\,\mathrm dx
			-\int_0^L\int_0^1(\rho_f\partial_x(hB_h^{-1}B_h^{-\top}))\partial_t\bv\cdot\partial_x\bv\dz\,\mathrm dx\nonumber\\
			&\quad-\mu\sum_{i,j}\int_0^L\int_0^1(\partial_xA_h)_{ij}
			\partial_i\widehat{\bu}\cdot\partial_j(B_h^{-\top}\partial_x\bv)\dz\,\mathrm dx\nonumber\\
			&\quad-\mu\sum_{i,j}\int_0^L\int_0^1(A_h)_{ij}
			\partial_i((\partial_xB_h^{-\top})\bv)\cdot\partial_j(B_h^{-\top}\partial_x\bv)\dz\,\mathrm dx\nonumber\\
			&\quad-\mu\sum_{i,j}\int_0^L\int_0^1(A_h)_{ij}
			\partial_i\widehat{\bu}\cdot\partial_j((\partial_xB_h^{-\top})\partial_x\bv)\dz\,\mathrm dx 
			+\rho_f\int_0^L\int_0^1h(\partial_tB_h^{-\top})\bv\cdot B_h^{-\top}\partial_x^2\bv\dz\,\mathrm dx\nonumber\\
			&\quad
			+\rho_f\int_0^L\int_0^1
			\big[(\widehat\bu-\partial_t\chi_h)\cdot(B_h\nabla)\big]\widehat\bu
			\cdot B_h^{-\top}\partial_x^2\bv\dz\,\mathrm dx
			\eqqcolon \sum_{i=1}^7 I_i.
			\label{eq:lower-horizontal-identity}
		\end{align}
		By \eqref{eq:lower-uniform-coercivity}, we have
		\begin{align*}
			\left|\mu\sum_{i,j}\int_0^L\int_0^1(A_h)_{ij}
			\partial_i(B_h^{-\top}\partial_x\bv)\cdot\partial_j(B_h^{-\top}\partial_x\bv)\dz\,\mathrm dx\right|
			\geq c\norm{\partial_x\bv}_{H^1(\Omega_1)}^2,
		\end{align*}
		where $c > 0$ depends only on the mechanical energy and the positive gap.
		
		For the right-hand side, we estimate it term by term. First, the time derivative of $hB_h^{-1}B_h^{-\top}$ is a sum of bounded coefficients multiplying $\partial_th$ and $\partial_t\partial_xh$. Since $\norm{\partial_x\bv}_{L^2(\Omega_1)}\leq C(1+\cD)^{1/2}$, for every $\eta>0$, it holds
		\begin{align*}
			|I_1|
			\leq C\norm{\partial_th}_{H_\sharp^1(0,L)}
			\norm{\partial_x\bv}_{L^2(\Omega_1)}^{3/2}
			\norm{\partial_x\bv}_{H^1(\Omega_1)}^{1/2}
			&\leq\eta\norm{\partial_x\bv}_{H^1(\Omega_1)}^2
			+C_\eta(1+\cD)\norm{\partial_th}_{H_\sharp^1(0,L)}^{4/3}
			\\
			&\leq\eta\norm{\partial_x\bv}_{H^1(\Omega_1)}^2+C_\eta(1+\cD)\cE_1.
		\end{align*}
		Similarly, we know the $x$-derivative of $hB_h^{-1}B_h^{-\top}$ is bounded in $L_\sharp^2(0,L;L^\infty(0,1))$. Consequently,
		\begin{align*}
			|I_2|\leq C\norm{\partial_t\bv}_{L^2(\Omega_1)}(1+\cD)^{1/4}
			\norm{\partial_x\bv}_{H^1(\Omega_1)}^{1/2}
			\leq\eta\norm{\partial_t\bv}_{L^2(\Omega_1)}^2
			+\eta\norm{\partial_x\bv}_{H^1(\Omega_1)}^2+C_\eta(1+\cD).
		\end{align*}
		
		We now estimate the three viscous remainders in
		\eqref{eq:lower-horizontal-identity}. For the first one, applying Cauchy--Schwarz inequality in $z$ and H\"older's inequality in $x$ yields
		\begin{align*}
			|I_3|
			&\leq
			C\norm{\partial_xA_h}_{L_\sharp^2(0,L;L^\infty(0,1))}
			\norm{\nabla\widehat\bu}_{L_\sharp^\infty(0,L;L^2(0,1))}
			\norm{\nabla(B_h^{-\top}\partial_x\mathbf v)}_{L^2(\Omega_1)}\\
			&\leq
			C(1+\cD)^{1/4}
			\norm{\partial_x\mathbf v}_{H^1(\Omega_1)}^{3/2}
			+C(1+\cD)^{1/2}
			\left(1+\norm{\partial_x^3h}_{L_\sharp^2(0,L)}\right)^{1/2}
			\norm{\partial_x\mathbf v}_{H^1(\Omega_1)}\\
			&\leq
			\eta\norm{\partial_x\mathbf v}_{H^1(\Omega_1)}^2
			+C_\eta(1+\cD)\cE_1.
		\end{align*}		
		By the boundedness of $A_h$ and
		\eqref{eq:lower-matrix-product}, we have
		\begin{align*}
			|I_4|&\leq
			C\norm{\nabla((\partial_xB_h^{-\top})\mathbf v)}_{L^2(\Omega_1)}
			\norm{\partial_x\mathbf v}_{H^1(\Omega_1)}\\
			&\leq
			C\left(1+\norm{\partial_x^3h}_{L_\sharp^2(0,L)}\right)
			(1+\cD)^{1/2}
			\norm{\partial_x\mathbf v}_{H^1(\Omega_1)}
			+C(1+\cD)^{1/4}
			\norm{\partial_x\mathbf v}_{H^1(\Omega_1)}^{3/2}\\
			&\leq
			\eta\norm{\partial_x\mathbf v}_{H^1(\Omega_1)}^2
			+C_\eta(1+\cD)\cE_1.
		\end{align*}
		Since
		\begin{align*}
			\partial_j((\partial_xB_h^{-\top})\partial_x\mathbf v)
			=(\partial_xB_h^{-\top})\partial_j\partial_x\mathbf v
			+(\partial_j\partial_xB_h^{-\top})\partial_x\mathbf v,
		\end{align*}
		one obtains from \eqref{eq:geometric-bound} that
		\begin{align*}
			|I_5|&\leq
			C\int_{\Omega_1}
			(1+|\partial_x^2h|)|\nabla\widehat\bu|
			\big(
			|\nabla\partial_x\mathbf v|+|\partial_x\mathbf v|
			\big)\dx\dz
			+C\int_{\Omega_1}
			|\partial_x^3h|\,|\nabla\widehat\bu|\,
			|\partial_x\mathbf v|\dx\dz.
		\end{align*}
		The first integral is estimated exactly as $I_3$, using \eqref{eq:lower-gradient-column}.
		For the integral containing $\partial_x^3h$, we instead derive
		\begin{align*}
			\int_{\Omega_1}
			|\partial_x^3h|\,|\nabla\widehat\bu|\,
			|\partial_x\mathbf v|\dx\dz
			&\leq
			\int_0^L|\partial_x^3h|
			\left(\int_0^1|\nabla\widehat\bu|^2\dz\right)^{1/2}
			\left(\int_0^1|\partial_x\mathbf v|^2\dz\right)^{1/2}
			\dx\\
			&\quad\leq
			\norm{\partial_x^3h}_{L_\sharp^2(0,L)}
			\norm{\nabla\widehat\bu}_{L^2(\Omega_1)}
			\norm{\partial_x\mathbf v}_{L_\sharp^\infty(0,L;L^2(0,1))}\\
			&\quad\leq
			C\norm{\partial_x^3h}_{L_\sharp^2(0,L)}
			(1+\cD)^{3/4}
			\norm{\partial_x\mathbf v}_{H^1(\Omega_1)}^{1/2},
		\end{align*}
		which then by Young's inequality yields
		\begin{align*}
			|I_5| \leq
			\eta\norm{\partial_x\mathbf v}_{H^1(\Omega_1)}^2
			+C_\eta(1+\cD)\cE_1.
		\end{align*}
		
		Finally, invoking \eqref{eq:lower-nonlinear}, we find
		\begin{align*}
			|I_6| + |I_7|
			& \leq C \Big(\norm{h(\partial_tB_h^{-\top})\bv}_{L^2(\Omega_1)}
			+\norm{[(\widehat\bu-\partial_t\chi_h)\cdot(B_h\nabla)]\widehat\bu}_{L^2(\Omega_1)}\Big) \norm{B_h^{-\top} \partial_x^2 \bv}_{L^2(\Omega_1)} \\
			& \leq \delta\norm{\bv}_{H^2(\Omega_1)}^2
			+ \eta\norm{\partial_x\bv}_{H^1(\Omega_1)}^2
			+C_{\delta,\eta}(1+\cD)\cE_1.
		\end{align*}
		
		Choosing $\eta > 0$ sufficiently small such that $12 \eta \leq c$ and integrating \eqref{eq:lower-horizontal-identity} proves \eqref{eq:lower-horizontal-bound}.
		The trace condition and the zero mean of $\partial_th$ show \eqref{eq:lower-velocity-trace}.
	\end{proof}
	
	\subsection{Fractional-order estimates for the height}
	We next derive a fractional-order tangential estimate for the height function, which would be essential in the acceleration test later on. Since the test function will not correspond to the beam velocity, one requires an suitable construction of the pair fluid test function to be admissible. To this end, we seek a solenoidal lifting from the boundary trace into the bulk domain.
	
	\begin{lemma}\label{lem:lower-lift}
		For a smooth mean-zero periodic function $\xi$, let $\mathscr{E}(\xi)$ be the stationary Stokes velocity on $\Omega_1$ determined by
		\begin{align*}
			&\Div\mathscr{E}(\xi)=0,\qquad
			\mathscr{E}(\xi)|_{z=0}=0,\qquad \mathscr{E}(\xi)|_{z=1}=\xi\be_2,\\
			&\int_0^L\int_0^1\nabla\mathscr{E}(\xi):\nabla\boldsymbol\psi\dz\,\mathrm dx=0
		\end{align*}
		for every $x$-periodic $\boldsymbol\psi\in H^1(\Omega_1)$ satisfying $\Div\boldsymbol\psi=0$ and $\boldsymbol\psi|_{z=0}=\boldsymbol\psi|_{z=1}=0$. The lifting is linear and independent of time, and the estimates for this boundary value problem are
		\begin{align*}
			\norm{\mathscr{E}(\xi)}_{L^2(\Omega_1)}\leq C\norm\xi_{H_\sharp^{-1/2}(0,L)},
			\quad
			\norm{\mathscr{E}(\xi)}_{H^1(\Omega_1)}\leq C\norm\xi_{H_\sharp^{1/2}(0,L)},\quad
			\norm{\mathscr{E}(\xi)}_{H^2(\Omega_1)}\leq C\norm\xi_{H_\sharp^{3/2}(0,L)}.
		\end{align*}
		Moreover, for $\mathbf f\in L^2(0,1;H_\sharp^1(0,L))$,
		\begin{align}
			\left|\int_0^L\int_0^1\mathbf f\cdot\mathscr{E}(\xi)\dz\,\mathrm dx\right|
			\leq C\norm{\mathbf f}_{L^2(0,1;H_\sharp^1(0,L))}\norm\xi_{H_\sharp^{-3/2}(0,L)}.
			\label{eq:lower-lift-adjoint}
		\end{align}
		The constants depend only on $L$.
	\end{lemma}
	\begin{proof}
		These bounds can be obtained directly from the horizontal Fourier modes. At a nonzero frequency $k\in(2\pi/L)\mathbb Z$, the vertical component $w_k=(\mathscr{E}(\xi)_k)_2$ solves
		\begin{align*}
			(\partial_z^2-k^2)^2w_k=0,\qquad
			w_k(0)=\partial_zw_k(0)=\partial_zw_k(1)=0,\qquad w_k(1)=\xi_k,
		\end{align*}
		and the horizontal component is $(\mathscr{E}(\xi)_k)_1=-(ik)^{-1}\partial_zw_k$. Solving this fourth-order equation with the basis $e^{|k|z}$, $e^{-|k|z}$, $ze^{|k|z}$, and $ze^{-|k|z}$, we integrate the resulting exponential terms and find
		\begin{align}
			\norm{\partial_z^j\mathscr{E}(\xi)_k}_{L^2(0,1)}
			\leq C_j(1+k^2)^{j/2-1/4}|\xi_k|,\qquad j=0,1,2.
			\label{eq:lower-lift-modes}
		\end{align}
		The zero mode vanishes. Multiplying \eqref{eq:lower-lift-modes} by the appropriate powers of $|k|$ and summing by Parseval's identity proves the three estimates. The constructed velocity satisfies the stated Stokes problem. Uniqueness follows by testing the difference of two solutions by itself. The first estimate extends the solution to negative-order boundary data by transposition, cf.~\cite[Appendix~A]{Raymond2007}.
		
		Finally, applying Cauchy--Schwarz to the Fourier series, we obtain
		\begin{align*}
			\left|\int_{\Omega_1}\mathbf f\cdot\mathscr{E}(\xi)\dx\dz\right|
			&\leq C\sum_{k\neq0}\norm{\mathbf f_k}_{L^2(0,1)}(1+k^2)^{-1/4}|\xi_k|\\
			&\leq C\left(\sum_k(1+k^2)\norm{\mathbf f_k}_{L^2(0,1)}^2\right)^{1/2}
			\left(\sum_{k\neq0}(1+k^2)^{-3/2}|\xi_k|^2\right)^{1/2},
		\end{align*}
		which is \eqref{eq:lower-lift-adjoint}. Thus the time derivative of the fluid test can be estimated by duality.
	\end{proof}
	
	\begin{lemma}\label{lem:lower-height}
		Under the standing assumptions of this section, for every $\varepsilon>0$ and $0\leq t\leq T$,
		\begin{align}
			\int_0^t\norm{h}_{H_\sharp^{7/2}(0,L)}^2\,\mathrm ds
			&\leq\varepsilon\left(\norm{h(t)}_{H_\sharp^3(0,L)}^2+\norm{\partial_th(t)}_{H_\sharp^1(0,L)}^2\right)
			+C_\varepsilon(1+\cE_1(0))\nonumber\\
			&\quad+C\int_0^t\norm{\partial_x\bv}_{H^1(\Omega_1)}^2\,\mathrm ds
			+C\int_0^t(1+\cD)\cE_1\,\mathrm ds.
			\label{eq:lower-height-bound}
		\end{align}
	\end{lemma}
	
	\begin{proof}
		Let $\Lambda=(1-\partial_x^2)^{1/2}$ be the tangential operator, defined by horizontal Fourier series with frequencies $2\pi k/L$, $k\in\mathbb Z$. We choose $\zeta=\Lambda^3(h-H)$ and use $\bphi = \mathscr{E} (\zeta)$ in \eqref{eq:lower-coupled}. Since $h-H$ has zero mean, the test pair is admissible by Lemma~\ref{lem:lower-lift}. Integrating the two acceleration terms by parts in time, we arrive at
		\begin{align}
			&\frac{\mathrm d}{\mathrm dt}\left[
			\rho_f\int_0^L\int_0^1h\,\widehat\bu\cdot B_h^{-\top}\bphi\dz\,\mathrm dx
			+\rho_s\int_0^L\partial_th\,\Lambda^3(h-H)\dx\right]\nonumber\\
			&\quad+\alpha\norm{\Lambda^{3/2}\partial_x^2h}_{L_\sharp^2(0,L)}^2
			+\beta\norm{\Lambda^{3/2}\partial_xh}_{L_\sharp^2(0,L)}^2\nonumber\\
			&=\rho_s\norm{\partial_th}_{H_\sharp^{3/2}(0,L)}^2
			+\rho_f\int_0^L\int_0^1h\,\widehat\bu\cdot B_h^{-\top}\partial_t\bphi\dz\,\mathrm dx
			+\rho_f\int_0^L\int_0^1\partial_t(hB_h^{-1}B_h^{-\top})\bv\cdot\bphi\dz\,\mathrm dx\nonumber\\
			&\quad-\mu\sum_{i,j=1}^2\int_0^L\int_0^1(A_h)_{ij}
			\partial_i\widehat\bu\cdot\partial_j(B_h^{-\top}\bphi)\dz\,\mathrm dx
			-\rho_f\int_0^L\int_0^1h(\partial_tB_h^{-\top})\bv\cdot B_h^{-\top}\bphi\dz\,\mathrm dx\nonumber\\
			&\quad-\rho_f\int_0^L\int_0^1
			\big[(\widehat\bu-\partial_t\chi_h)\cdot(B_h\nabla)\big]\widehat\bu
			\cdot B_h^{-\top}\bphi\dz\,\mathrm dx
			\eqqcolon\sum_{i=1}^6J_i.
			\label{eq:lower-height-identity}
		\end{align}
		Since $h-H$ has zero mean and $\alpha>0$, it holds
		\begin{align*}
			\alpha\norm{\Lambda^{3/2}\partial_x^2h}_{L_\sharp^2(0,L)}^2
			+\beta\norm{\Lambda^{3/2}\partial_xh}_{L_\sharp^2(0,L)}^2
			\geq c\norm{h-H}_{H_\sharp^{7/2}(0,L)}^2
		\end{align*}
		for a fixed $c>0$. From the lifting estimates, we also have
		\begin{align*}
			\norm{\bphi}_{L^2(\Omega_1)}\leq C\norm{h-H}_{H_\sharp^{5/2}(0,L)},\qquad
			\norm{\bphi}_{H^1(\Omega_1)}\leq C\norm{h-H}_{H_\sharp^{7/2}(0,L)}.
		\end{align*}
		
		We estimate the right-hand side term by term. By \eqref{eq:lower-velocity-trace},
		\begin{align*}
			|J_1|\leq C\norm{\partial_x\bv}_{H^1(\Omega_1)}^2.
		\end{align*}
		Since lifting is linear and independent of time, $\partial_t\bphi$ solves the same problem with boundary value $\Lambda^3\partial_th\,\be_2$. For the term containing $\partial_t\bphi$, we apply \eqref{eq:lower-lift-adjoint} and obtain
		\begin{align*}
			|J_2|
			&\leq C\norm{hB_h^{-1}B_h^{-\top}\bv}_{L^2(0,1;H_\sharp^1(0,L))}
			\norm{\partial_th}_{H_\sharp^{3/2}(0,L)}\\
			&\leq C(1+\cD)^{1/2}\norm{\partial_th}_{H_\sharp^{3/2}(0,L)}
			\leq C\norm{\partial_x\bv}_{H^1(\Omega_1)}^2+C(1+\cD).
		\end{align*}
		The multiplication estimate in the second step uses only $\partial_x^2h\in L_\sharp^2(0,L)$ and the column supremum for $\bv$.
		
		For the coefficient-time term, the entries of $\partial_t(hB_h^{-1}B_h^{-\top})$ contain bounded factors multiplying $\partial_th$ and $\partial_t\partial_xh$. Consequently,
		\begin{align*}
			\norm{\partial_t(hB_h^{-1}B_h^{-\top})\bv}_{L^2(\Omega_1)}
			&\leq C\norm{\partial_th}_{L_\sharp^\infty(0,L)}\norm{\bv}_{L^2(\Omega_1)}
			+C\norm{\partial_t\partial_xh}_{L_\sharp^2(0,L)}
			\norm{\bv}_{L_\sharp^\infty(0,L;L^2(0,1))}\\
			&\leq C(1+\cD)^{1/2}\cE_1^{1/2}.
		\end{align*}
		Using the $L^2(\Omega_1)$ bound for the lift and Young's inequality, for every $\eta>0$ we obtain
		\begin{align*}
			|J_3|\leq C(1+\cD)^{1/2}\cE_1^{1/2}\norm{\bphi}_{L^2(\Omega_1)}
			\leq\eta\norm{h-H}_{H_\sharp^{7/2}(0,L)}^2+C_\eta(1+\cD)\cE_1.
		\end{align*}
		For viscosity, boundedness of multiplication by $B_h^{-\top}$ on $H^1$ implies
		\begin{align*}
			|J_4|&\leq C\norm{\nabla\widehat\bu}_{L^2(\Omega_1)}\norm{\bphi}_{H^1(\Omega_1)}
			\leq C(1+\cD)^{1/2}\norm{\bphi}_{H^1(\Omega_1)}\\
			&\leq\eta\norm{h-H}_{H_\sharp^{7/2}(0,L)}^2+C_\eta(1+\cD).
		\end{align*}
		Arguing as for $J_3$, we find
		\begin{align*}
			|J_5|\leq\eta\norm{h-H}_{H_\sharp^{7/2}(0,L)}^2+C_\eta(1+\cD)\cE_1.
		\end{align*}
		
		The transport term requires only the $H^1(\Omega_1)$ norm of the test. Since $B_h^\top(\widehat\bu-\partial_t\chi_h)=\bv-z\partial_th\be_2$ has zero normal component at $z=0,1$ and divergence $-\partial_th$, integration by parts yields
		\begin{align*}
			&\int_{\Omega_1}\big[(\widehat\bu-\partial_t\chi_h)\cdot(B_h\nabla)\big]\widehat\bu
			\cdot B_h^{-\top}\bphi\,\mathrm dx\,\mathrm dz\\
			&\quad=-\int_{\Omega_1}\widehat\bu\cdot
			\big[(\widehat\bu-\partial_t\chi_h)\cdot(B_h\nabla)\big](B_h^{-\top}\bphi)\,\mathrm dx\,\mathrm dz
			+\int_{\Omega_1}\partial_th\,\widehat\bu\cdot B_h^{-\top}\bphi\,\mathrm dx\,\mathrm dz.
		\end{align*}
		The two-dimensional inequality $\norm{\bv}_{L^4(\Omega_1)}^2\leq C\norm{\bv}_{L^2(\Omega_1)}\norm{\bv}_{H^1(\Omega_1)}$ controls the quadratic fluid product by $C(1+\cD)^{1/2}$. Products with $\partial_th$ use its $L_\sharp^\infty(0,L)$ norm, bounded by $C\cE_1^{1/2}$. In particular,
		\begin{align*}
			|J_6|&\leq C\left(\norm{\bv}_{L^4(\Omega_1)}^2
			+\norm{\partial_th}_{L_\sharp^\infty(0,L)}\norm{\bv}_{L^2(\Omega_1)}\right)\norm{\bphi}_{H^1(\Omega_1)}\\
			&\leq C(1+\cD)^{1/2}\cE_1^{1/2}\norm{\bphi}_{H^1(\Omega_1)}
			\leq\eta\norm{h-H}_{H_\sharp^{7/2}(0,L)}^2+C_\eta(1+\cD)\cE_1.
		\end{align*}
		The estimates of $J_5$ and $J_6$ apply to every admissible test of the stated regularity. In particular,
		\begin{align}
			&\left|\rho_f\int_{\Omega_1}h(\partial_tB_h^{-\top})\bv\cdot B_h^{-\top}\bphi\,\mathrm dx\,\mathrm dz\right|\nonumber\\
			&\quad+\left|\rho_f\int_{\Omega_1}\big[(\widehat\bu-\partial_t\chi_h)\cdot(B_h\nabla)\big]\widehat\bu
			\cdot B_h^{-\top}\bphi\,\mathrm dx\,\mathrm dz\right|
			\leq C(1+\cD)^{1/2}\cE_1^{1/2}\norm{\bphi}_{H^1(\Omega_1)}.
			\label{eq:lower-weak-nonlinear}
		\end{align}
		
		Choose $\eta>0$ so that $4\eta\leq c/2$ and integrate \eqref{eq:lower-height-identity}. Since $\cE_1\geq1$, the preceding estimates yield
		\begin{align*}
			&\left[\rho_f\int_{\Omega_1}h\,\widehat\bu\cdot B_h^{-\top}\bphi\dx\dz
			+\rho_s\int_0^L\partial_th\,\Lambda^3(h-H)\dx\right]_0^t
			+\frac c2\int_0^t\norm{h-H}_{H_\sharp^{7/2}(0,L)}^2\,\mathrm ds\\
			&\quad\leq C\int_0^t\norm{\partial_x\bv}_{H^1(\Omega_1)}^2\,\mathrm ds
			+C\int_0^t(1+\cD)\cE_1\,\mathrm ds.
		\end{align*}
		It remains to estimate the fluid and beam pairings evaluated at times $0$ and $t$. By the mechanical energy bound and interpolation,
		\begin{align*}
			\left|\rho_s\int_0^L\partial_th\,\Lambda^3(h-H)\dx\right|
			&\leq C\norm{\partial_th}_{H_\sharp^1(0,L)}\norm{h-H}_{H_\sharp^2(0,L)}
			\leq C\norm{\partial_th}_{H_\sharp^1(0,L)},\\
			\left|\rho_f\int_{\Omega_1}h\,\widehat\bu\cdot B_h^{-\top}\bphi\,\mathrm dx\,\mathrm dz\right|
			&\leq C\norm{\bphi}_{L^2(\Omega_1)}
			\leq C\norm{h-H}_{H_\sharp^{5/2}(0,L)}
			\leq C\left(1+\norm{h}_{H_\sharp^3(0,L)}^{1/2}\right).
		\end{align*}
		Hence, for every $\varepsilon>0$, Young's inequality implies
		\begin{align*}
			\left|\rho_f\int_{\Omega_1}h\,\widehat\bu\cdot B_h^{-\top}\bphi\dx\dz
			+\rho_s\int_0^L\partial_th\,\Lambda^3(h-H)\dx\right|
			\leq 
			\varepsilon\left(\norm{h(t)}_{H_\sharp^3(0,L)}^2+\norm{\partial_th(t)}_{H_\sharp^1(0,L)}^2\right)+C_\varepsilon,
		\end{align*}
		and its value at $t=0$ is bounded by $C(1+\cE_1(0))$.
		This completes the proof of \eqref{eq:lower-height-bound}.
	\end{proof}
	
	\subsection{The beam acceleration estimate}
	In this section, we proceed with the acceleration estimate in negative tangential Sobolev spaces, which can be derived from the tangential estimates above and will be crucial in the fluid controls.
	
	To estimate the errors produced by tangential smoothing, we use the shifted Fourier multiplier $\Lambda_N=(N^2-\partial_x^2)^{1/2}$, where $N\geq1$ is a large parameter, by which one shall filter the higher-order perturbations due to commutator estimates. For an operator $T$ and a coefficient $g$, we denote the commutator by $[T,g]f= T(gf)-gTf$. We first record the following result for the commutator.
	\begin{lemma}\label{lem:lower-commutator}
		Let $N\geq1$ and $\Lambda_N=(N^2-\partial_x^2)^{1/2}$ on periodic functions. For a scalar periodic $g\in H_\sharp^1(0,L)$ and $f\in L_\sharp^2(0,L)$,
		\begin{align}
			\norm{[\Lambda_N^{-1},g]f}_{L_\sharp^2(0,L)}
			&\leq C N^{-3/2}\norm{\partial_xg}_{L_\sharp^2(0,L)}\norm{f}_{L_\sharp^2(0,L)},\nonumber\\
			\norm{\Lambda_N[\Lambda_N^{-1},g]f}_{L_\sharp^2(0,L)}
			&\leq C N^{-1/2}\norm{\partial_xg}_{L_\sharp^2(0,L)}\norm{f}_{L_\sharp^2(0,L)},\nonumber\\
			\norm{\Lambda_N^{1/2}[\Lambda_N^{-1},g]f}_{L_\sharp^2(0,L)}
			&\leq C N^{-1}\norm{\partial_xg}_{L_\sharp^2(0,L)}\norm{f}_{L_\sharp^2(0,L)}.
			\label{eq:lower-sharp-commutator}
		\end{align}
		For rougher coefficients $g\in H_\sharp^{1/2}(0,L)$,
		\begin{align}
			\norm{\Lambda_N^{1/2}[\Lambda_N^{-1},g]f}_{L_\sharp^2(0,L)}
			\leq C_N\norm{g}_{H_\sharp^{1/2}(0,L)}\norm{f}_{L_\sharp^2(0,L)}.
			\label{eq:lower-rough-commutator}
		\end{align}
		Moreover, if $g\in H_\sharp^1(0,L)$ is a symmetric matrix with $g\geq g_0 I$, where $g_0>0$, then, for every periodic vector field $f\in L_\sharp^2(0,L)$,
		\begin{align}
			\int_0^L(gf)\cdot\Lambda_N^{-1}f\,\mathrm dx
			\geq \left(g_0-CN^{-1/2}\norm{\partial_xg}_{L_\sharp^2(0,L)}\right)
			\norm{\Lambda_N^{-1/2}f}_{L_\sharp^2(0,L)}^2.
			\label{eq:lower-positive-mass}
		\end{align}
		Here the matrix norm in the error can be taken to be the sum of the component norms. The constants $C$ are independent of $N\geq1$. 
	\end{lemma}
	\begin{proof}
		Indeed, $|g(x)-g(y)|\leq \norm{\partial_xg}_{L_\sharp^2(0,L)}d(x,y)^{1/2}$, where $d$ denotes periodic distance. The periodized kernel of $\Lambda_N^{-1}$ has moment of order $1/2$ bounded by $CN^{-3/2}$ and its derivative has the corresponding moment bounded by $CN^{-1/2}$. These statements follow by scaling the Bessel kernel on the line and then periodizing. Moreover,
		\begin{align*}
			\partial_x[\Lambda_N^{-1},g]f
			=[\partial_x\Lambda_N^{-1},g]f-(\partial_xg)\Lambda_N^{-1}f,
			\qquad
			\norm{\Lambda_N^{-1}f}_{L_\sharp^\infty(0,L)}\leq C N^{-1/2}\norm{f}_{L_\sharp^2(0,L)}.
		\end{align*}
		The last inequality follows from the $L^2$ norm of the Bessel kernel. These facts prove the first two estimates in \eqref{eq:lower-sharp-commutator}, and Hilbert-space interpolation leads to the third one.
		
		To show \eqref{eq:lower-rough-commutator}, write the commutator as $\Lambda_N^{-1}(gf)-g\Lambda_N^{-1}f$. In one dimension $H^{1/2}\cdot L^2\subset H^{-1/2}$ and $H^{1/2}\cdot H^1\subset H^{1/2}$. The first product assertion also follows immediately by duality from $H^{1/2}\subset L^4$. Applying these two product estimates, we have \eqref{eq:lower-rough-commutator}. All coefficient families below are finite polynomials in $z$, so the same bounds on $\Omega_1$ follow from their slice estimates.
		
		We next establish the positivity needed for the acceleration terms and the principal part of the tangentially smoothed viscous form. With $\mathbf{w}=\Lambda_N^{-1/2}f$, the difference from $\int(g\mathbf{w})\cdot\mathbf{w}$ is
		\begin{align*}
			\int_0^L \mathbf{w}\cdot[\Lambda_N^{1/2},g]\Lambda_N^{-1/2}\mathbf{w}\,\mathrm dx.
		\end{align*}
		Interpolating between the $L^2$ and $L^\infty$ increment estimates, we obtain
		\begin{align*}
			\norm{g(\cdot-r)-g}_{L_\sharp^4(0,L)}
			\leq C d(r,0)^{3/4}\norm{\partial_xg}_{L_\sharp^2(0,L)}.
		\end{align*}
		The positive-order kernel for $\Lambda_N^{1/2}$ has moment of order $3/4$ bounded by $CN^{-1/4}$. One direct verification uses the periodic heat kernel and
		\begin{align*}
			\Lambda_N^{1/2}f
			=N^{1/2}f+
			\frac{1}{|\Gamma(-1/4)|}\int_0^\infty
			e^{-N^2r}r^{-5/4}(f-e^{r\partial_x^2}f)\,\mathrm dr.
		\end{align*}
		The periodic heat kernel has the required moment bound $Cr^{3/8}$, whose integral against $e^{-N^2r}r^{-5/4}$ is bounded by $CN^{-1/4}$. By Minkowski's and H\"older's inequalities,
		\begin{align*}
			\norm{[\Lambda_N^{1/2},g]w}_{L_\sharp^2(0,L)}
			\leq CN^{-1/4}\norm{\partial_xg}_{L_\sharp^2(0,L)}\norm{w}_{L_\sharp^4(0,L)}.
		\end{align*}
		Finally $H_\sharp^{1/4}(0,L)\subset L_\sharp^4(0,L)$ yields
		\begin{align*}
			\norm{\Lambda_N^{-1}f}_{L_\sharp^4(0,L)}
			\leq CN^{-1/4}\norm{\Lambda_N^{-1/2} f}_{L_\sharp^2(0,L)},
		\end{align*}
		which proves \eqref{eq:lower-positive-mass}.
	\end{proof}
	Now we state our tangential acceleration estimate.
	\begin{lemma}\label{lem:lower-acceleration}
		Under the standing assumptions, there exists $N_0\geq1$, depending only on the energy and geometric bounds and the physical parameters, such that for every $N\geq N_0$, every $\eta>0$, and every $t\in[0,T]$, one has, with $\Lambda_N=(N^2-\partial_x^2)^{1/2}$,
		\begin{align}
			&\int_0^t\left(\norm{\Lambda_N^{-1/2}\partial_t\bv}_{L^2(\Omega_1)}^2
			+\norm{\Lambda_N^{-1/2}\partial_t^2h}_{L_\sharp^2(0,L)}^2\right)\,\mathrm ds
			\nonumber\\
			&\quad\leq \eta\cE_1(t)+C_{\eta,N}(1+\cE_1(0))
			+C_N\int_0^t \norm{\partial_x\bv(s)}_{H^1(\Omega_1)}^2\,\mathrm ds\nonumber\\
			&\qquad+C N^{-2}\int_0^t\norm{\bv}_{H^2(\Omega_1)}^2\,\mathrm ds
			+C_N\int_0^t (1+\cD(s))\cE_1(s)\,\mathrm ds.
			\label{eq:lower-filtered-bound}
		\end{align}
	\end{lemma}
	
	\begin{proof}
		We test \eqref{eq:lower-coupled} with the admissible pair $(\bphi,\zeta)=(\Lambda_N^{-1}\partial_t\bv,\Lambda_N^{-1}\partial_t^2h)$. The operator $\Lambda_N^{-1}$ commutes with spatial derivatives and boundary traces, so the pair is solenoidal and preserves the coupling condition. Integration by parts in time in the elastic terms and symmetry of the viscous form yield
		\begin{align}
			&\frac{\mathrm d}{\mathrm dt}\Bigg[
			\frac\mu2\sum_{i,j=1}^2\int_{\Omega_1}(A_h)_{ij}
			\partial_i(B_h^{-\top}\bv)\cdot\partial_j(B_h^{-\top}\Lambda_N^{-1}\bv)\dx\dz\nonumber\\
			&\hspace{22mm}+\int_0^L\left(\alpha\partial_x^2h\,\Lambda_N^{-1}\partial_t\partial_x^2h
			+\beta\partial_xh\,\Lambda_N^{-1}\partial_t\partial_xh\right)\dx\Bigg]\nonumber\\
			&\quad+\rho_f\int_{\Omega_1}(hB_h^{-1}B_h^{-\top}\partial_t\bv)\cdot\Lambda_N^{-1}\partial_t\bv\dx\dz
			+\rho_s\norm{\Lambda_N^{-1/2}\partial_t^2h}_{L_\sharp^2(0,L)}^2\nonumber\\
			&=\alpha\norm{\Lambda_N^{-1/2}\partial_t\partial_x^2h}_{L_\sharp^2(0,L)}^2
			+\beta\norm{\Lambda_N^{-1/2}\partial_t\partial_xh}_{L_\sharp^2(0,L)}^2\nonumber\\
			&\quad+\frac\mu2\sum_{i,j=1}^2\int_{\Omega_1}\Big[
			(\partial_tA_h)_{ij}\partial_i(B_h^{-\top}\bv)\cdot\partial_j(B_h^{-\top}\Lambda_N^{-1}\bv)\nonumber\\
			&\hspace{22mm}+(A_h)_{ij}\partial_i((\partial_tB_h^{-\top})\bv)\cdot\partial_j(B_h^{-\top}\Lambda_N^{-1}\bv)
			+(A_h)_{ij}\partial_i(B_h^{-\top}\bv)\cdot\partial_j((\partial_tB_h^{-\top})\Lambda_N^{-1}\bv)
			\Big]\dx\dz\nonumber\\
			&\quad-\frac\mu2\sum_{i,j=1}^2\int_{\Omega_1}(A_h)_{ij}\Big[
			\partial_i(B_h^{-\top}\bv)\cdot\partial_j(B_h^{-\top}\Lambda_N^{-1}\partial_t\bv)
			-\partial_i(B_h^{-\top}\Lambda_N^{-1}\bv)\cdot\partial_j(B_h^{-\top}\partial_t\bv)
			\Big]\dx\dz\nonumber\\
			&\quad-\rho_f\int_{\Omega_1}h(\partial_tB_h^{-\top})\bv\cdot B_h^{-\top}\Lambda_N^{-1}\partial_t\bv\dx\dz\nonumber\\
			&\quad-\rho_f\int_{\Omega_1}\big[(\widehat\bu-\partial_t\chi_h)\cdot(B_h\nabla)\big]\widehat\bu
			\cdot B_h^{-\top}\Lambda_N^{-1}\partial_t\bv\dx\dz
			\eqqcolon\sum_{i=1}^5 K_i.
			\label{eq:lower-acceleration-identity}
		\end{align}
		Here $K_1$ is the sum of the two structural norms.
		The coefficients $\rho_fhB_h^{-1}B_h^{-\top}$ and the full principal block matrix $\big((A_h)_{ij}B_h^{-1}B_h^{-\top}\big)_{i,j=1}^2$ are uniformly positive and bounded in $L^\infty(0,1;H_\sharp^1(0,L))$ using only the energy and gap. By Lemma~\ref{lem:lower-commutator}, we may choose $N\geq N_0$ depending only on these quantities so that
		\begin{align*}
			&\rho_f\int_{\Omega_1}(hB_h^{-1}B_h^{-\top}\partial_t\bv)\cdot\Lambda_N^{-1}\partial_t\bv\dx\dz
			+\rho_s\norm{\Lambda_N^{-1/2}\partial_t^2h}_{L_\sharp^2(0,L)}^2\\
			&\qquad\geq c\left(\norm{\Lambda_N^{-1/2}\partial_t\bv}_{L^2(\Omega_1)}^2
			+\norm{\Lambda_N^{-1/2}\partial_t^2h}_{L_\sharp^2(0,L)}^2\right).
		\end{align*}
		The same choice ensures positivity of the principal part of the viscous quadratic form inside the time derivative in \eqref{eq:lower-acceleration-identity}.
		
		We now estimate the right-hand side term by term. From the Fourier symbol of $\Lambda_N^{-1/2}$ and \eqref{eq:lower-velocity-trace}, we obtain, uniformly for $N\geq1$,
		\begin{align*}
			0\leq K_1\leq C\norm{\partial_th}_{H_\sharp^{3/2}(0,L)}^2
			\leq C\norm{\partial_x\bv}_{H^1(\Omega_1)}^2.
		\end{align*}
		
		To estimate $K_2$ and $K_3$, we first expand the viscous form. For an arbitrary test $\bphi$, its expansion is
		\begin{align}
			&\mu\sum_{i,j=1}^2\int_{\Omega_1}(A_h)_{ij}
			\partial_i(B_h^{-\top}\bv)\cdot\partial_j(B_h^{-\top}\bphi)\,\mathrm dx\,\mathrm dz\nonumber\\
			&=\mu\sum_{i,j=1}^2\int_{\Omega_1}\Big[
			((A_h)_{ij}B_h^{-1}B_h^{-\top}\partial_i\bv)\cdot\partial_j\bphi
			+((A_h)_{ij}(\partial_jB_h^{-1})B_h^{-\top}\partial_i\bv)\cdot\bphi\nonumber\\
			&\hspace{27mm}+((A_h)_{ij}B_h^{-1}(\partial_iB_h^{-\top})\bv)\cdot\partial_j\bphi
			+((A_h)_{ij}(\partial_jB_h^{-1})(\partial_iB_h^{-\top})\bv)\cdot\bphi
			\Big]\,\mathrm dx\,\mathrm dz.
			\label{eq:lower-coefficient-expansion}
		\end{align}
		Here $\partial_1=\partial_x$ and $\partial_2=\partial_z$. We keep the four coefficient families separate, since only the principal coefficient needs a decaying power of $N$.
		
		The mechanical energy and the positive gap imply
		\begin{align}
			&\norm{\rho_fhB_h^{-1}B_h^{-\top}}_{L^\infty(0,1;H_\sharp^1(0,L))}
			+\sum_{i,j=1}^2\norm{(A_h)_{ij}B_h^{-1}B_h^{-\top}}_{L^\infty(0,1;H_\sharp^1(0,L))}\leq C,
			\\
			&\sum_{i,j=1}^2\Big(
			\norm{\partial_x((A_h)_{ij}B_h^{-1}B_h^{-\top})}_{L^\infty(0,1;H_\sharp^{1/2}(0,L))}
			+\norm{(A_h)_{ij}B_h^{-1}(\partial_iB_h^{-\top})}_{L^\infty(0,1;H_\sharp^{1/2}(0,L))}\Big)\nonumber\\
			&\qquad\leq C\left(1+\norm{\partial_x^3h}_{L_\sharp^2(0,L)}^{1/2}\right).
		\end{align}
		The transposed mixed family has the same bound. A horizontal derivative of a mixed coefficient and an undifferentiated zeroth-order coefficient satisfy
		\begin{align}
			&\sum_{i,j=1}^2\norm{\partial_x((A_h)_{ij}B_h^{-1}(\partial_iB_h^{-\top}))}_{L^\infty(0,1;H_\sharp^{1/2}(0,L))}
			\leq C\left(1+\norm{\partial_x^3h}_{L_\sharp^2(0,L)}+\norm{h}_{H_\sharp^{7/2}(0,L)}\right),\nonumber\\
			&\sum_{i,j=1}^2\norm{(A_h)_{ij}(\partial_jB_h^{-1})(\partial_iB_h^{-\top})}_{L^\infty(0,1;H_\sharp^{1/2}(0,L))}
			\leq C\left(1+\norm{\partial_x^3h}_{L_\sharp^2(0,L)}\right).
		\end{align}
		The same bounds hold after the $z$ derivatives used below. Indeed, these coefficients are finite polynomials in $z$. Their scalar coefficient functions depend smoothly on $h,h^{-1},\partial_xh$, with at most a linear factor of $\partial_x^2h$ in the mixed terms and a quadratic factor in the zeroth-order terms. Smooth functions of $h,h^{-1},\partial_xh$ have uniformly bounded $H_\sharp^1(0,L)$ norms. The only quadratic structural product needing explanation is controlled by
		\begin{align*}
			\norm{(\partial_x^2h)^2}_{H_\sharp^{1/2}(0,L)}
			\leq C\norm{\partial_x^2h}_{L_\sharp^2(0,L)}
			\norm{\partial_x^2h}_{H_\sharp^1(0,L)}.
		\end{align*}
		This follows from the one-dimensional fractional product estimate and interpolation, or from $\norm{f}_{L_\sharp^4(0,L)}\leq C\norm{f}_{L_\sharp^2(0,L)}^{3/4}\norm{f}_{H_\sharp^1(0,L)}^{1/4}$ and $\norm{(1-\partial_x^2)^{1/4}f}_{L_\sharp^4(0,L)}\leq C\norm{f}_{L_\sharp^2(0,L)}^{1/4}\norm{f}_{H_\sharp^1(0,L)}^{3/4}$.
		
		We now estimate $K_2$, the coefficient-time integral in \eqref{eq:lower-acceleration-identity}. In the principal coefficients a time derivative produces $\partial_th$ and $\partial_t\partial_xh$. The mixed coefficients additionally produce $\partial_t\partial_x^2h$, and the zeroth-order coefficients produce $\partial_x^2h\,\partial_t\partial_x^2h$. Integrate the latter two derivatives once in $x$. The resulting products are obtained directly from \eqref{eq:lower-coefficient-expansion}, with at most the following structural factors:
		\begin{align*}
			&\partial_th,\quad \partial_t\partial_xh,\quad
			\partial_th\,\partial_x^2h,\quad
			\partial_t\partial_xh\,\partial_x^2h,\quad
			\partial_th\,(\partial_x^2h)^2,\quad
			\partial_t\partial_xh\,(\partial_x^2h)^2,\quad
			\partial_t\partial_xh\,\partial_x^3h.
		\end{align*}
		They multiply two fluid factors, one of which is acted on by $\Lambda_N^{-1}$. The only second fluid derivatives generated in this integration by parts have a horizontal derivative and are bounded by $\norm{\partial_x\bv}_{H^1(\Omega_1)}$. The estimates use
		\begin{align*}
			&\norm{\partial_th}_{L_\sharp^\infty(0,L)}\leq C \cE_1^{1/4},\qquad
			\norm{\partial_x^2h}_{L_\sharp^\infty(0,L)}\leq C(1+\norm{\partial_x^3h}_{L_\sharp^2(0,L)}^{1/2}),\\
			&\norm{\Lambda_N^{-1}f}_{L_\sharp^\infty(0,L;L^2(0,1))}\leq C_N\norm{f}_{L^2(\Omega_1)},\\
			&\norm{\nabla\bv}_{L_\sharp^\infty(0,L;L^2(0,1))}
			\leq C (1+\cD)^{1/4}((1+\cD)^{1/4}+\norm{\partial_x\bv}_{H^1(\Omega_1)}^{1/2}).
		\end{align*}
		Here the first inequality also uses the energy bound for $\partial_th$ in $L^2$. To make the highest products explicit, their absolute integrals over $\Omega_1$ satisfy
		\begin{align*}
			&\int_{\Omega_1}|\partial_t\partial_xh|\,|\Lambda_N^{-1}\bv|\,
			|\nabla\partial_x\bv|
			\leq C_N\cE_1^{1/2}\norm{\partial_x\bv}_{H^1(\Omega_1)}\leq\delta \norm{\partial_x\bv}_{H^1(\Omega_1)}^2+C_{\delta,N}(1+\cD)\cE_1,\\
			&\int_{\Omega_1}|\partial_t\partial_xh|\,|\partial_x^3h|\,
			|\bv|\,|\Lambda_N^{-1}\bv|
			\leq C_N (1+\cD)^{1/4}\cE_1\leq C_N(1+\cD)\cE_1,\\
			&\int_{\Omega_1}|\partial_t\partial_xh|\,|\partial_x^2h|^2\,
			|\bv|\,|\Lambda_N^{-1}\bv|
			\leq C_N(1+\cD)^{1/4}\cE_1^{3/4}\leq C_N(1+\cD)\cE_1,\\
			&\int_{\Omega_1}|\partial_t\partial_xh|\,|\partial_x^2h|\,
			|\nabla\bv|\,|\Lambda_N^{-1}\bv|
			\leq C_N\cE_1^{1/2}(1+\cD)^{1/4}\left((1+\cD)^{1/4}+\norm{\partial_x\bv}_{H^1(\Omega_1)}^{1/2}\right)\\
			&\phantom{\int_{\Omega_1}|\partial_t\partial_xh|\,|\partial_x^2h|\,
				|\nabla\bv|\,|\Lambda_N^{-1}\bv|}\leq\delta \norm{\partial_x\bv}_{H^1(\Omega_1)}^2+C_{\delta,N}(1+\cD)\cE_1.
		\end{align*}
		Each estimate uses Cauchy--Schwarz in $z$ before applying H\"older's inequality in $x$. The other products have fewer derivatives and satisfy the same bound. Thus, expanding $2K_2$ in \eqref{eq:lower-acceleration-identity}, we obtain
		\begin{align}
			&\mu\Bigg|\sum_{i,j=1}^2\int_{\Omega_1}\Big[
			(\partial_t((A_h)_{ij}B_h^{-1}B_h^{-\top})\partial_i\bv)\cdot\partial_j\Lambda_N^{-1}\bv\nonumber\\
			&\hspace{16mm}+(\partial_t((A_h)_{ij}(\partial_jB_h^{-1})B_h^{-\top})\partial_i\bv)\cdot\Lambda_N^{-1}\bv
			+(\partial_t((A_h)_{ij}B_h^{-1}(\partial_iB_h^{-\top}))\bv)\cdot\partial_j\Lambda_N^{-1}\bv\nonumber\\
			&\hspace{16mm}+(\partial_t((A_h)_{ij}(\partial_jB_h^{-1})(\partial_iB_h^{-\top}))\bv)\cdot\Lambda_N^{-1}\bv
			\Big]\,\mathrm dx\,\mathrm dz\Bigg|\nonumber\\
			&\qquad\leq\delta\norm{\partial_x\bv}_{H^1(\Omega_1)}^2+C_{\delta,N}(1+\cD)\cE_1.
		\end{align}
		Since $K_2$ is one half of this integral, we fix $\delta$, for example $\delta=1$, and obtain
		\begin{align*}
			|K_2|\leq C_N\norm{\partial_x\bv}_{H^1(\Omega_1)}^2+C_N(1+\cD)\cE_1.
		\end{align*}
		
		For $K_3$, we estimate the commutator difference in \eqref{eq:lower-acceleration-identity}. Expanding $-2K_3$ by \eqref{eq:lower-coefficient-expansion}, commuting the self-adjoint operator $\Lambda_N^{-1}$, and integrating by parts once, we obtain
		\begin{align}
			&\mu\sum_{i,j=1}^2\int_{\Omega_1}(A_h)_{ij}\Big[
			\partial_i(B_h^{-\top}\bv)\cdot\partial_j(B_h^{-\top}\Lambda_N^{-1}\partial_t\bv)
			-\partial_i(B_h^{-\top}\Lambda_N^{-1}\bv)\cdot\partial_j(B_h^{-\top}\partial_t\bv)
			\Big]\,\mathrm dx\,\mathrm dz\nonumber\\
			&=\mu\sum_{i,j=1}^2\int_{\Omega_1}\Big\{
			-\partial_j\big([\Lambda_N^{-1},(A_h)_{ij}B_h^{-1}B_h^{-\top}]\partial_i\bv\big)
			-\partial_j\big([\Lambda_N^{-1},(A_h)_{ij}B_h^{-1}(\partial_iB_h^{-\top})]\bv\big)\nonumber\\
			&\hspace{24mm}+[\Lambda_N^{-1},(A_h)_{ij}(\partial_jB_h^{-1})B_h^{-\top}]\partial_i\bv
			+[\Lambda_N^{-1},(A_h)_{ij}(\partial_jB_h^{-1})(\partial_iB_h^{-\top})]\bv
			\Big\}\cdot\partial_t\bv\,\mathrm dx\,\mathrm dz\nonumber\\
			&\quad+\mu\sum_{i=1}^2\int_0^L\Big(
			[\Lambda_N^{-1},(A_h)_{i2}B_h^{-1}B_h^{-\top}]\partial_i\bv
			+[\Lambda_N^{-1},(A_h)_{i2}B_h^{-1}(\partial_iB_h^{-\top})]\bv
			\Big)_2\big|_{z=1}\partial_t^2h\,\mathrm dx.
			\label{eq:lower-commutator-boundary}
		\end{align}
		The bottom contribution vanishes because $\partial_t\bv=0$ there. To proceed, we first expand
		\begin{align*}
			\partial_j\big([\Lambda_N^{-1},(A_h)_{ij}B_h^{-1}B_h^{-\top}]
			\partial_i\bv\big)=[\Lambda_N^{-1},(A_h)_{ij}B_h^{-1}B_h^{-\top}]
			\partial_j\partial_i\bv
			+[\Lambda_N^{-1},\partial_j((A_h)_{ij}B_h^{-1}B_h^{-\top})]
			\partial_i\bv.
		\end{align*}
		The same product rule applies to the differentiated mixed commutator. For the terms with two derivatives on $\bv$, \eqref{eq:lower-sharp-commutator} yields
		\begin{align*}
			\sum_{i,j=1}^2
			\norm{\Lambda_N^{1/2}
				[\Lambda_N^{-1},(A_h)_{ij}B_h^{-1}B_h^{-\top}]
				\partial_j\partial_i\bv}_{L^2(\Omega_1)}
			\leq CN^{-1}\norm{\bv}_{H^2(\Omega_1)}.
		\end{align*}
		Among the other terms, the three containing a first derivative of $\bv$ have a differentiated principal coefficient or an undifferentiated mixed coefficient. Their $H_\sharp^{1/2}(0,L)$ norms are bounded by $C(1+\norm{\partial_x^3h}_{L_\sharp^2(0,L)}^{1/2})$, uniformly in $z$. The two terms containing $\bv$ have a differentiated mixed coefficient or a zeroth-order coefficient, with norms bounded by $C(1+\norm{\partial_x^3h}_{L_\sharp^2(0,L)}+\norm{h}_{H_\sharp^{7/2}(0,L)})$. We apply \eqref{eq:lower-rough-commutator} to these five terms.
		
		At $z=1$, both $\bv$ and $\Lambda_N^{-1}\bv$ have vertical traces, while $\partial_iB_h^{-\top}$ has zero second column. Hence
		\begin{align*}
			[\Lambda_N^{-1},(A_h)_{i2}B_h^{-1}(\partial_iB_h^{-\top})]
			\bv\big|_{z=1}=0.
		\end{align*}
		The principal boundary terms satisfy, by \eqref{eq:lower-sharp-commutator} and the fixed-strip trace theorem,
		\begin{align*}
			\sum_{i=1}^2
			\norm{\left.\Lambda_N^{1/2}
				[\Lambda_N^{-1},(A_h)_{i2}B_h^{-1}B_h^{-\top}]
				\partial_i\bv\right|_{z=1}}_{L_\sharp^2(0,L)}
				\leq CN^{-1}\sum_{i=1}^2
			\norm{\partial_i\bv|_{z=1}}_{L_\sharp^2(0,L)}
			\leq CN^{-1}\norm{\bv}_{H^2(\Omega_1)}.
		\end{align*}
		Pairing the volume and boundary terms with $\Lambda_N^{-1/2}\partial_t\bv$ and $\Lambda_N^{-1/2}\partial_t^2h$, respectively, we obtain
		\begin{align*}
			2|K_3|
			&\leq \bigg[
			CN^{-1}\norm{\bv}_{H^2(\Omega_1)}
			+C_N\left(1+\norm{\partial_x^3h}_{L_\sharp^2(0,L)}^{1/2}\right)
			\norm{\nabla\bv}_{L^2(\Omega_1)}\\
			&\hspace{13mm}
			+C_N\left(1+\norm{\partial_x^3h}_{L_\sharp^2(0,L)}
			+\norm{h}_{H_\sharp^{7/2}(0,L)}\right)
			\norm{\bv}_{L^2(\Omega_1)}
			\bigg]\norm{\Lambda_N^{-1/2}\partial_t\bv}_{L^2(\Omega_1)}\\
			&\quad+CN^{-1}\norm{\bv}_{H^2(\Omega_1)}
			\norm{\Lambda_N^{-1/2}\partial_t^2h}_{L_\sharp^2(0,L)}.
		\end{align*}
		Using $\norm{\bv}_{L^2(\Omega_1)}\leq C$, $\norm{\nabla\bv}_{L^2(\Omega_1)}^2\leq C\cD$, and $1+\norm{\partial_x^3h}_{L_\sharp^2(0,L)}^2\leq C\cE_1$, it follows from Young's inequality that
		\begin{align*}
			2|K_3|
			&\leq\delta\left(
			\norm{\Lambda_N^{-1/2}\partial_t\bv}_{L^2(\Omega_1)}^2
			+\norm{\Lambda_N^{-1/2}\partial_t^2h}_{L_\sharp^2(0,L)}^2\right)
			+C_\delta N^{-2}\norm{\bv}_{H^2(\Omega_1)}^2\\
			&\quad+C_{\delta,N}\left((1+\cD)\cE_1
			+\norm{h}_{H_\sharp^{7/2}(0,L)}^2\right).
		\end{align*}
		
		Finally, since $\norm{\Lambda_N^{-1}\partial_t\bv}_{L^2(\Omega_1)}\leq N^{-1/2}\norm{\Lambda_N^{-1/2}\partial_t\bv}_{L^2(\Omega_1)}$, we apply Young's inequality and \eqref{eq:lower-nonlinear} with its parameter proportional to $N^{-1}$. For $K_4$ and $K_5$, we obtain, for every fixed $\delta>0$,
		\begin{align*}
			|K_4|+|K_5|\leq\delta\norm{\Lambda_N^{-1/2}\partial_t\bv}_{L^2(\Omega_1)}^2
			+C_\delta N^{-2}\norm{\bv}_{H^2(\Omega_1)}^2+C_{\delta,N}(1+\cD)\cE_1.
		\end{align*}
		
		We next estimate the two expressions inside the time derivative in \eqref{eq:lower-acceleration-identity}. The structural squares generated by integration by parts in time constitute $K_1$ and have already been estimated. For the elastic pairing, we have
		\begin{align*}
			\left|\int_0^L\left(\alpha\partial_x^2h\,\Lambda_N^{-1}\partial_t\partial_x^2h
			+\beta\partial_xh\,\Lambda_N^{-1}\partial_t\partial_xh\right)\,\mathrm dx\right|
			\leq C\norm{h}_{H_\sharp^3(0,L)}\norm{\partial_th}_{L_\sharp^2(0,L)}
			\leq \eta\cE_1+C_\eta.
		\end{align*}
		
		For the fluid viscous quadratic form at time $t$, we first consider the terms in which both spatial derivatives fall on the fluid factors. The full principal block is uniformly positive: for every $\boldsymbol\xi_1,\boldsymbol\xi_2\in\mathbb R^2$,
		\begin{align*}
			\sum_{i,j=1}^2
			\big((A_h)_{ij}B_h^{-1}B_h^{-\top}\boldsymbol\xi_i\big)
			\cdot\boldsymbol\xi_j
			&=\sum_{i,j=1}^2(A_h)_{ij}
			(B_h^{-\top}\boldsymbol\xi_i)\cdot
			(B_h^{-\top}\boldsymbol\xi_j)\\
			&\geq c\sum_{i=1}^2
			|B_h^{-\top}\boldsymbol\xi_i|^2
			\geq c_0\sum_{i=1}^2|\boldsymbol\xi_i|^2.
		\end{align*}
		Moreover, its entries have uniformly bounded horizontal $H_\sharp^1(0,L)$ norms, with constants depending only on the energy and gap. Applying \eqref{eq:lower-positive-mass} to this block and the vector $(\partial_x\bv,\partial_z\bv)$ at each fixed $z$, and then integrating in $z$, we obtain
		\begin{align*}
			&\sum_{i,j=1}^2\int_{\Omega_1}
			\big((A_h)_{ij}B_h^{-1}B_h^{-\top}\partial_i\bv\big)
			\cdot\Lambda_N^{-1}\partial_j\bv\dx\dz\\
			&\qquad\geq
			(c_0-CN^{-1/2})
			\norm{\Lambda_N^{-1/2}\nabla\bv}_{L^2(\Omega_1)}^2
			\geq \frac{c_0}{2}
			\norm{\Lambda_N^{-1/2}\nabla\bv}_{L^2(\Omega_1)}^2,
		\end{align*}
		provided $N$ is sufficiently large. This is the positive principal contribution used to absorb the mixed terms below. For either mixed term, move $\Lambda_N^{-1/2}$ to the differentiated fluid factor and use the product estimates underlying \eqref{eq:lower-rough-commutator}. The mixed coefficient has $H_\sharp^{1/2}(0,L)$ norm bounded by $C(1+\norm{\partial_x^3h}_{L_\sharp^2(0,L)}^{1/2})$, while $\norm{\Lambda_N^{-1}\bv}_{L^2(0,1;H_\sharp^1(0,L))}\leq C_N\norm{\bv}_{L^2(\Omega_1)}$ is energy controlled. 
		For the zeroth-order family in \eqref{eq:lower-coefficient-expansion}, set
		\begin{align*}
			\mathbf Z_{ij}=(A_h)_{ij}(\partial_jB_h^{-1})(\partial_iB_h^{-\top}).
		\end{align*}
		The energy bound and one-dimensional interpolation imply
		\begin{align*}
			\norm{\mathbf Z_{ij}}_{L_\sharp^2(0,L;L^\infty(0,1))}
			&\leq C\left(1+\norm{\partial_x^2h}_{L_\sharp^4(0,L)}^2\right)
			\leq C\left(1+\norm{\partial_x^3h}_{L_\sharp^2(0,L)}^{1/2}\right),\\
			\norm{\Lambda_N^{-1}\bv}_{L^\infty(0,L;L^2(0,1))}
			&\leq C_N\norm{\bv}_{L^2(\Omega_1)}.
		\end{align*}
		Applying Cauchy--Schwarz in $z$ and H\"older's inequality in $x$ , we obtain
		\begin{align*}
			\left|\int_{\Omega_1}(\mathbf Z_{ij}\bv)\cdot\Lambda_N^{-1}\bv\dx\dz\right|
			&\leq\norm{\mathbf Z_{ij}}_{L_\sharp^2(0,L;L^\infty(0,1))}
			\norm{\bv}_{L^2(\Omega_1)}
			\norm{\Lambda_N^{-1}\bv}_{L^\infty(0,L;L^2(0,1))}\\
			&\leq C_N\left(1+\norm{\partial_x^3h}_{L_\sharp^2(0,L)}^{1/2}\right).
		\end{align*}
		The product of the two fluid norms in the preceding estimate is bounded by $C_N\norm{\bv}_{L^2(\Omega_1)}^2$, which is controlled by the mechanical energy. Combining the principal, mixed, and zeroth-order terms, we conclude that
		\begin{align}
			&\mu\sum_{i,j=1}^2\int_{\Omega_1}(A_h)_{ij}
			\partial_i(B_h^{-\top}\bv)\cdot\partial_j(B_h^{-\top}\Lambda_N^{-1}\bv)\,\mathrm dx\,\mathrm dz\nonumber\\
			&\geq c\norm{\Lambda_N^{-1/2}\nabla\bv}_{L^2(\Omega_1)}^2
			-C_N\left(1+\norm{\partial_x^3h}_{L_\sharp^2(0,L)}\right)
			\geq -\eta\cE_1-C_{\eta,N}.
		\end{align}
		The absolute value of the expression inside the time derivative in \eqref{eq:lower-acceleration-identity}, evaluated at $t=0$, is bounded by $C_N(1+\cE_1(0))$.
		
		Now we integrate \eqref{eq:lower-acceleration-identity} over $(0,t)$. Choose the small parameters in the estimates of $K_3,K_4,K_5$ so that their acceleration terms are absorbed by the mass coercivity. In view of the estimates of $K_1,K_2$ and the preceding bounds at times $0$ and $t$, we conclude that
		\begin{align*}
			&\int_0^t\left(\norm{\Lambda_N^{-1/2}\partial_t\bv}_{L^2(\Omega_1)}^2
			+\norm{\Lambda_N^{-1/2}\partial_t^2h}_{L_\sharp^2(0,L)}^2\right)\,\mathrm ds\\
			&\quad\leq\eta\cE_1(t)+C_{\eta,N}(1+\cE_1(0))
			+CN^{-2}\int_0^t\norm{\bv}_{H^2(\Omega_1)}^2\,\mathrm ds
			\\&\qquad+C_N\int_0^t\left(\norm{\partial_x\bv}_{H^1(\Omega_1)}^2+(1+\cD)\cE_1+\norm{h}_{H_\sharp^{7/2}(0,L)}^2\right)\,\mathrm ds.
		\end{align*}
		Applying the preceding estimate with $\eta/2$ in place of $\eta$, and then using \eqref{eq:lower-height-bound} with $\varepsilon$ chosen so that $C_N\varepsilon\leq\eta/2$. we record \eqref{eq:lower-filtered-bound}. For $N\geq1$, the Fourier symbols satisfy
		\begin{align*}
			(1+k^2)^{-1/2}\leq N(N^2+k^2)^{-1/2},
			\qquad k\in(2\pi/L)\mathbb Z.
		\end{align*}
		Consequently,
		\begin{align}
			\norm{f}_{H_\sharp^{-1/2}(0,L)}^2
			\leq CN\norm{\Lambda_N^{-1/2}f}_{L_\sharp^2(0,L)}^2.
			\label{eq:lower-norm}
		\end{align}
		which finishes the proof. 
	\end{proof}
	
	By \eqref{eq:lower-norm}, the coefficient of the full fluid $H^2$ norm becomes $CN^{-1}$ after conversion to the unshifted acceleration norm. It can be made small by a further fixed choice of $N$, still depending only on the energy and gap. We shall combine this estimate with the preceding horizontal estimate and the fluid estimate proved in the next.
	
	\subsection{The fluid estimate}\label{sec:lower-fluid}
	
	For each time, let $\widetilde{\bv} = \mathscr{E} (\partial_t h)$ be the stationary solenoidal lifting from Lemma \ref{lem:lower-lift} with upper trace $\partial_th\be_2$, and put $\overline{\bv}=\bv-\widetilde{\bv}$, which has zero boundary trace. Both $\overline{\bv}$ and $\partial_t \overline{\bv}$ are solenoidal and have zero trace.
	
	The lifting operator is fixed and linear, so
	\begin{align}
		\norm{\partial_t\widetilde{\bv}}_{L^2(\Omega_1)}\leq C\norm{\partial_t^2h}_{H_\sharp^{-1/2}(0,L)}.
		\label{eq:lower-lift-time}
	\end{align}
	The zero mean of $\partial_th$ is its flux compatibility condition. By the mechanical energy, the $H^1$ norm equivalence, Lemma \ref{lem:lower-lift}, and \eqref{eq:lower-velocity-trace}, one records
	\begin{align}
		\norm{\overline{\bv}}_{L^2(\Omega_1)}+\norm{\widetilde{\bv}}_{L^2(\Omega_1)}&\leq C,\nonumber\\
		\norm{\overline{\bv}}_{H^1(\Omega_1)}^2+\norm{\widetilde{\bv}}_{H^1(\Omega_1)}^2&\leq C(1+\cD),\nonumber\\
		\norm{\partial_x\overline{\bv}}_{H^1(\Omega_1)}^2+\norm{\widetilde{\bv}}_{H^2(\Omega_1)}^2&\leq C\norm{\partial_x\bv}_{H^1(\Omega_1)}^2.
	\end{align}
	Applying \eqref{eq:lower-anisotropic} to these lifting bounds, we obtain
	\begin{subequations}
		\label{eq:lower-fluid-column-bounds}
		\begin{align}
			&\norm{\overline{\bv}}_{L_\sharp^\infty(0,L;L^2(0,1))}+\norm{\widetilde{\bv}}_{L_\sharp^\infty(0,L;L^2(0,1))}\leq C(1+\cD)^{1/4},\\
			&\norm{\nabla\overline{\bv}}_{L_\sharp^4(0,L;L^2(0,1))}^2
			\leq C(1+\cD)^{3/4}\left(1+\cD+\norm{\partial_x\bv}_{H^1(\Omega_1)}^2\right)^{1/4},\\
			&\norm{\nabla\overline{\bv}}_{L_\sharp^\infty(0,L;L^2(0,1))}
			\leq C(1+\cD)^{1/4}\left(1+\cD+\norm{\partial_x\bv}_{H^1(\Omega_1)}^2\right)^{1/4}.
		\end{align}
	\end{subequations}
	The energy-level bounds for $\partial_th$ and $\partial_x^2h$ also by interpolation yield
	\begin{subequations}
		\label{eq:lower-fluid-structural-bounds}
		\begin{align}
			&\norm{\partial_th}_{L_\sharp^\infty(0,L)}+\norm{\partial_x^2h}_{L_\sharp^\infty(0,L)}
			+\norm{\partial_x^2h}_{L_\sharp^4(0,L)}^2\leq C\cE_1^{1/4},\\
			&\norm{\partial_t\partial_xh}_{L_\sharp^2(0,L)}+\norm{\partial_x^3h}_{L_\sharp^2(0,L)}\leq C\cE_1^{1/2}.
		\end{align}
	\end{subequations}
	Here the energy-level $L^2$ bounds are retained in interpolation. The resulting powers are needed to keep the final bound linear in $\cE_1$.
	\begin{lemma}\label{lem:lower-fluid}
		Under the assumptions fixed above, there exist constants $c,C>0$, depending only on the mechanical energy, the prescribed positive gap, the conserved mean height, the period, and the physical parameters, such that, for every $0\leq t\leq T$,
		\begin{align}
			&\mu\sum_{i,j=1}^2\int_{\Omega_1}(A_h(t))_{ij}
			\partial_i(B_{h(t)}^{-\top}\overline{\bv}(t))\cdot\partial_j(B_{h(t)}^{-\top}\overline{\bv}(t))\dx\dz\nonumber\\
			&\quad+c\int_0^t\left(\norm{\partial_t\bv}_{L^2(\Omega_1)}^2
			+\norm{\bv}_{H^2(\Omega_1)}^2+\norm{\widehat p}_{H^1(\Omega_1)}^2\right)\,\mathrm ds\nonumber\\
			&\leq\mu\sum_{i,j=1}^2\int_{\Omega_1}(A_{h_0})_{ij}
			\partial_i(B_{h_0}^{-\top}\overline{\bv}(0))\cdot\partial_j(B_{h_0}^{-\top}\overline{\bv}(0))\dx\dz\nonumber\\
			&\quad+C\int_0^t\norm{\partial_t^2h}_{H_\sharp^{-1/2}(0,L)}^2\,\mathrm ds
			+C\int_0^t\norm{\partial_x\bv}_{H^1(\Omega_1)}^2\,\mathrm ds
			+C\int_0^t(1+\cD)\cE_1\,\mathrm ds.
			\label{eq:lower-fluid-loop}
		\end{align}
		Moreover, for every $\delta>0$ and every $0\leq t\leq T$,
		\begin{align}
			\norm{\bv(t)}_{H^1(\Omega_1)}^2
			&\leq C\mu\sum_{i,j=1}^2\int_{\Omega_1}(A_h(t))_{ij}
			\partial_i(B_{h(t)}^{-\top}\overline{\bv}(t))\cdot\partial_j(B_{h(t)}^{-\top}\overline{\bv}(t))\dx\dz\nonumber\\
			&\quad+\delta\norm{\partial_t\partial_xh(t)}_{L_\sharp^2(0,L)}^2+C_\delta.
			\label{eq:lower-fluid-pairing}
		\end{align}
		The initial viscous integral on the right-hand side of \eqref{eq:lower-fluid-loop} is bounded by $C\cE_1(0)$.
	\end{lemma}
	\begin{proof}
		With the reduction above, we now test the transformed fluid equation \eqref{eq:lower-strong-fluid} by $\partial_t\overline{\bv}$. Its pressure term vanishes because the test is solenoidal and has zero trace. Using $\partial_t\bv=\partial_t\overline{\bv}+\partial_t\widetilde{\bv}$, we obtain
		\begin{align}
			&\rho_f\int_{\Omega_1}h|B_h^{-\top}\partial_t\overline{\bv}|^2\dx\dz
			+\frac\mu2\frac{\mathrm d}{\mathrm dt}\sum_{i,j=1}^2\int_{\Omega_1}(A_h)_{ij}
			\partial_i(B_h^{-\top}\overline{\bv})\cdot\partial_j(B_h^{-\top}\overline{\bv})\dx\dz\nonumber\\
			&\quad=\frac\mu2\sum_{i,j=1}^2\int_{\Omega_1}(\partial_tA_h)_{ij}
			\partial_i(B_h^{-\top}\overline{\bv})\cdot\partial_j(B_h^{-\top}\overline{\bv})\dx\dz\nonumber\\
			&\qquad+\mu\sum_{i,j=1}^2\int_{\Omega_1}(A_h)_{ij}
			\partial_i((\partial_tB_h^{-\top})\overline{\bv})\cdot\partial_j(B_h^{-\top}\overline{\bv})\dx\dz\nonumber\\
			&\qquad-\rho_f\int_{\Omega_1}h(B_h^{-\top}\partial_t\widetilde{\bv})\cdot(B_h^{-\top}\partial_t\overline{\bv})\dx\dz
			+\mu\int_{\Omega_1}B_h^{-1}\Div\big(A_h\nabla(B_h^{-\top}\widetilde{\bv})\big)\cdot\partial_t\overline{\bv}\dx\dz\nonumber\\
			&\qquad-\rho_f\int_{\Omega_1}B_h^{-1}\left[h(\partial_tB_h^{-\top})\bv+
			\big((\widehat\bu-\partial_t\chi_h)\cdot(B_h\nabla)\big)\widehat\bu\right]\cdot\partial_t\overline{\bv}\dx\dz
			\eqqcolon \sum_{i=1}^5 L_i.
		\end{align}
		
		We first estimate the expression obtained by differentiating only the geometric coefficients of the viscous integral. The formula
		\begin{align*}
			\partial_tB_h^{-\top}
			=\begin{pmatrix}-\partial_th/h^2&0\\ z\partial_t\partial_xh/h-z\partial_xh\,\partial_th/h^2&0\end{pmatrix}
		\end{align*}
		shows that the only second spatial derivative of the beam velocity is a coefficient times $\partial_t\partial_x^2h\,\overline{\bv}\cdot\nabla(B_h^{-\top}\overline{\bv})$. Integrating this term once in $x$ leaves only $\partial_t\partial_xh$. After expanding the remaining factors, we obtain the following eight types of products, together with terms obtained by deleting $\partial_x^2h$ or one derivative of $\overline{\bv}$:
		\begin{align*}
			&|\partial_th|\,|\nabla\overline{\bv}|^2,\quad |\partial_t\partial_xh|\,|\nabla\overline{\bv}|^2,\quad
			|\partial_t\partial_xh|\,|\overline{\bv}|\,|\partial_x\nabla\overline{\bv}|,\\
			&|\partial_th|\,|\partial_x^2h|\,|\overline{\bv}|\,|\nabla\overline{\bv}|,\quad
			|\partial_t\partial_xh|\,|\partial_x^2h|\,|\overline{\bv}|\,|\nabla\overline{\bv}|,\\
			&|\partial_th|\,|\partial_x^2h|^2|\overline{\bv}|^2,\quad
			|\partial_t\partial_xh|\,|\partial_x^2h|^2|\overline{\bv}|^2,\quad
			|\partial_t\partial_xh|\,|\partial_x^3h|\,|\overline{\bv}|^2.
		\end{align*}
		All their coefficients are bounded in terms of the energy and the gap. At each $x$, the two fluid factors are paired in $L^2(0,1)$. Equations \eqref{eq:lower-fluid-column-bounds}--\eqref{eq:lower-fluid-structural-bounds} yield, for every $\delta>0$,
		\begin{align*}
			\int_{\Omega_1}|\partial_th|\,|\nabla\overline{\bv}|^2\dx\dz
			&\leq C\cE_1^{1/4}(1+\cD)\leq C(1+\cD)\cE_1,\\
			\int_{\Omega_1}|\partial_t\partial_xh|\,|\nabla\overline{\bv}|^2\dx\dz
			&\leq C\cE_1^{1/2}(1+\cD)^{3/4}
			\left(1+\cD+\norm{\partial_x\bv}_{H^1(\Omega_1)}^2\right)^{1/4}\\
			&\leq\delta\norm{\partial_x\bv}_{H^1(\Omega_1)}^2+C_\delta(1+\cD)\cE_1,\\
			\int_{\Omega_1}|\partial_t\partial_xh|\,|\overline{\bv}|\,|\partial_x\nabla\overline{\bv}|\dx\dz
			&\leq C\cE_1^{1/2}(1+\cD)^{1/4}\norm{\partial_x\bv}_{H^1(\Omega_1)}\\
			&\leq\delta\norm{\partial_x\bv}_{H^1(\Omega_1)}^2+C_\delta(1+\cD)\cE_1.
		\end{align*}
		For the products with additional derivatives of the geometric coefficients, we obtain
		\begin{align*}
			\int_{\Omega_1}|\partial_th|\,|\partial_x^2h|\,|\overline{\bv}|\,|\nabla\overline{\bv}|\dx\dz
			&\leq C\cE_1^{1/4}(1+\cD)^{3/4}\leq C(1+\cD)\cE_1,\\
			\int_{\Omega_1}|\partial_t\partial_xh|\,|\partial_x^2h|\,|\overline{\bv}|\,|\nabla\overline{\bv}|\dx\dz
			&\leq C\cE_1^{1/2}(1+\cD)^{1/2}\left(1+\cD+\norm{\partial_x\bv}_{H^1(\Omega_1)}^2\right)^{1/4}\\
			&\leq\delta\norm{\partial_x\bv}_{H^1(\Omega_1)}^2+C_\delta(1+\cD)\cE_1,\\
			\int_{\Omega_1}|\partial_th|\,|\partial_x^2h|^2|\overline{\bv}|^2\dx\dz
			&\leq C\cE_1^{1/4}(1+\cD)^{1/2}\leq C(1+\cD)\cE_1,\\
			\int_{\Omega_1}|\partial_t\partial_xh|\,|\partial_x^2h|^2|\overline{\bv}|^2\dx\dz
			&\leq C\cE_1^{3/4}(1+\cD)^{1/2}\leq C(1+\cD)\cE_1,\\
			\int_{\Omega_1}|\partial_t\partial_xh|\,|\partial_x^3h|\,|\overline{\bv}|^2\dx\dz
			&\leq C\cE_1(1+\cD)^{1/2}\leq C(1+\cD)\cE_1.
		\end{align*}
		The last estimate pairs the two structural factors in $L_\sharp^2(0,L)$ and uses the column supremum for each fluid factor. Thus
		\begin{align*}
			|L_1| + |L_2|
			\leq\delta\norm{\partial_x\bv}_{H^1(\Omega_1)}^2+C_\delta(1+\cD)\cE_1.
		\end{align*}
		In view of \eqref{eq:lower-lift-time}, one has
		\begin{align*}
			|L_3| \leq C \norm{\partial_t^2 h}_{H_\sharp^{-1/2}(0,L)} \norm{\partial_t \overline{\bv}}_{L^2(\Omega_1)}
			\leq \delta \norm{\partial_t \overline{\bv}}_{L^2(\Omega_1)}^2 + C \norm{\partial_t^2 h}_{H_\sharp^{-1/2}(0,L)}^2.
		\end{align*}
		By direct spatial differentiation, we have
		\begin{align*}
			\left|B_h^{-1}\Div\big(A_h\nabla(B_h^{-\top}\widetilde{\bv})\big)\right|
			\leq C\big[|\nabla^2\widetilde{\bv}|+(1+|\partial_x^2h|)|\nabla\widetilde{\bv}|
			+(1+|\partial_x^3h|+|\partial_x^2h|^2)|\widetilde{\bv}|\big].
		\end{align*}
		The terms involving the highest geometric derivatives satisfy
		\begin{align*}
			\norm{\partial_x^2h\,\nabla\widetilde{\bv}}_{L^2(\Omega_1)}^2
			&\leq C\norm{\nabla\widetilde{\bv}}_{L_\sharp^\infty(0,L;L^2(0,1))}^2
			\leq C\norm{\widetilde{\bv}}_{H^2(\Omega_1)}^2
			\leq C\norm{\partial_x\bv}_{H^1(\Omega_1)}^2,\\
			\norm{\partial_x^3h\,\widetilde{\bv}}_{L^2(\Omega_1)}^2
			&\leq C\cE_1(1+\cD)^{1/2}\leq C(1+\cD)\cE_1,\\
			\norm{|\partial_x^2h|^2\widetilde{\bv}}_{L^2(\Omega_1)}^2
			&\leq C\cE_1^{1/2}(1+\cD)^{1/2}\leq C(1+\cD)\cE_1.
		\end{align*}
		It follows that
		\begin{align*}
			\norm{B_h^{-1}\Div\big(A_h\nabla(B_h^{-\top}\widetilde{\bv})\big)}_{L^2(\Omega_1)}^2
			\leq C\norm{\partial_x\bv}_{H^1(\Omega_1)}^2+C(1+\cD)\cE_1,
		\end{align*}
		which implies
		\begin{align*}
			|L_4| 
			& \leq C \norm{B_h^{-1}\Div\big(A_h\nabla(B_h^{-\top}\widetilde{\bv})\big)}_{L^2(\Omega_1)} \norm{\partial_t \overline{\bv}}_{L^2(\Omega_1)} \\
			& \leq \delta\norm{\partial_t \overline{\bv}}_{L^2(\Omega_1)}^2
			+ C \norm{\partial_x\bv}_{H^1(\Omega_1)}^2+C(1+\cD)\cE_1.
		\end{align*}
		Moreover, it holds from \eqref{eq:lower-nonlinear} that
		\begin{align*}
			|L_5| \leq \delta(\norm{\bv}_{H^2(\Omega_1)}^2 + \norm{\partial_t \overline{\bv}}_{L^2(\Omega_1)}^2)+C(1+\cD)\cE_1.
		\end{align*}
		By the coercivity
		\begin{align*}
			\rho_f\int_{\Omega_1}h|B_h^{-\top}\partial_t\overline{\bv}|^2\dx\dz \geq C \norm{\partial_t \overline{\bv}}_{L^2(\Omega_1)}^2,
		\end{align*}
		which then absorbs the fluid acceleration terms on the right-hand side by taking small $\delta>0$.
		Then we conclude
		\begin{align}
			&\mu\frac{\mathrm d}{\mathrm dt}\sum_{i,j=1}^2\int_{\Omega_1}(A_h)_{ij}
			\partial_i(B_h^{-\top}\overline{\bv})\cdot\partial_j(B_h^{-\top}\overline{\bv})\dx\dz
			+c\norm{\partial_t\bv}_{L^2(\Omega_1)}^2\nonumber\\
			&\qquad\leq C\norm{\partial_t^2h}_{H_\sharp^{-1/2}(0,L)}^2
			+C\norm{\partial_x\bv}_{H^1(\Omega_1)}^2
			+\delta\norm{\bv}_{H^2(\Omega_1)}^2+C_\delta(1+\cD)\cE_1,
			\label{eq:lower-fluid-preliminary}
		\end{align}
		
		To recover the spatial fluid derivatives, we apply \cite[Theorem~4.5]{SchwarzacherSu2025} with $s=q=2$, boundary integrability index $2$, and boundary Besov order $2>3/2$. The height is bounded by the energy in $H_\sharp^2(0,L)$ and hence in $H^2(0,L)=B_{2,2}^2(0,L)$, where $B_{2,2}^2(0,L)$ denotes the usual Besov space. Its uniform $C^{1,1/2}$ modulus permits boundary charts with uniformly small Lipschitz constants after rotation, while the positive gap keeps the two boundaries separated. The stationary Stokes constant therefore depends only on the energy and the gap. Denoting by $p_0$ the pressure of zero spatial mean, we obtain
		\begin{align*}
			\norm{\bu}_{H^2(\Omega_h)}^2+\norm{p_0}_{H^1(\Omega_h)}^2
			\leq C\left(\norm{\partial_t\bu}_{L^2(\Omega_h)}^2
			+\norm{\bu\cdot\nabla\bu}_{L^2(\Omega_h)}^2
			+\norm{\partial_th}_{H_\sharp^{3/2}(0,L)}^2\right).
		\end{align*}
		The boundary Sobolev norm is measured through the same uniformly controlled graph charts. By the convection estimate in the proof of Lemma~\ref{lem:lower-coefficients} and Young's inequality, one has
		\begin{align*}
			\norm{\bu\cdot\nabla\bu}_{L^2(\Omega_h)}^2
			\leq\delta\norm\bu_{H^2(\Omega_h)}^2+C_\delta(1+\cD)\cE_1.
		\end{align*}
		We then use \eqref{eq:lower-norm-time}, \eqref{eq:lower-norm-spatial}, and \eqref{eq:lower-velocity-trace} to control the other two terms left with explicit remainder $(1+\cD)\cE_1$.
		The physical pressure constant is determined by the mean beam equation. Write $p=p_0+p_1(t)$ with $p_0$ of zero spatial mean. Since $\partial_th$ has zero mean,
		\begin{align}
			Lp_1(t)&=\int_0^L\be_2\cdot\bsigma(\bu,p_0)(-\partial_xh,1)^\top\big|_{y=h}\dx,\nonumber\\
			|p_1(t)|&\leq C\left(\norm\bu_{H^2(\Omega_h)}+\norm{p_0}_{H^1(\Omega_h)}\right).
			\label{eq:lower-pressure}
		\end{align}
		The trace estimate uses only the energy-controlled geometry and the positive gap. We then conclude that
		\begin{align}
			\norm{\bv}_{H^2(\Omega_1)}^2+\norm{\widehat p}_{H^1(\Omega_1)}^2
			\leq C\left(\norm{\partial_t\bv}_{L^2(\Omega_1)}^2
			+\norm{\partial_x\bv}_{H^1(\Omega_1)}^2\right)+C(1+\cD)\cE_1.
			\label{eq:lower-elliptic-fluid}
		\end{align}
		Now let us choose $\delta$ in \eqref{eq:lower-fluid-preliminary} sufficiently small to absorb its acceleration contribution through \eqref{eq:lower-elliptic-fluid} and apply \eqref{eq:lower-elliptic-fluid} once more, and integrate in time. This proves \eqref{eq:lower-fluid-loop}.
		
		Finally, \eqref{eq:lower-uniform-coercivity} shows that the viscous integral on the left-hand side of \eqref{eq:lower-fluid-loop}, evaluated at time $t$, controls $\norm{\overline{\bv}(t)}_{H^1(\Omega_1)}^2$. Since the mechanical energy bounds $\partial_th$ in $L_\sharp^2(0,L)$, interpolation and Lemma~\ref{lem:lower-lift} therefore imply \eqref{eq:lower-fluid-pairing} for every $\delta>0$.
	\end{proof}
	
	\subsection{Closure on an arbitrary finite interval}\label{sec:lower-closure}
	Now we are in the position to close all the estimates above.
	\begin{proof}[Proof of Proposition~\ref{prop:lower-estimates}]
		We retain $\widetilde{\bv}=\mathscr E(\partial_th)$ and $\overline{\bv}=\bv-\widetilde{\bv}$ from Section~\ref{sec:lower-fluid}. Combining \eqref{eq:lower-filtered-bound} with \eqref{eq:lower-norm}, the contribution of the full fluid $H^2$ norm to \eqref{eq:lower-fluid-loop} has coefficient $C_0N^{-1}$. Let $c_f>0$ be the coercivity constant multiplying the fluid integral in \eqref{eq:lower-fluid-loop}. We first fix $N\geq N_0$ so that $C_0N^{-1}\leq c_f/2$, also keeping the positivity ensured by \eqref{eq:lower-positive-mass}. This choice depends only on the energy, the prescribed gap, and the fixed parameters. After absorbing this contribution, we obtain, for every $\eta>0$,
		\begin{align}
			&\mu\sum_{i,j=1}^2\int_{\Omega_1}(A_h(t))_{ij}
			\partial_i(B_{h(t)}^{-\top}\overline{\bv}(t))\cdot\partial_j(B_{h(t)}^{-\top}\overline{\bv}(t))\dx\dz\nonumber\\
			&\quad+c\int_0^t\left(\norm{\partial_t\bv}_{L^2(\Omega_1)}^2
			+\norm{\bv}_{H^2(\Omega_1)}^2+\norm{\widehat p}_{H^1(\Omega_1)}^2\right)\,\mathrm ds\nonumber\\
			&\qquad\leq\eta\cE_1(t)+C_{\eta,N}\cE_1(0)
			+C_N\int_0^t\norm{\partial_x\bv}_{H^1(\Omega_1)}^2\,\mathrm ds
			+C_N\int_0^t(1+\cD)\cE_1\,\mathrm ds.
			\label{eq:lower-fluid-closed}
		\end{align}	
		By \eqref{eq:lower-fluid-pairing} and the mechanical energy bounds, one has
		\begin{align*}
			\cE_1
			\leq C \bigg(
			1& +\rho_f\int_{\Omega_1}h(t)|B_{h(t)}^{-\top}\partial_x\bv(t)|^2\,\mathrm dx\,\mathrm dz \\
			&  +\beta\norm{\partial_x^2h(t)}_{L_\sharp^2(0,L)}^2
			+\alpha\norm{\partial_x^3h(t)}_{L_\sharp^2(0,L)}^2
			+\rho_s\norm{\partial_t\partial_xh(t)}_{L_\sharp^2(0,L)}^2  \\
			& + \mu\sum_{i,j=1}^2\int_{\Omega_1}(A_h(t))_{ij}
			\partial_i(B_{h(t)}^{-\top}\overline{\bv}(t))\cdot\partial_j(B_{h(t)}^{-\top}\overline{\bv}(t))\dx\dz\bigg).
		\end{align*}
		Let $c_x>0$ and $c_v>0$ denote the coefficients of the fluid integrals on the left-hand sides of \eqref{eq:lower-horizontal-bound} and \eqref{eq:lower-fluid-closed}, respectively, and let $K_N$ be the coefficient of $\int_0^t\norm{\partial_x\bv}_{H^1(\Omega_1)}^2\,\mathrm ds$ on the right-hand side of \eqref{eq:lower-fluid-closed}. With $N$ now fixed, choose
		\begin{align*}
			0<\vartheta\leq\min\left\{1,\frac{c_x}{4(1+K_N)}\right\},
			\qquad 0<\delta\leq\frac{\vartheta c_v}{4}.
		\end{align*}
		We add \eqref{eq:lower-horizontal-bound}, taken with this $\delta$, to $\vartheta$ times \eqref{eq:lower-fluid-closed}. The first choice absorbs the tangential dissipation on the right-hand side, and the second absorbs the full fluid terms contributed by \eqref{eq:lower-horizontal-bound}. By the preceding norm comparison, the terms evaluated at time $t$ on the left, together with the constant $1$ already present in \eqref{eq:lower-horizontal-bound}, control $c_\vartheta\cE_1(t)$ for some $c_\vartheta>0$. We finally choose $\eta>0$ so that $\vartheta\eta\leq c_\vartheta/4$, which them leads to
		\begin{align}
			&\cE_1(t)+c\int_0^t\left(
			\norm{\partial_x\bv}_{H^1(\Omega_1)}^2
			+\norm{\partial_t\bv}_{L^2(\Omega_1)}^2
			+\norm{\bv}_{H^2(\Omega_1)}^2+\norm{\widehat p}_{H^1(\Omega_1)}^2\right)\,\mathrm ds\nonumber\\
			&\qquad\leq C\cE_1(0)+C\int_0^t(1+\cD(s))\cE_1(s)\,\mathrm ds,
		\end{align}
		after division by the fixed positive coercivity constants.
		 
		By the mechanical energy identity, $\int_0^T(1+\cD(s))\,\mathrm ds\leq T+\cE(0)/\mu$. Applying Gronwall's inequality, we conclude that
		\begin{align}
			&\sup_{0\leq t\leq T}\cE_1(t)
			+\int_0^T\left(
			\norm{\partial_x\bv}_{H^1(\Omega_1)}^2
			+\norm{\partial_t\bv}_{L^2(\Omega_1)}^2
			+\norm{\bv}_{H^2(\Omega_1)}^2+\norm{\widehat p}_{H^1(\Omega_1)}^2\right)\dt\nonumber\\
			&\qquad\leq C\cE_1(0)\exp\left(C\left[T+\frac{\cE(0)}\mu\right]\right).
		\end{align}
		Equations \eqref{eq:lower-velocity-trace}, \eqref{eq:lower-height-bound}, \eqref{eq:lower-filtered-bound}, and \eqref{eq:lower-norm} control the remaining structural norms. The spatial and temporal Piola transform then recover the corresponding physical fluid norms.
	\end{proof}
	
	\begin{remark}
		The estimates are uniform for any family of exact higher strong solutions with common bounds for the lower initial norms, the mechanical energy, and the inverse height on the chosen finite interval. 
	\end{remark}
	
	\section{Construction and Global Continuation}\label{sec:global-proof}
	
	In this section, we prove the global existence and uniqueness assertions of Theorem~\ref{thm:global} by approximating the initial data. Each exact higher-regularity solution is global by the distance estimate and the higher continuation theorem. Proposition~\ref{prop:lower-estimates} then supplies uniform lower-order bounds on every finite interval. These bounds suffice for the construction, the initial traces, the energy equality, and uniqueness.
	
	\subsection{Approximation of the problem}
	To justify the preceding arguments, we approximate the initial data by compatible data in $H^2(\Omega_{h_0})\times H_\sharp^4(0,L)\times H_\sharp^2(0,L)$, for which a unique higher strong solution exists up to contact by \cite[Theorems~1.1 and~4.7]{SchwarzacherSu2025}. Although the cited result assumes positive tension, the construction and estimates also apply when $\beta=0$, as follows directly from the bending coercivity. Indeed, for a zero-mean periodic function $f$,
	\begin{align}
		\norm{f}_{L_\sharp^2(0,L)}\leq (L/(2\pi))^2\norm{\partial_x^2f}_{L_\sharp^2(0,L)},
		\qquad
		\norm{\partial_xf}_{L_\sharp^2(0,L)}\leq L/(2\pi)\norm{\partial_x^2f}_{L_\sharp^2(0,L)}.
		\label{eq:construct-bending}
	\end{align}
	Thus $\alpha\norm{\partial_x^2f}_{L_\sharp^2(0,L)}^2+\beta\norm{\partial_xf}_{L_\sharp^2(0,L)}^2$ controls the full $H_\sharp^2(0,L)$ norm uniformly for $\beta$ in a bounded interval of $[0,\infty)$. This applies to $h-H$ and $\partial_th$. The coefficients in the fluid estimates and the mass matrices in the Galerkin construction do not involve $\beta$. Deleting the tension contribution therefore preserves their estimates and the coercivity of the structural stiffness. The initial-acceleration estimate uses $\alpha\partial_x^4h_0-\beta\partial_x^2h_0$, which has the required regularity also at $\beta=0$. More precisely, the differentiated structural energy in \cite[(3.7)]{SchwarzacherSu2025} contains
	\begin{align*}
		\frac{\rho_s}{2}\norm{\partial_t^2h}_{L_\sharp^2(0,L)}^2
		+\frac\beta2\norm{\partial_t\partial_xh}_{L_\sharp^2(0,L)}^2
		+\frac\alpha2\norm{\partial_t\partial_x^2h}_{L_\sharp^2(0,L)}^2.
	\end{align*}
	At zero tension, its last term still controls the full $H_\sharp^2(0,L)$ norm of $\partial_th$ by \eqref{eq:construct-bending}. The structural basis in \cite[Section~4.2]{SchwarzacherSu2025} uses the mean-zero bending space. The mass matrices in \cite[(4.11) and (4.14)]{SchwarzacherSu2025} remain positive, and \cite[Proposition~4.2]{SchwarzacherSu2025} bounds the initial acceleration without division by $\beta$. Consequently, the estimates and continuation argument in hold with constants uniform for $\beta$ in bounded subsets of $[0,\infty)$. The beam equation recovers $H_\sharp^4(0,L)$ using $\alpha>0$. These observations justify the same existence and continuation input throughout the parameter range of Theorem~\ref{thm:global}.
	
	Let the initial data satisfy \eqref{eq:data}. We choose smooth periodic $h_0^n\to h_0$ in $H_\sharp^3(0,L)$, preserving the mean height, and smooth periodic $h_1^n\to h_1$ in $H_\sharp^1(0,L)$, preserving zero mean. Omitting finitely many terms, we have
	\begin{align}
		h_0^n\geq\frac12h_0,\qquad
		\int_0^L\frac{\dx}{h_0^n}\leq2\int_0^L\frac{\dx}{h_0}.
		\label{eq:construct-initial-gap}
	\end{align}
	On $\Omega_1$, we have $\bv_0=B_{h_0}^\top\widehat\bu_0$. Let $\widetilde{\bv}_0$ and $\widetilde{\bv}_0^n$ be the stationary Stokes velocity fields with respective upper traces $h_1\be_2$ and $h_1^n\be_2$, zero lower trace, and periodicity in $x$, as in Lemma \ref{lem:lower-lift}. By the compatibility conditions,
	$\bv_0-\widetilde{\bv}_0$ is solenoidal and has zero trace at both horizontal boundaries. Then we approximate this difference in $H^1(\Omega_1)$ by smooth periodic solenoidal fields $\bv_{0,\mathrm{zero}}^n$ with the same zero boundary traces. Such approximations may be obtained by finite expansions in homogeneous Stokes eigenfunctions. Set
	\begin{align}
		\bv_0^n=\widetilde{\bv}_0^n+\bv_{0,\mathrm{zero}}^n,
		\qquad
		\widehat\bu_0^n=B_{h_0^n}^{-\top}\bv_0^n,
		\qquad
		\bu_0^n=\widehat\bu_0^n\circ\chi_{h_0^n}^{-1}.
	\end{align}
	The Piola identity ensures the exact divergence and trace conditions on each actual initial domain. In particular,
	\begin{align}
		&\bv_0^n\to\bv_0\quad\text{in }H^1(\Omega_1),
		\qquad \cE^n(0)\to \cE(0),\nonumber\\
		&\sup_n\left(\norm{\bv_0^n}_{H^1(\Omega_1)}
		+\norm{h_0^n}_{H_\sharp^3(0,L)}
		+\norm{h_1^n}_{H_\sharp^1(0,L)}\right)<\infty.
		\label{eq:construct-data-convergence}
	\end{align}
	
	\begin{remark}
		We use the higher strong theory of \cite{SchwarzacherSu2025} to construct exact solutions of the undamped system from compatible smooth approximations of the initial data. The separation and lower-order estimates therefore apply directly to these solutions. Alternatively, one could use the global damped solutions of \cite{GrandmontHillairet2016}, with the prescribed initial data at the lower regularity level, and pass to the limit as the damping coefficient $\gamma$ tends to zero. This would require extending Theorem~\ref{thm:separation} and Proposition~\ref{prop:lower-estimates} to the beam equation with the additional term $-\gamma\partial_x^2\partial_th$, with constants independent of $\gamma$. The additional contribution in the distance test can be controlled by the bending coercivity and the damped energy identity. In the lower-order estimates, the horizontal test introduces nonnegative dissipation, while the fractional height test and the acceleration test produce total time derivatives of nonnegative quadratic forms. After integration in time, these quadratic forms are nonnegative at time $t$, and their initial values are uniformly bounded. The same compactness argument then applies. 
	\end{remark}
	\subsection{Uniform separation and global lower-order bounds}
	
	For each $n$, take the exact maximal higher strong solution furnished by the preceding existence result. Theorem~\ref{thm:separation} applies on every compact subinterval of its existence interval. If the maximal time were finite, its quantitative bound would keep the height uniformly positive up to that time. The higher estimate and continuation theorem of \cite[Theorems~1.1 and~4.7]{SchwarzacherSu2025}, with the zero-tension modification justified above, would then extend that solution. Thus every approximate solution is global. This argument is applied separately for each $n$ and permits dependence on its finite higher initial norms.
	
	Fix any finite $T>0$. To make the distance constants uniform in $n$, we choose a common energy bound larger than $\cE(0)$, by which we have fix constants and the elliptic multiplier parameter in the distance argument. Together with the conserved mean and \eqref{eq:construct-initial-gap}, we find
	\begin{align}
		\inf_n\inf_{0\leq t\leq T}\min_{x\in[0,L]}h^n(t,x)\geq d_T>0.
		\label{eq:construct-uniform-gap}
	\end{align}
	The constant $d_T$ depends only on $T$, the common energy bound, the initial reciprocal height, the mean, the period, and the physical parameters.
	
	In view of Proposition~\ref{prop:lower-estimates} on the entire interval $[0,T]$, \eqref{eq:construct-data-convergence}, and \eqref{eq:construct-uniform-gap}, one obtains
	\begin{align}
		&\sup_n\sup_{0\leq t\leq T}\left(
		\norm{\bv^n(t)}_{H^1(\Omega_1)}^2
		+\norm{h^n(t)}_{H_\sharp^3(0,L)}^2
		+\norm{\partial_th^n(t)}_{H_\sharp^1(0,L)}^2\right)\nonumber\\
		&\quad+\sup_n\int_0^T\left(
		\norm{\partial_t\bv^n}_{L^2(\Omega_1)}^2
		+\norm{\bv^n}_{H^2(\Omega_1)}^2
		+\norm{p^n}_{H^1(\Omega_{h^n})}^2\right)\dt\nonumber\\
		&\quad+\sup_n\int_0^T\left(
		\norm{h^n}_{H_\sharp^{7/2}(0,L)}^2
		+\norm{\partial_th^n}_{H_\sharp^{3/2}(0,L)}^2
		+\norm{\partial_t^2h^n}_{H_\sharp^{-1/2}(0,L)}^2\right)\dt\leq C_T.
		\label{eq:construct-uniform-lower}
	\end{align}
	The corresponding physical bounds in \eqref{eq:lower-estimate}, including the Eulerian derivative bound, follow from \eqref{eq:lower-eulerian-} and the norm comparisons in the lower-order estimates. The constant $C_T$ depends only on the quantities permitted in Theorem~\ref{thm:global}. 
	
	\subsection{Compactness, pressure reconstruction, and initial traces}
	
	Fix $T<\infty$. By the uniform bounds and compact Sobolev embeddings, we obtain, after extraction,
	\begin{align}
		\bv^n&\to\bv
		&&\text{in }L^2(0,T;H^1(\Omega_1))\cap C([0,T];L^2(\Omega_1)),
		\label{eq:construct-fluid-compactness}\\
		h^n&\to h
		&&\text{in }C([0,T];H_\sharp^{3-\delta}(0,L)),\nonumber\\
		\partial_th^n&\to\partial_th
		&&\text{in }C([0,T];H_\sharp^{1-\delta}(0,L))
		\cap L^2(0,T;H_\sharp^1(0,L)),\qquad\delta>0.
		\label{eq:construct-beam-compactness}
	\end{align}
	For \eqref{eq:construct-fluid-compactness}, the sequence is bounded in $H^1(0,T;L^2(\Omega_1))\cap L^2(0,T;H^2(\Omega_1))\cap L^\infty(0,T;H^1(\Omega_1))$. For \eqref{eq:construct-beam-compactness}, use the bounds on $\partial_th^n$ in $L^\infty(0,T;H_\sharp^1(0,L))$ and on $\partial_t^2h^n$ in $L^2(0,T;H_\sharp^{-1/2}(0,L))$, together with the spatial compactness. Taking $0<\delta<1/2$ leads to uniform convergence of the heights and of their first two spatial derivatives. Thus \eqref{eq:construct-uniform-gap} passes to the limit.
	
	The coefficients of \eqref{eq:lower-coupled}, including those containing $\partial_x^2h^n$, converge strongly. The convection passes to the limit using \eqref{eq:construct-fluid-compactness}. For the coefficient-time products, the strong $L^2(0,T;H_\sharp^1(0,L))$ convergence of $\partial_th^n$ and the bounded $L^\infty(0,T;L_\sharp^\infty(0,L;L^2(0,1)))$ norm of $\bv^n$ control the factor containing $\partial_t\partial_xh^n$. The remaining product converges by the strong $L^2(0,T;H^1(\Omega_1))$ convergence of $\bv^n$. Hence the limiting pair solves the coupled equations with the prescribed boundary conditions and initial data.
	
	Recover the pressure of zero spatial mean, denoted by $p_0$, from the fluid equation. The full pressure is $p=p_0+p_1(t)$, with $p_1$ fixed by the mean beam equation as in \eqref{eq:lower-pressure}. The pressure-constant estimate established there and the trace estimate on the separated graph imply
	\begin{align*}
		p\in L^2(0,T;H^1(\Omega_h)),\qquad
		-\be_2\cdot\bsigma(\bu,p)\mathbf N_h\big|_{y=h}
		\in L^2(0,T;H_\sharp^{1/2}(0,L)).
	\end{align*}
	The beam equation consequently holds in $L^2(0,T;H_\sharp^{-1/2}(0,L))$.
	
	By the trace theorem, $\bv\in H^1(0,T;L^2(\Omega_1))\cap L^2(0,T;H^2(\Omega_1))$ implies $\bv\in C([0,T];H^1(\Omega_1))$ with initial value $\bv_0$. To recover the strong beam traces, first use \eqref{eq:construct-beam-compactness} and the uniform bounds to obtain weak continuity in $H_\sharp^3(0,L)\times H_\sharp^1(0,L)$. Write the vertical load as
	\begin{align*}
		\Phi=-\be_2\cdot\bsigma(\bu,p)\mathbf N_h\big|_{y=h}
		\in L^2(0,T;H_\sharp^{1/2}(0,L)).
	\end{align*}
	Applying spatial Fourier truncation to the beam equation differentiated once in $x$, we obtain
	\begin{align}
		&\frac12\left(\rho_s\norm{\partial_t\partial_xh(t)}_{L_\sharp^2(0,L)}^2
		+\beta\norm{\partial_x^2h(t)}_{L_\sharp^2(0,L)}^2
		+\alpha\norm{\partial_x^3h(t)}_{L_\sharp^2(0,L)}^2\right)
		\nonumber\\
		&\quad=\frac12\left(\rho_s\norm{\partial_xh_1}_{L_\sharp^2(0,L)}^2
		+\beta\norm{\partial_x^2h_0}_{L_\sharp^2(0,L)}^2
		+\alpha\norm{\partial_x^3h_0}_{L_\sharp^2(0,L)}^2\right)
		+\int_0^t\langle\partial_x\Phi,\partial_t\partial_xh\rangle\,\mathrm ds.
	\end{align}
	The limiting duality is justified by $\partial_x\Phi\in L^2(0,T;H_\sharp^{-1/2}(0,L))$ and $\partial_t\partial_xh\in L^2(0,T;H_\sharp^{1/2}(0,L))$. Hence the energy on the left is absolutely continuous. Weak continuity and continuity of this Hilbert norm imply
	\begin{align}
		h\in C([0,T];H_\sharp^3(0,L)),\qquad
		\partial_th\in C([0,T];H_\sharp^1(0,L)).
	\end{align}
	The conserved mean and the zero mean of $\partial_th$ turn the displayed seminorms into full norms. The Piola transformation then implies the asserted continuity of the physical fluid velocity after pullback.
	
	The limiting fields satisfy the coupled equations and boundary conditions, and the preceding bounds and strong initial traces place them in the class of Definition~\ref{def:strong}. Lemma~\ref{lem:mechanical-energy} therefore establishes the mechanical energy equality with the prescribed initial energy. This equality follows from the limiting solution itself; lower semicontinuity alone would provide only an inequality.
	
	The bounds obtained from Proposition~\ref{prop:lower-estimates} on the exact higher approximations pass to the limit by weak lower semicontinuity, using the convergence of the geometric coefficients. The continuous representatives established above turn the resulting essential suprema into suprema on $[0,T]$. Consequently, \eqref{eq:lower-bound} holds for the constructed solution in Definition~\ref{def:strong}, with the same dependence on the lower initial norms and the common positive gap. The spatial and temporal norm conversions in Lemma~\ref{lem:lower-coefficients} then yield \eqref{eq:lower-estimate}. A diagonal extraction over $T=1,2,\ldots$ produces one global solution. The reciprocal-height functional converges uniformly since
	\begin{align*}
		\sup_{0\leq t\leq T}\left|\int_0^L\frac{\dx}{h^n(t,x)}
		-\int_0^L\frac{\dx}{h(t,x)}\right|
		\leq Ld_T^{-2}\norm{h^n-h}_{C([0,T]\times[0,L])}\to0.
	\end{align*}
	Applying Theorem~\ref{thm:separation} to the limiting lower strong solution recovers its distance bound with constants determined by its exact initial energy.
	
	\subsection{Uniqueness and completion of the proof}
	
	The two-dimensional case of \cite[Theorem~1.2]{SchwarzacherSroczinski2022} requires
	\begin{align}
		\bu\in L^2(0,T;W^{1,2}(\Omega_h)),\qquad
		\partial_t\bu\in L^2(0,T;\widetilde W^{-1,2}(\Omega_h)).
	\end{align}
	Here $\widetilde W^{-1,2}$ is the dual of the full $W^{1,2}$ velocity-test space satisfying the fixed-wall condition, which is not restricted to solenoidal test fields. Both hypotheses follow from Definition~\ref{def:strong}, since the Eulerian derivative belongs to $L^2(0,T;L^2(\Omega_h))$. The height is separated on every finite interval. The proof in \cite{SchwarzacherSroczinski2022} applies to the periodic geometry, with cancellation of the lateral boundary terms, as used in \cite{SchwarzacherSu2025}. Its pure-bending case also covers $\beta=0$. We obtain uniqueness in the lower strong class. In particular, any solution in Definition~\ref{def:strong} with the same initial data agrees on its existence interval with the solution constructed above. 
	
	The unweighted regularity on $[0,T]$ therefore suffices for uniqueness, including at the initial time. Together with the preceding construction, strong initial traces, energy equality, and quantitative distance bound, this completes the proof of Theorem~\ref{thm:global}. 
	
	\section{Partial Parabolic Smoothing}\label{sec:time-weighted-estimates}
	We first establish a time-weighted estimate for the exact higher strong solutions, with constants controlled by the lower-order bounds. We then pass to the global solution constructed in Section~\ref{sec:global-proof} and derive its spatial Lipschitz regularity. The argument is adapted from the time-weighted method for inhomogeneous Navier--Stokes equations with vacuum in \cite{DanchinMucha2019}. In the present coupled problem, interpolation of the beam acceleration leads to the square of the time weight.
	
	\subsection{Time-weighted estimates}
	\begin{proposition}[Time-weighted estimate]\label{prop:weighted-estimates}
		Let $0<T<\infty$ and let $(h,\bu,p)$ be an exact higher strong solution on $[0,T]$ with bounded mechanical energy and $h\geq \min h>0$. Then
		\begin{align}
			&\mathop{\mathrm{ess\,sup}}_{0<t\leq T}t^2\left(
			\norm{\partial_t\bu(t)}_{L^2(\Omega_h(t))}^2
			+\norm{\partial_t^2h(t)}_{L_\sharp^2(0,L)}^2
			+\norm{\partial_th(t)}_{H_\sharp^2(0,L)}^2\right.\nonumber\\
			&\hspace{41mm}\left.
			+\norm{\bu(t)}_{H^2(\Omega_h(t))}^2
			+\norm{p(t)}_{H^1(\Omega_h(t))}^2
			+\norm{h(t)}_{H_\sharp^4(0,L)}^2\right)\nonumber\\
			&\qquad
			+\int_0^Tt^2
			\norm{\nabla\partial_t\bu(t)}_{L^2(\Omega_h(t))}^2
			\dt
			\leq C_T.
			\label{eq:weighted-bound}
		\end{align}
		The constant $C_T>0$ depends only on $T$, $(\min h)^{-1}$, the initial mechanical energy, the mean height, the period, the physical parameters, and the three initial norms in \eqref{eq:data-regularity}. It is independent of the higher initial norms and the initial acceleration.
	\end{proposition}
	
	\begin{proof}
		We first propagate the acceleration from a regular time $s$. By the acceleration estimate in \cite[Section~3.2 \& Theorem~4.4]{SchwarzacherSu2025}, we have, initially for almost every pair of regular times $0\leq s\leq t\leq T$ and then for the corresponding energy representatives,
		\begin{align}
			&1+\norm{\partial_t\widehat\bu(t)}_{L^2(\Omega_1)}^2
			+\norm{\partial_t^2h(t)}_{L_\sharp^2(0,L)}^2
			+\norm{\partial_th(t)}_{H_\sharp^2(0,L)}^2
			+\int_s^t\norm{\nabla\partial_t\widehat\bu(r)}_{L^2(\Omega_1)}^2\,\mathrm dr
			\nonumber\\
			&\qquad\leq C_T\left(1+\norm{\partial_t\widehat\bu(s)}_{L^2(\Omega_1)}^2
			+\norm{\partial_t^2h(s)}_{L_\sharp^2(0,L)}^2
			+\norm{\partial_th(s)}_{H_\sharp^2(0,L)}^2\right).
			\label{eq:weighted-propagation}
		\end{align}
		Here we have used the upper and lower bounds for $h$, together with the zero mean of $\partial_th$ and $\alpha>0$, to write the differentiated physical energy in terms of its equivalent Sobolev norms. The constant $1$ absorbs the lower-order terms in the source estimate. The corrected velocity derivative in the source calculation is precisely
		\begin{align}
			B_h^{-\top}\partial_t\bv=\partial_t\widehat\bu+
			\begin{pmatrix}
				h^{-1}\partial_th\,\widehat u_1\\
				z h^{-1}\partial_th\,\partial_xh\,\widehat u_1
				-z\partial_t\partial_xh\,\widehat u_1
			\end{pmatrix},
			\label{eq:weighted-correction}
		\end{align}
		where all terms can be controlled independently of higher norms at the initial time $s$.
		
		Note that
		\begin{align*}
			\int_0^T2s\int_s^T\norm{\nabla\partial_t\widehat\bu(t)}_{L^2(\Omega_1)}^2\,\mathrm dt\,\mathrm ds
			=\int_0^Tt^2\norm{\nabla\partial_t\widehat\bu(t)}_{L^2(\Omega_1)}^2\,\mathrm dt.
		\end{align*}
		Consequently, on the fixed interval $t \in [0,T]$, integrating \eqref{eq:weighted-propagation} against $s\,\mathrm ds$ on $(0,t)$ leads to
		\begin{align}
			&\mathop{\mathrm{ess\,sup}}_{0<t\leq T}t^2\left(\norm{\partial_t\widehat\bu(t)}_{L^2(\Omega_1)}^2
			+\norm{\partial_t^2h(t)}_{L_\sharp^2(0,L)}^2
			+\norm{\partial_th(t)}_{H_\sharp^2(0,L)}^2\right)
			+\int_0^Tt^2\norm{\nabla\partial_t\widehat\bu}_{L^2(\Omega_1)}^2\,\mathrm dt\nonumber\\
			&\quad\leq C_T+C_T\int_0^Tt\norm{\partial_t\widehat\bu}_{L^2(\Omega_1)}^2\,\mathrm dt
			+C_T\int_0^Tt\norm{\partial_t^2h}_{L_\sharp^2(0,L)}^2\,\mathrm dt
			+C_T\int_0^Tt\norm{\partial_th}_{H_\sharp^2(0,L)}^2\,\mathrm dt.
			\label{eq:weighted-averaging}
		\end{align}
		The fluid acceleration term on the right is controlled by the unweighted lower estimate. It remains to estimate the two beam terms.
		
		To estimate the beam velocity, retain the stationary lifting $\widetilde{\bv}=\mathscr{E}(\partial_th)$ from Section~\ref{sec:lower-fluid}. Its boundary conditions and time derivative are given by Lemma~\ref{lem:lower-lift} and \eqref{eq:lower-lift-time}. Differentiating \eqref{eq:lower-strong-fluid} and beam equation \eqref{eq:beam} in time and testing with $(\bphi,\zeta)=(\widetilde{\bv},\partial_th)$, we obtain
		\begin{align}
			&\frac{\mathrm d}{\mathrm dt} \mathcal H
			+\alpha\int_0^L|\partial_t\partial_x^2h|^2\dx
			+\beta\int_0^L|\partial_t\partial_xh|^2\dx\nonumber\\
			&=\rho_s\int_0^L|\partial_t^2h|^2\dx
			+\rho_f\int_0^L\int_0^1h(B_h^{-\top}\partial_t\bv)\cdot B_h^{-\top}\partial_t\widetilde{\bv}\dz\,\mathrm dx\nonumber\\
			&\quad+\rho_f\int_0^L\int_0^1h((\partial_tB_h^{-\top})\bv)\cdot B_h^{-\top}\partial_t\widetilde{\bv}\dz\,\mathrm dx \nonumber\\
			&\quad+\rho_f\int_0^L\int_0^1
			\big((\widehat\bu-\partial_t\chi_h)\cdot(B_h\nabla)\big)\widehat\bu
			\cdot B_h^{-\top}\partial_t\widetilde{\bv}\dz\,\mathrm dx\nonumber\\
			&\quad-\mu\sum_{i,j=1}^2\int_0^L\int_0^1(A_h)_{ij}
			\partial_i(B_h^{-\top}\partial_t\bv)\cdot\partial_j(B_h^{-\top}\widetilde{\bv})\dz\,\mathrm dx\nonumber\\
			&\quad-\mu\sum_{i,j=1}^2\int_0^L\int_0^1(\partial_tA_h)_{ij}
			\partial_i(B_h^{-\top}\bv)\cdot\partial_j(B_h^{-\top}\widetilde{\bv})\dz\,\mathrm dx\nonumber\\
			&\quad-\mu\sum_{i,j=1}^2\int_0^L\int_0^1(A_h)_{ij}
			\partial_i((\partial_tB_h^{-\top})\bv)\cdot\partial_j(B_h^{-\top}\widetilde{\bv})\dz\,\mathrm dx\nonumber\\
			&\quad-\mu\sum_{i,j=1}^2\int_0^L\int_0^1(A_h)_{ij}
			\partial_i(B_h^{-\top}\bv)\cdot\partial_j((\partial_tB_h^{-\top})\widetilde{\bv})\dz\,\mathrm dx
			\eqqcolon\sum_{i=1}^8 O_i.
			\label{eq:weighted-virial}
		\end{align}
		Here we have
		\begin{align*}
			\mathcal H & \coloneqq
			\rho_f\int_0^L\int_0^1 h(B_h^{-\top}\partial_t\bv)\cdot B_h^{-\top}\widetilde{\bv}\dz\,\mathrm dx +\rho_f\int_0^L\int_0^1h((\partial_tB_h^{-\top})\bv)\cdot B_h^{-\top}\widetilde{\bv}\dz\,\mathrm dx \\
			& \qquad
			+\rho_f\int_0^L\int_0^1
			\big((\widehat\bu-\partial_t\chi_h)\cdot(B_h\nabla)\big)\widehat\bu
			\cdot B_h^{-\top}\widetilde{\bv}\dz\,\mathrm dx 
			+\rho_s\int_0^L\partial_t^2h\,\partial_th\dx.
		\end{align*}
		Multiplying \eqref{eq:weighted-virial} by $t$ and integrating by parts from $0$ to $T$, we obtain
		\begin{align}
			& \alpha \int_0^Tt\int_0^L|\partial_t\partial_x^2h|^2\dx\dt
			+\beta\int_0^Tt\int_0^L|\partial_t\partial_xh|^2\dx\dt \nonumber\\
			& \quad = - T\mathcal H(T) + \int_0^T \mathcal H(t) \dt + \sum_{i=1}^8 \int_0^TtO_i\dt 
			\label{eq:weighted-t}
		\end{align}
		We estimate the full cross functional using the original coupled equation. Applying \eqref{eq:lower-coupled} with $(\bphi,\zeta)=(\widetilde{\bv},\partial_th)$ and integrating by parts in the periodic variable, we obtain
		\begin{align*}
			\mathcal H
			&=-\mu\sum_{i,j=1}^2\int_{\Omega_1}(A_h)_{ij}
			\partial_i(B_h^{-\top}\bv)\cdot\partial_j(B_h^{-\top}\widetilde{\bv})\dx\dz\\
			&\quad+\alpha\int_0^L\partial_x^3h\,\partial_t\partial_xh\dx
			-\beta\int_0^L\partial_xh\,\partial_t\partial_xh\dx.
		\end{align*}
		By \eqref{eq:lower-bound}, the upper bound in \eqref{eq:lower-uniform-coercivity}, and Lemma~\ref{lem:lower-lift}, for regular times we have
		\begin{align}
			|\mathcal H(t)|
			\leq C_T\left(
			\norm{\bv(t)}_{H^1(\Omega_1)}
			\norm{\widetilde{\bv}(t)}_{H^1(\Omega_1)}
			+\norm{h(t)}_{H_\sharp^3(0,L)}
			\norm{\partial_th(t)}_{H_\sharp^1(0,L)}
			\right)\leq C_T.
			\label{eq:weighted-cross-bound}
		\end{align}
		Consequently,
		\begin{align*}
			|T\mathcal H(T)|+\int_0^T|\mathcal H(t)|\dt\leq C_T.
		\end{align*}
		Note that $t\mathcal H(t)\to0$ as $t\downarrow0$, uniformly along the higher approximations with common lower-order bounds. 
		
		We now estimate $O_i$, $i=1,\dots,8$ in order. For $O_1$, we apply the trace theorem to the boundary identity $\partial_t\bv|_{z=1}=\partial_t^2h\,\be_2$ and find
		\begin{align*}
			\norm{\partial_t^2h}_{H_\sharp^{1/2}(0,L)}
			\leq C\norm{\partial_t\bv}_{H^1(\Omega_1)}.
		\end{align*}
		Interpolation between $H_\sharp^{-1/2}(0,L)$ and $H_\sharp^{1/2}(0,L)$, followed by Young's inequality, yields, for every $\delta>0$,
		\begin{align}
			\int_0^TtO_1 \dt
			=\rho_s \int_0^T t\norm{\partial_t^2h}_{L_\sharp^2(0,L)}^2\dt
			\leq\delta \int_0^T t^2\norm{\partial_t\bv}_{H^1(\Omega_1)}^2\dt
			+C_\delta\int_0^T\norm{\partial_t^2h}_{H_\sharp^{-1/2}(0,L)}^2\dt.
			\label{eq:weighted-acceleration}
		\end{align}
		The last term is controlled in \eqref{eq:lower-bound}.
		
		For $O_2$, the bounded geometric coefficients and \eqref{eq:lower-lift-time} imply
		\begin{align*}
			|O_2|
			\leq C_T\norm{\partial_t\bv}_{L^2(\Omega_1)}
			\norm{\partial_t\widetilde{\bv}}_{L^2(\Omega_1)}\leq C_T\norm{\partial_t\bv}_{L^2(\Omega_1)}
			\norm{\partial_t^2h}_{H_\sharp^{-1/2}(0,L)}.
		\end{align*}
		Both factors belong to $L^2(0,T)$ by \eqref{eq:lower-bound}. Therefore by Cauchy--Schwarz in time and $t\leq T$, we have
		\begin{align*}
			\int_0^Tt|O_2|\dt\leq C_T.
		\end{align*}
		
		For $O_3$, the entries of $\partial_tB_h^{-\top}$ contain only bounded factors multiplying $\partial_th$ or $\partial_t\partial_xh$. By the lower-order bounds and the column estimate in Lemma~\ref{lem:lower-coefficients}, we have $\norm{h(\partial_tB_h^{-\top})\bv}_{L^2(\Omega_1)}\leq C_T$. Consequently, by \eqref{eq:lower-lift-time} and \eqref{eq:lower-bound} we infer
		\begin{align*}
			\int_0^Tt|O_3|\dt
			&\leq C_T\int_0^Tt\norm{h(\partial_tB_h^{-\top})\bv}_{L^2(\Omega_1)}
			\norm{\partial_t^2h}_{H_\sharp^{-1/2}(0,L)}\dt
			\leq C_T.
		\end{align*}
		
		For $O_4$, the same lifting estimate with \eqref{eq:lower-nonlinear} and \eqref{eq:lower-bound} yields
		\begin{align*}
			\int_0^Tt|O_4|\dt
			&\leq C_T
			\int_0^Tt\norm{\big[(\widehat\bu-\partial_t\chi_h)\cdot(B_h\nabla)\big]\widehat\bu}_{L^2(\Omega_1)}
			\norm{\partial_t^2h}_{H_\sharp^{-1/2}(0,L)}\dt\\
			&\leq C_T \norm{\big[(\widehat\bu-\partial_t\chi_h)\cdot(B_h\nabla)\big]\widehat\bu}_{L^2(0,T;L^2(\Omega_1))} \norm{\partial_t^2h}_{L^2(0,T;H_\sharp^{-1/2}(0,L))}\leq C_T,
		\end{align*}
		
		For $O_5$, the $H^1$ multiplication bounds for $B_h^{-\top}$ and the stationary lifting estimate in Lemma \ref{lem:lower-lift} imply
		\begin{align}
			\int_0^Tt|O_5|\dt
			&\leq C_T\int_0^Tt\norm{\partial_t\bv}_{H^1(\Omega_1)}
			\norm{\widetilde{\bv}}_{H^1(\Omega_1)}\dt\nonumber\\
			&\leq\delta \int_0^Tt^2\norm{\partial_t\bv}_{H^1(\Omega_1)}^2\dt
			+C_{\delta,T}\int_0^T\norm{\partial_th}_{H_\sharp^{1/2}(0,L)}^2\dt.
		\end{align}
		The last term is finite by \eqref{eq:lower-bound}.
		
		By Lemma \ref{lem:lower-lift} and one-dimensional Sobolev embedding with values in $L^2(0,1)$, one records
		\begin{subequations}
			
			\begin{align}
				\norm{\widetilde{\bv}}_{L_\sharp^\infty(0,L;L^2(0,1))}
				&\leq C\norm{\partial_th}_{H_\sharp^{1/2}(0,L)},\nonumber\\
				\norm{\nabla\widetilde{\bv}}_{L_\sharp^\infty(0,L;L^2(0,1))}
				+\norm{\nabla\partial_x\widetilde{\bv}}_{L^2(\Omega_1)}
				&\leq C\norm{\partial_th}_{H_\sharp^{3/2}(0,L)}.
			\end{align}
		\end{subequations}
		The strong solution bounds $h$ in $L^\infty(0,T;H_\sharp^3(0,L))$ and both $\bv$ and $\widetilde{\bv}$ in $L^\infty(0,T;H^1(\Omega_1))$. In particular, their column suprema and the spatial coefficient bounds used below are uniform on $[0,T]$. Since the entries of $\partial_tA_h$ contain only $\partial_th$ and $\partial_t\partial_xh$, for $O_6$, similar to the proof of Lemma \ref{lem:lower-fluid}, we have
		\begin{align*}
			\int_0^Tt|O_6|\dt
			&\leq C_T\int_0^Tt\left(
			\norm{\nabla\widehat\bu}_{L^2(\Omega_1)}
			+\norm{\nabla\widehat\bu}_{L_\sharp^\infty(0,L;L^2(0,1))}\right)
			\norm{B_h^{-\top}\widetilde{\bv}}_{H^1(\Omega_1)}\dt\\
			&\leq C_T\int_0^Tt\left(1+\norm{\bv}_{H^2(\Omega_1)}^2\right)\dt \leq C_T,
		\end{align*}
		where the second inequality uses the same column estimates and norm conversions as in Lemma~\ref{lem:lower-coefficients}.
		
		For $O_7$, set $a_j=\frac{z(A_h)_{1j}}h,a_1=z,a_2=-\frac{z^2\partial_xh}{h}$.
		The separation bound and \eqref{eq:lower-bound} imply
		\begin{align*}
			\sum_{j=1}^2\left(
			\norm{a_j}_{L^\infty(\Omega_1)}
			+\norm{\partial_xa_j}_{L^\infty(\Omega_1)}
			\right)\leq C_T.
		\end{align*}
		We shall use, for $\mathbf f\in H^2(\Omega_1)$,
		\begin{align*}
			\norm{\mathbf f}_{L_\sharp^\infty(0,L;L^2(0,1))}
			+\norm{B_h^{-\top}\mathbf f}_{H^1(\Omega_1)}
			&\leq C_T\norm{\mathbf f}_{H^1(\Omega_1)},\\
			\norm{\nabla(B_h^{-\top}\mathbf f)}
			_{L_\sharp^\infty(0,L;L^2(0,1))}
			+\norm{\partial_x\nabla(B_h^{-\top}\mathbf f)}
			_{L^2(\Omega_1)}
			&\leq C_T\norm{\mathbf f}_{H^2(\Omega_1)}.
		\end{align*}
		These bounds follow from the coefficient-product argument in Section~\ref{sec:lower-coefficients}, in particular the expansion underlying \eqref{eq:lower-matrix-product}, and the column estimate \eqref{eq:lower-anisotropic}, using the uniform $H_\sharp^3(0,L)$ bound for $h$ from \eqref{eq:lower-bound}.
		
		Explicit differentiation of $B_h^{-\top}$ leads to the pointwise bound
		\begin{align*}
			&|\partial_tB_h^{-\top}|
			+|\partial_z\partial_tB_h^{-\top}|
			+\left|
			\partial_x\partial_tB_h^{-\top}
			-\frac zh\partial_t\partial_x^2h
			(\be_2\otimes\be_1)
			\right|
			\leq C_T
			\left(|\partial_th|+|\partial_t\partial_xh|\right).
		\end{align*}
		Expanding $\partial_i((\partial_tB_h^{-\top})\bv)$ and applying Cauchy--Schwarz in $z$ and H\"older's inequality in $x$, we obtain
		\begin{align*}
			\int_0^Tt|O_7|\dt
			\leq
			\mu\sum_{j=1}^2
			\int_0^Tt\left|\int_0^L\int_0^1
			\partial_t\partial_x^2h\,a_jv_1
			\partial_j(B_h^{-\top}\widetilde{\bv})_2
			\dz\,\mathrm dx\right|\dt+C_T.
		\end{align*}
		For the first term, integration by parts in $x$ yields, for $j=1,2$,
		\begin{align*}
			&\left|\int_0^L\int_0^1
			\partial_t\partial_x^2h\,a_jv_1
			\partial_j(B_h^{-\top}\widetilde{\bv})_2
			\dz\,\mathrm dx\right|\\
			&\quad=
			\left|\int_0^L\int_0^1
			\partial_t\partial_xh\,
			\partial_x\left(
			a_jv_1\partial_j(B_h^{-\top}\widetilde{\bv})_2
			\right)\dz\,\mathrm dx\right|\\
			&\quad\leq
			\norm{\partial_t\partial_xh}_{L_\sharp^2(0,L)}
			\bigg(
			\norm{\partial_xa_j}_{L^\infty(\Omega_1)}
			\norm{v_1}_{L_\sharp^\infty(0,L;L^2(0,1))}
			\norm{\partial_j(B_h^{-\top}\widetilde{\bv})_2}_{L^2(\Omega_1)}\\
			&\qquad\quad+
			\norm{a_j}_{L^\infty(\Omega_1)}
			\norm{\partial_xv_1}_{L^2(\Omega_1)}
			\norm{\partial_j(B_h^{-\top}\widetilde{\bv})_2}
			_{L_\sharp^\infty(0,L;L^2(0,1))}\\
			&\qquad\quad+
			\norm{a_j}_{L^\infty(\Omega_1)}
			\norm{v_1}_{L_\sharp^\infty(0,L;L^2(0,1))}
			\norm{\partial_x\partial_j(B_h^{-\top}\widetilde{\bv})_2}
			_{L^2(\Omega_1)}
			\bigg)\\
			&\quad\leq
			C_T\norm{\partial_th}_{H_\sharp^1(0,L)}
			\norm{\bv}_{H^1(\Omega_1)}
			\norm{\widetilde{\bv}}_{H^2(\Omega_1)}.
		\end{align*}
		The remaining terms in the estimate of $O_7$ have the same bound. Consequently, by \eqref{eq:lower-bound} and Lemma \ref{lem:lower-lift}, one infers
		\begin{align*}
			\int_0^Tt|O_7|\dt
			\leq C_T.
		\end{align*}
		
		For $O_8$, in view of the preceding product estimates with $\bv$ and $\widetilde{\bv}$, we find
		\begin{align*}
			\int_0^Tt|O_8|\dt
			&\leq C_T
			\int_0^Tt\norm{\partial_th}_{H_\sharp^1(0,L)}
			\norm{\widetilde{\bv}}_{H^1(\Omega_1)}
			\norm{\bv}_{H^2(\Omega_1)} \dt
			\leq C_T\int_0^Tt\norm{\bv}_{H^2(\Omega_1)}\dt
			\leq C_T.
		\end{align*}
		
		Using \eqref{eq:weighted-t}, \eqref{eq:weighted-cross-bound}, the preceding estimates of $O_1,\ldots,O_8$, and the zero mean of $\partial_th$, we obtain, for every $\delta>0$,
		\begin{align*}
			\int_0^Tt\norm{\partial_th}_{H_\sharp^2(0,L)}^2\dt
			\leq\delta\int_0^Tt^2\norm{\partial_t\bv}_{H^1(\Omega_1)}^2\dt
			+C_{\delta,T}.
		\end{align*}
		Combining this estimate with \eqref{eq:weighted-acceleration} and the unweighted lower bound for $\partial_t\widehat\bu$, and rescaling $\delta$, yields
		\begin{align}
			\int_0^Tt\left(
			\norm{\partial_t\widehat\bu}_{L^2(\Omega_1)}^2
			+\norm{\partial_t^2h}_{L_\sharp^2(0,L)}^2
			+\norm{\partial_th}_{H_\sharp^2(0,L)}^2
			\right)\dt
			\leq\delta\int_0^Tt^2\norm{\partial_t\bv}_{H^1(\Omega_1)}^2\dt
			+C_{\delta,T}.
			\label{eq:weighted-integral-preliminary}
		\end{align}
		To absorb the last integral, differentiating $\bv=B_h^\top\widehat\bu$ and using the global bound \eqref{eq:lower-bound} lead to
		\begin{align}
			\norm{\partial_t\bv}_{H^1(\Omega_1)}^2
			\leq C_T\bigl(1+\norm{\nabla\partial_t\widehat\bu}_{L^2(\Omega_1)}^2
			+\norm{\partial_t\widehat\bu}_{L^2(\Omega_1)}^2
			+\norm{\partial_th}_{H_\sharp^2(0,L)}^2
			+\norm{\bv}_{H^2(\Omega_1)}^2\bigr).
			\label{eq:weighted-piola-derivative}
		\end{align}
		In particular, $\partial_t\partial_x^2h\,\widehat u_1$ is controlled by $\norm{\partial_th}_{H_\sharp^2(0,L)}$ and $\norm{\widehat u_1}_{L_\sharp^\infty(0,L;L^2(0,1))}$. Products of $\partial_t\partial_xh$ with $\nabla\widehat u_1$ are controlled by $\norm{\bv}_{H^2(\Omega_1)}$. These account for all derivatives of the Piola coefficients.
		
		Collecting \eqref{eq:weighted-averaging}, \eqref{eq:weighted-piola-derivative}, and \eqref{eq:weighted-integral-preliminary} together, we conclude that
		\begin{align}
			&\mathop{\mathrm{ess\,sup}}_{0<t\leq T}t^2\left(\norm{\partial_t\widehat\bu(t)}_{L^2(\Omega_1)}^2
			+\norm{\partial_t^2h(t)}_{L_\sharp^2(0,L)}^2
			+\norm{\partial_th(t)}_{H_\sharp^2(0,L)}^2\right)
			\nonumber\\
			&\quad+\int_0^{T}t\left(\norm{\partial_t\widehat\bu}_{L^2(\Omega_1)}^2
			+\norm{\partial_t^2h}_{L_\sharp^2(0,L)}^2
			+\norm{\partial_th}_{H_\sharp^2(0,L)}^2\right)\,\mathrm dt+\int_0^{T}t^2\norm{\nabla\partial_t\widehat\bu}_{L^2(\Omega_1)}^2\mathrm dt\leq C_T,
			\label{eq:weighted-closed}
		\end{align}
		where we have chosen $\delta > 0$ sufficiently small such that all the norms on the right-hand side can be absorbed suitably via the weighted higher energy or dissipation.
		
		We next recover the Eulerian and spatial estimates using the temporal conversion \eqref{eq:lower-eulerian-} and $\widehat\bu=B_h^{-\top}\bv$.
		After taking a spatial derivative, the additional products contain $\partial_th\,\nabla^2\widehat\bu$ and $\partial_t\partial_xh\,\nabla\widehat\bu$. The former uses $H_\sharp^1(0,L)\hookrightarrow L_\sharp^\infty(0,L)$, and the latter satisfies
		\begin{align*}
			\norm{\partial_t\partial_xh\,\nabla\widehat\bu}_{L^2(\Omega_1)}
			\leq\norm{\partial_t\partial_xh}_{L_\sharp^2(0,L)}
			\norm{\nabla\widehat\bu}_{L_\sharp^\infty(0,L;L^2(0,1))}
			\leq C_T\norm{\bv}_{H^2(\Omega_1)}.
		\end{align*}
		Thus \eqref{eq:weighted-closed} also controls $t\partial_t\bu$ in $L^\infty(0,T;L^2(\Omega_h))\cap L^2(0,T;H^1(\Omega_h))$.
		
		For the spatial estimates, use the stationary Stokes estimate and convection absorption from the proof of Lemma~\ref{lem:lower-fluid}, together with the pressure normalization \eqref{eq:lower-pressure}. Combining these with the beam equation \eqref{eq:beam}, the fluid-traction trace estimate, and \eqref{eq:lower-bound}, we obtain
		\begin{align}
			&\norm{\bu}_{H^2(\Omega_h)}^2+\norm{p}_{H^1(\Omega_h)}^2
			+\norm{h}_{H_\sharp^4(0,L)}^2 \nonumber\\
			& \qquad\leq C_T\left(1+\norm{\partial_t\widehat\bu}_{L^2(\Omega_1)}^2
			+\norm{\partial_t^2h}_{L_\sharp^2(0,L)}^2
			+\norm{\partial_th}_{H_\sharp^2(0,L)}^2\right).
			\label{eq:weighted-spatial}
		\end{align}
		Multiplying \eqref{eq:weighted-spatial} by $t^2$ and using \eqref{eq:weighted-closed} shows \eqref{eq:weighted-bound}, which completes the proof.
	\end{proof}
	
	\subsection{Proof of Theorem~\ref{thm:smoothing}}
	Fix $T>0$ and let us consider the exact global approximation sequence $(h^n,\bu^n,p^n)$ constructed in Section~\ref{sec:global-proof}. The common separation bound \eqref{eq:construct-uniform-gap} and the uniform lower-order estimate \eqref{eq:construct-uniform-lower} allow us to apply Proposition~\ref{prop:weighted-estimates} with a constant independent of $n$. In particular, \eqref{eq:weighted-estimate} holds uniformly along this sequence, independently of its higher initial norms and initial acceleration.
	
	Fix $0<\tau<T$. Since the time weights are bounded away from zero on $[\tau,T]$, weak and weak-star compactness after pullback to $\Omega_1$, together with \eqref{eq:construct-fluid-compactness}--\eqref{eq:construct-beam-compactness}, identifies the limits with the corresponding derivatives of the solution furnished by Theorem~\ref{thm:global}. For the Eulerian time derivative, use \eqref{eq:lower-eulerian-}, while the pressure is identified by the fluid equation and \eqref{eq:lower-pressure}. By the strong convergence of the geometric coefficients and weak lower semicontinuity, we obtain \eqref{eq:weighted-estimate} restricted to $[\tau,T]$, with a constant independent of $\tau$. Letting $\tau\downarrow0$, using monotone convergence for the integral and monotonicity of the essential suprema, we conclude that \eqref{eq:weighted-estimate} holds on $(0,T]$.
	
	We finally prove \eqref{eq:velocity-lipschitz}. Fix $1\leq q<2$ and choose
	\begin{align*}
		0<\theta<\frac1q-\frac12.
	\end{align*}
	The separation estimate and the uniform $H_\sharp^3(0,L)$ bound for $h$ ensure that the spatial interpolation, multiplication, and Stokes estimates below have constants uniform on $[0,T]$. Recall by \eqref{eq:lower-estimate} and \eqref{eq:weighted-estimate} that
	\begin{align*}
		&\norm{\partial_t\bu}_{L^2(0,T;L^2(\Omega_h))}
		+\norm{t\partial_t\bu}_{L^2(0,T;H^1(\Omega_h))}
		+\norm{\partial_th}_{L^\infty(0,T;H_\sharp^1(0,L))}
		+\norm{t\partial_th}_{L^\infty(0,T;H_\sharp^2(0,L))}
		\leq C_T.
	\end{align*}
	In view of spatial interpolation, it holds, for almost every $t>0$, that
	\begin{align*}
		\norm{\partial_t\bu(t)}_{H^\theta(\Omega_h(t))}
		&\leq C_T
		\norm{\partial_t\bu(t)}_{L^2(\Omega_h(t))}^{1-\theta}
		\norm{\partial_t\bu(t)}_{H^1(\Omega_h(t))}^{\theta}\\
		&=C_Tt^{-\theta}
		\norm{\partial_t\bu(t)}_{L^2(\Omega_h(t))}^{1-\theta}
		\norm{t\partial_t\bu(t)}_{H^1(\Omega_h(t))}^{\theta}.
	\end{align*}
	Since $2q\theta/(2-q)<1$, the function $t^{-\theta}$ belongs to $L^{2q/(2-q)}(0,T)$. H\"older's inequality therefore implies
	\begin{align*}
		\norm{\partial_t\bu}_{L^q(0,T;H^\theta(\Omega_h))}
		\leq C_{T,q,\theta}
		\norm{\partial_t\bu}_{L^2(0,T;L^2(\Omega_h))}^{1-\theta}
		\norm{t\partial_t\bu}_{L^2(0,T;H^1(\Omega_h))}^{\theta}
		\leq C_{T,q,\theta}.
	\end{align*}
	For the boundary velocity, we interpolate between $H_\sharp^1(0,L)$ and $H_\sharp^2(0,L)$ to obtain
	\begin{align*}
		\norm{\partial_th(t)}_{H_\sharp^{3/2+\theta}(0,L)}
		\leq C
		\norm{\partial_th(t)}_{H_\sharp^1(0,L)}^{1/2-\theta}
		\norm{\partial_th(t)}_{H_\sharp^2(0,L)}^{1/2+\theta}
		\leq C_Tt^{-1/2-\theta}.
	\end{align*}
	Our choice of $\theta$ ensures that $q(1/2+\theta)<1$, and hence
	\begin{align*}
		\norm{\partial_th}_{L^q(0,T;H_\sharp^{3/2+\theta}(0,L))}
		\leq C_{T,q,\theta}.
	\end{align*}
	
	By the two-dimensional multiplication estimate, valid for $0<\theta<1$, we have
	\begin{align*}
		\norm{\bu\cdot\nabla\bu}_{H^\theta(\Omega_h)}
		\leq C_{T,\theta}
		\norm{\bu}_{H^1(\Omega_h)}
		\norm{\bu}_{H^2(\Omega_h)}.
	\end{align*}
	Thus \eqref{eq:lower-estimate} implies
	\begin{align*}
		\norm{\bu\cdot\nabla\bu}_{L^q(0,T;H^\theta(\Omega_h))}
		\leq C_{T,q,\theta},
	\end{align*}
	because $q<2$ and $T<\infty$.
	
	We now apply the stationary Stokes estimate in \cite[Theorem~4.5]{SchwarzacherSu2025}, with spatial integrability exponent $2$, spatial order $2+\theta$, boundary integrability index $2$, and boundary Besov order $3>3/2+\theta$. The uniform bound $h\in L^\infty(0,T;H_\sharp^3(0,L))$ supplies this boundary regularity, while the boundary charts and positive gap are controlled as in the proof of Lemma~\ref{lem:lower-fluid}. The boundary compatibility condition follows from $\int_0^L\partial_th\dx=0$. Hence
	\begin{align*}
		\norm{\bu}_{H^{2+\theta}(\Omega_h)}
		\leq C_{T,\theta}\left(
		\norm{\partial_t\bu}_{H^\theta(\Omega_h)}
		+\norm{\bu\cdot\nabla\bu}_{H^\theta(\Omega_h)}
		+\norm{\partial_th}_{H_\sharp^{3/2+\theta}(0,L)}
		\right).
	\end{align*}
	Combining the preceding bounds and using the two-dimensional embedding $H^{2+\theta}(\Omega_h(t))\hookrightarrow W^{1,\infty}(\Omega_h(t))$, we obtain
	\begin{align*}
		\norm{\bu}_{L^q(0,T;W^{1,\infty}(\Omega_h))}
		\leq C_{T,\theta}
		\norm{\bu}_{L^q(0,T;H^{2+\theta}(\Omega_h))}
		\leq C_{T,q}.
	\end{align*}
	Here $\theta$ is fixed depending only on $q$. This proves \eqref{eq:velocity-lipschitz}. In particular, the spatial Lipschitz norm of the fluid velocity is integrable on the whole interval $(0,T)$. 
	
	This completes the proof. \qed
	
	\section*{Declarations}
	\subsection*{Data availability}
	Data sharing is not applicable to this article.
	\subsection*{Conflict of interest}
	The author declares no conflict of interest.
	\subsection*{AI disclosure} 
	AI-assisted tool (GPT-6 Astra Ultra) was used to improve the clarity and presentation of proofs written by the authors.
	All mathematical statements and proofs were rigorously verified by the author, who takes full responsibility for the contents of the paper.
	
	\subsection*{Funding}
	YL was partially supported by the Natural Science Foundation of Jiangsu Province (Grant No.~BK20240572) and the China Postdoctoral Science Foundation (Grant~2025M773078).

\end{document}